\documentclass[11pt]{article}
\usepackage{amsmath,amssymb,amsthm,float}
\usepackage{latexsym, fullpage}
\usepackage{longtable}
\usepackage{enumitem}
\usepackage{array}
\usepackage{calc}
\usepackage[bf,width=.9\textwidth,hang]{caption}
\usepackage[title]{appendix}
\usepackage[parfill]{parskip}
\begingroup
    \makeatletter
    \@for\theoremstyle:=definition,remark,plain\do{%
        \expandafter\g@addto@macro\csname th@\theoremstyle\endcsname{%
            \addtolength\thm@preskip\parskip
            }%
        }
\endgroup

\newtheorem{thm*}{Theorem}
\newtheorem{thm}{Theorem}[section]

\newtheorem{lem}[thm]{Lemma}
\newtheorem{cor}[thm]{Corollary}
\newtheorem{cor*}{Corollary}

\newenvironment{defn}[1][Definition]{\begin{trivlist}
\item[\hskip \labelsep {\bfseries #1}]}{\end{trivlist}}
\newenvironment{pf}[1]{\begin{trivlist}
\item[\hskip \labelsep {\bfseries #1}]}{\end{trivlist}}

\newcommand{\mfra}[1]{\genfrac{}{}{0pt}{0}{}{#1}}

\title{Simple Irreducible Subgroups of Exceptional Algebraic Groups}
\author{Adam R. Thomas}

\begin{document}

\maketitle

\begin{abstract}
A closed subgroup of a semisimple algebraic group is called irreducible if it lies in no proper parabolic subgroup. In this paper we classify all irreducible subgroups of exceptional algebraic groups $G$ which are connected, closed and simple of rank at least $2$. Consequences are given concerning the representations of such subgroups on various $G$-modules: for example, with one exception, the conjugacy classes of irreducible simple connected subgroups of rank at least $2$ are determined by their composition factors on the adjoint module of $G$.  
\end{abstract}
\vspace{-0.4cm}
\section{Introduction} \label{intro}
\vspace{-0.4cm}
Let $G$ be a reductive connected algebraic group. A subgroup $X$ of $G$ is called \emph{$G$-irreducible} (or just \emph{irreducible} if $G$ is clear from the context) if it is closed and not contained in any proper parabolic subgroup of $G$. This definition, due to Serre in \cite{ser04}, generalises the standard notion of an irreducible subgroup of $GL(V)$. Indeed, if $G = GL(V)$, a subgroup $X$ is $G$-irreducible if and only if $X$ acts irreducibly on $V$. Similarly, the notion of complete reducibility can be generalised (see \cite{ser04}): a subgroup $X$ of $G$ is said to be \emph{$G$-completely reducible} (or \emph{$G$-cr} for short) if, whenever it is contained in a parabolic subgroup of $G$, it is contained in a Levi subgroup of that parabolic.      

Now let $G$ be a connected semisimple group. In \cite{LT}, Liebeck and Testerman studied connected $G$-irreducible subgroups for the first time, showing amongst other things that they are semisimple and only have a finite number of overgroups in $G$. Connected $G$-irreducible subgroups play an important role in determining both the $G$-cr and non-$G$-cr connected subgroups of $G$. The $G$-cr subgroups of $G$ are simply the $L'$-irreducible subgroups of $L'$ for each Levi subgroup $L$ of $G$ (noting that $G$ is a Levi subgroup of itself). To determine the non-$G$-cr subgroups of $G$ one strategy is as follows. Let $P$ be a proper parabolic subgroup with unipotent radical $Q$ and Levi complement $L$. Then for each $L'$-irreducible subgroup $X$, determine the complements to $Q$  in $X Q$ that are not $Q$-conjugate to $X$ (if any exist). Any non-$G$-cr connected subgroup will be of this form for some $L'$-irreducible connected subgroup $X$.   

We now restrict our attention further by letting $G$ be a simple algebraic group of exceptional type over an algebraically closed field $K$ of characteristic $p$ (setting $p= \infty$ for characteristic 0). In this paper we classify the simple, connected $G$-irreducible subgroups of rank at least $2$. For $G = G_2$ this is a trivial consequence of \cite[Theorem~1]{LS5} and for $G = F_4$ this has already been done \linebreak \cite[Theorem~4]{dav}. We therefore only have to deal with $G= E_6,E_7$ and $E_8$, for which we prove the following three theorems. The tables referred to in the statements can be found in Section \ref{tabs} of the paper.  
\pagebreak
\begin{thm*} \label{Thm1}

Suppose $X$ is a simple, connected, irreducible subgroup of $E_6$ of rank at least $2$. Then $X$ is $\text{Aut}(E_6)$-conjugate to exactly one subgroup of Table \ref{E6tab}. 

\end{thm*}

\begin{thm*} \label{Thm2}

Suppose $X$ is a simple, connected, irreducible subgroup of $E_7$ of rank at least $2$. Then $X$ is $E_7$-conjugate to exactly one subgroup of Table \ref{E7tab}. 

\end{thm*}

\begin{thm*} \label{Thm3}

Suppose $X$ is a simple, connected, irreducible subgroup of $E_8$ of rank at least $2$. Then $X$ is $E_8$-conjugate to exactly one subgroup of Table \ref{E8tab}. 

\end{thm*}  

We note that Amende \cite{bon} covers the $G$-irreducible subgroups of rank $1$ in $G = G_2$, $F_4$, $E_6$ and $E_7$. The semisimple (non-simple) $G$-irreducible subgroups and the irreducible subgroups of $E_8$ of rank~1 will be covered in forthcoming work of the author. Also, under various assumptions on the characteristic ($p >7$ covers all of them), Theorems 1--3 can be deduced from the results in \cite{LS3}. Our contribution is to remove these characteristic restrictions. 

Each subgroup in Tables \ref{E6tab}--\ref{E8tab} is described by its embedding in some maximal connected subgroup, given in Theorem \ref{max}. Notation for the embeddings is given in Section \ref{nota}.    

From these results we can prove a number of representation-theoretic corollaries. For the first of these, we need the following definition. Let $G$ be a simple algebraic group (of arbitrary type), $V$ be a module for $G$ and $X$ and $Y$ be subgroups of $G$. Then we say $X$ and $Y$ share the same composition factors on $V$ if there exists a morphism from $X$ to $Y$, which is an isomorphism of abstract groups sending the composition factors of $X$ to composition factors of $Y$. 

The first of our corollaries shows that if $G$ is an exceptional algebraic group then, with one exception, conjugacy between $G$-irreducible subgroups is determined by their composition factors on $L(G)$.

\begin{cor*}
Let $G$ be a simple exceptional algebraic group and $X$ and $Y$ be simple, connected irreducible subgroups of $G$ of rank at least $2$. If $X$ and $Y$ have the same composition factors on $L(G)$ then either:

\begin{enumerate}[leftmargin=*,label=\normalfont(\arabic*)]
\item $X$ is conjugate to $Y$ in Aut$(G)$, or
\item $G = E_8$, $X \cong Y \cong A_2$, $p \neq 3$, $X \hookrightarrow A_2^2 < D_4^2$ via $(10,10^{[r]})$ and $Y \hookrightarrow A_2^2 < D_4^2$ via $(10,01^{[r]})$ (or vice versa) where $r \neq 0$ and $A_2^2$ is irreducibly embedded in $D_4^2$.  
\end{enumerate}

\end{cor*}  
 
Again, this is proved in \cite[Theorem 4]{LS3} with restrictions on the characteristic $p$. The notation ``$X \hookrightarrow A_2^2$ via $(10,10^{[r]})$'' is explained in Section \ref{nota}. 

The next corollary highlights the interesting subgroups that are $M$-irreducible but not $G$-irreducible for some reductive, maximal connected subgroup $M$. Here ``interesting'' means that the $M$-irreducible subgroup is not obviously $G$-reducible, i.e. $M'$-reducible for some other reductive, maximal connected subgroup $M'$ or contained in a proper Levi subgroup. 

\begin{cor*} \label{nongcr}
Let $G$ be an exceptional algebraic group and $X$ be a simple connected subgroup of rank at least $2$ of $G$. Suppose that whenever $X$ is contained in a reductive, maximal connected subgroup $M$ it is $M$-irreducible and assume that such an overgroup $M$ exists. Assume further that $X$ is not contained in a proper Levi subgroup of $G$. Then either:

\begin{enumerate}[leftmargin=*,label=\normalfont(\arabic*)]
\item $X$ is $G$-irreducible, or
\item $X$ is Aut($G$)-conjugate to a subgroup in Table \ref{cortab} below. Such $X$ are non-$G$-cr and satisfy the hypothesis.  
\end{enumerate}
   
\end{cor*}

\begin{longtable}{p{0.08\textwidth - 2\tabcolsep}p{0.15\textwidth - 2\tabcolsep}p{0.1\textwidth - 2\tabcolsep}>{\raggedright\arraybackslash}p{0.67\textwidth-\tabcolsep}@{}}

\caption{Non-$G$-cr subgroups that are irreducible in every (and at least one) maximal, reductive overgroup \label{cortab}} \\

\hline \noalign{\smallskip} 

$G$ & Max. $M$ & $p$ & $M$-irreducible subgroup $X$ \\

\hline \noalign{\smallskip} 

$F_4$ & $A_2 \tilde{A}_2$ & $p = 3$ & $A_2 \hookrightarrow A_2 \tilde{A}_2$ via $(10,01)$ \\

\hline \noalign{\smallskip}

$E_6$ & $A_2 G_2$ & $p =3$ & $A_2 \hookrightarrow A_2 \tilde{A}_2$ via $(10,10)$ \\

\hline \noalign{\smallskip}

$E_7$ & $A_2 A_5$ & $p =3$ & $A_2 \hookrightarrow A_2 A_2^{(*)} < A_2 A_5$ via $(10,10)$ (see Lemma \ref{a2a2starnotation}) \\

& $A_7$ & $p=2$ & $C_4$ \\

& & $p=2$ & $D_4$ \\

& $G_2 C_3$ & $p=2$ & $G_2 \hookrightarrow G_2 G_2$ via $(10,10)$ \\

\hline \noalign{\smallskip}

$E_8$ & $D_8$ & $p=2$ & $B_4 (\ddagger)$ (see Lemma \ref{b4notation}) \\

& & $p=2$ & $B_2 \hookrightarrow B_2^2 (\ddagger)$ via $(10,10^{[r]})$ $(r \neq 0)$, $(10,02)$ or $(10,02^{[r]})$ $(r \neq 0)$ (see Lemma \ref{b4notation}) \\

& $A_8$ & $p=3$ & $A_2 \hookrightarrow A_2^2 < A_8$ via $(10,10^{[r]})$ $(r \neq 0)$ or $(10,01^{[r]})$ $(r \neq 0)$  \\

& $G_2 F_4$ & $p=7$ & $G_2 \hookrightarrow G_2 G_2 < G_2 F_4$ via $(10,10)$ \\

\hline
\end{longtable}
\vspace{-0.3cm}
Again, the notation in the table is explained in Section \ref{nota}.  

A natural question to ask is whether $G$-irreducible subgroups of a certain type exist, especially in small characteristics. When $G$ is a simple exceptional algebraic group \cite[Theorem 2]{LT} (corrected in \cite[Theorem 7.4]{bon}) shows that $G$-irreducible $A_1$ subgroups exist, except for $G = E_6$ when $p=2$. The following corollary shows that  $G$-irreducible $A_2$ subgroups almost always exist.

\begin{cor*} \label{A2subgroups}
Let $G$ be an exceptional algebraic group. Then $G$ contains a $G$-irreducible $A_2$ subgroup, unless $G = E_7$ and $p=2$. 
\end{cor*}

Given the existence of irreducible $A_2$ subgroups, we study their overgroups. In many cases there exists a unique reductive maximal connected subgroup $M$ containing a representative of each conjugacy class of $G$-irreducible $A_2$ subgroups. 

\begin{cor*} \label{A2overgroups}
Let $G$ be an exceptional algebraic group. Then there exists a reductive, maximal connected subgroup $M$ containing representatives of every $\text{Aut}(G)$-conjugacy class of $G$-irreducible $A_2$ subgroups, unless $(G,p)$ is one of the following:  $(G_2,3)$, $(E_6, p \neq 2)$, $(E_7, p \geq 5)$ or $(E_8, p \neq 3)$, in which cases either two or three reductive, maximal connected subgroups are required. The following table lists such overgroups $M$.   
\end{cor*}
\setlength\LTleft{0.05\textwidth}

\begin{longtable}{p{0.1\textwidth - 2\tabcolsep}>{\raggedright\arraybackslash}p{0.3\textwidth-2\tabcolsep}>{\raggedright\arraybackslash}p{0.25\textwidth-2\tabcolsep}>{\raggedright\arraybackslash}p{0.25\textwidth-\tabcolsep}@{}}

\caption{Maximal connected overgroups for $G$-irreducible $A_2$ subgroups. \label{A2overgroupstab}} \\

\hline \noalign{\smallskip}

$G$ & $p \geq 5$ & $p=3$ & $p=2$ \\

\hline \noalign{\smallskip}

$G_2$ & $A_2$ &  $A_2$ and $\tilde{A}_2$ & $A_2$ \\

$F_4$ & $A_2 \tilde{A}_2$ & $A_2 \tilde{A}_2$ & $A_2 \tilde{A}_2$ \\

$E_6$ & $A_2^3$ and $A_2$ & $A_2^3$, $A_2 G_2$ and $G_2$ & $A_2^3$  \\

$E_7$ & $A_2 A_5$ and $A_2$ & $A_2 A_5$ & \textbf{---}\\

$E_8$ & $A_2 E_6$ and $D_8$ & $A_2 E_6$ & $A_2 E_6$ and $D_8$ \\
 
\hline 

\end{longtable}

\setlength\LTleft{0pt}

We require the following definition before discussing the final set of corollaries. As before, let $G$ be a simple exceptional algebraic group.

\begin{defn} \textup{\cite[p. 263]{LS4}}
A simple, simply connected subgroup $X$ of $G$ is \emph{restricted} if all composition factors of $L(G) \downarrow X$ are restricted if $X \neq A_1$, and are of high weight at most $2p-2$ if $X=A_1$.    
\end{defn}

In Section \ref{vars} we prove a set of corollaries which extend \cite[Corollary 1]{LS4}. This states that when $p$ is a good prime for $G$, any simple $G$-cr subgroup $X$ is contained in a uniquely determined commuting product $Y_1 \ldots Y_k$ such that each $Y_i$ is a simple restricted subgroup of the same type as $X$ and each of the projections $X \rightarrow Y_i / Z(Y_i)$ is non-trivial and involves a different field twist. For each $G$, we extend this to all characteristics for each simple, connected $G$-irreducible subgroup of rank at least $2$. In each case, we obtain a small number of counterexamples in bad characteristic.       

We briefly describe the strategy for the proofs of Theorems 1--3 (see Section \ref{strat} for further details). Theorem \ref{max} lists all the maximal connected subgroups that are reductive and have no $A_1$ simple factor. For each of these maximal subgroups $M$, we find all simple $M$-irreducible subgroups of rank at least $2$ and call these ``candidate'' subgroups. It then remains to investigate which of the candidate subgroups are $G$-irreducible, and to solve the conjugacy problem for the candidates. 

The proofs of Theorems 2 and 3, as well as Corollary \ref{nongcr}, have some interesting features, notably when proving an $M$-irreducible subgroup is not $G$-irreducible. These include non-abelian cohomology (applied to the unipotent radicals of parabolic subgroups), finite subgroups and computations in Magma \cite{magma}. See Lemmas \ref{badA2}, \ref{badB4}, \ref{bada2a2} and \ref{badG2} for examples.

\section{Notation} \label{nota}

Let $G$ be a simple algebraic group over an algebraically closed field $K$. Let $\Phi$ be the root system of $G$ and $\Phi^+$ be the set of positive roots in $\Phi$. Write $\Pi = \{ \alpha_1, \ldots, \alpha_l \}$ for the simple roots of $G$ and $\lambda_1, \ldots, \lambda_l$ for the fundamental dominant weights of $G$, both with respect to the ordering of the Dynkin diagram as given in \cite[p. 250]{bourbaki}. We sometimes use $a_1 a_2 \ldots a_l$ to denote a dominant weight $a_1 \lambda_1 + a_2 \lambda_2 + \cdots + a_l \lambda_l$. We denote by $V_G(\lambda)$ (or just $\lambda$) the irreducible $G$-module of dominant high weight $\lambda$. The Weyl module of high weight $\lambda$ is denoted $W_G(\lambda)$ (or just $W(\lambda)$). Another module we refer to frequently is the adjoint module for $G$, which we denote $L(G)$. We let $V_7 : = W_{G_2}(10)$, $V_{26}:=W_{F_4}(0001)$, $V_{27} := V_{E_6}(\lambda_1)$ and $V_{56}:=V_{E_7}(\lambda_7)$. For a $G$-module $V$, we let $V^*$ denote the dual module of $V$. If $Y = Y_1 Y_2 \ldots Y_k$, a commuting product of simple algebraic groups, then $(V_1, \ldots, V_k)$ denotes the $Y$-module $V_1 \otimes \dots \otimes V_k$  where each $V_i$ is an irreducible $Y_i$-module. The notation $\bar{X}$ denotes a subgroup of $Y$ that is generated by long root subgroups of $Y$. If $Y$ has short root elements then $\tilde{X}$ means $\tilde{X}$ is generated by short root subgroups. 

Suppose char($K)=p < \infty$ (recalling that characteristic 0 is denoted $p = \infty$). Let $F: G \rightarrow G$ be the standard Frobenius endomorphism (acting on root groups $U_\alpha = \{ u_\alpha(c) | c \in K\}$ by $u_{\alpha}(c) \mapsto u_{\alpha}(c^p)$) and $V$ be a $G$-module afforded by a representation $\rho: G \rightarrow GL(V)$. Then $V^{[r]}$ denotes the module afforded by the representation $\rho^{[r]} : = \rho \circ F^r$. Let $M_1, \ldots, M_k$ be $G$-modules and $n_1, \ldots, n_k$ be positive integers. Then $M_1^{n_1} / \ldots / M_k^{n_k}$ denotes an $G$-module having the same composition factors as $M_1^{n_1} \oplus \dots \oplus M_k^{n_k}$. Furthermore, $V = M_1 | \ldots | M_k$ denotes an $G$-module with a socle series as follows: $\textrm{Soc}(V) \cong M_k$ and $\textrm{Soc}(V/M_i) \cong M_{i-1}$ for $k \geq i > 1$. Sometimes, to make things clearer, we will use a tower of modules $$\cfrac{M_1}{\cfrac{M_2}{M_3}}$$ to mean the same as $M_1 | M_2 | M_3$.

We need a notation for diagonal subgroups of $Y = H_1 H_2 \ldots H_k$, a commuting product, where all of the subgroups $H_i$ are simple and of the same type; call the simply connected group of this type $H$. Let $\hat{Y} = H \times H \times \cdots \times H$, the direct product of $k$ copies of $H$. Then we may regard $Y$ as $\hat{Y} / Z$ where $Z$ is a subgroup of the centre of $\hat{Y}$ and $H_i$ is the image of the $i$th projection map. A diagonal subgroup of $\hat{Y}$ is a subgroup $\hat{X} \cong H$ of the following form: $\hat{X} = \{ (\phi_1(h), \ldots, \phi_k(h)) | h \in H \}$ where each $\phi_i$ is an endomorphism of $H$. A diagonal subgroup $X$ of $Y$ is the image of a diagonal subgroup of $\hat{Y}$ under the natural map $\hat{Y} \rightarrow Y$. To describe such a subgroup it therefore suffices to give an endomorphism, $\phi_i$, of $H$ for each $i$. By \cite[Chapter 12]{stei}, $\phi_i = \alpha \theta_i  F^{r_i}$ where $\alpha$ is an inner automorphism, $\theta_i$ is a graph morphism and $F^{r_i}$ is a power of the standard Frobenius endomorphism. We only wish to distinguish these diagonal subgroups up to conjugacy and therefore assume $\alpha$ is trivial. For each $1 \leq i \leq k$ we must give a (possibly trivial) graph automorphism $\theta_i$ of $H$, and a non-negative integer $r_i$. 

Such a diagonal subgroup $X$ is denoted ``$X \hookrightarrow H_1 H_2 \ldots H_k$ via $(\lambda_1^{[\theta_1 r_1]}, \lambda_1^{[\theta_2 r_2]},  \ldots, \lambda_1^{[\theta_k r_k]})$''. We often abbreviate this to ``$X$ via $(\lambda_1^{[\theta_1 r_1]}, \ldots, \lambda_1^{[\theta_k r_k]})$'' if the group $Y$ is clear. Unless $X$ is of type $D_n$ ($n \geq 4$), a graph automorphism is uniquely determined by the image of $\lambda_1$ (including the special isogeny from $B_n$ to $C_n$ which takes $\lambda_1$ to $2\lambda_1$). In these cases, instead of writing $\lambda_1^{[\theta_i r_i]}$ we write $\mu^{[r_i]}$ where $\mu$ is the image of $\lambda_1$ under $\theta_i$. The only time we need a diagonal subgroup of a product of type $D_n$ subgroups is when dealing with $D_4^2$. We give a notation for the graph automorphisms of $D_4$: denote an order 3 automorphism by $\tau$ and an involutory automorphism by $\iota$. We usually use the letters $r, s, t, u, \dots $ to be the field twists and they are always assumed to be distinct. 
 
Let $J = \{\alpha_{j_1}, \alpha_{j_2}, \ldots, \alpha_{j_r}\} \subseteq \Pi$ and define $\Phi_J = \Phi \cap \mathbb{Z}J$. Then the standard parabolic subgroup corresponding to $J$ is the subgroup $P = \left< T, U_\alpha \, : \, \alpha \in \Phi_J \cup \Phi^+ \right>$. The Levi decomposition of $P$ is $P = Q L$ where $Q = R_{u}(P) = \left< U_\alpha \, | \, \alpha \in \Phi^+ \setminus \Phi_J \right>$, and $L = \left< U_\alpha \, | \, \alpha \in \Phi_J \right>$. For $i \geq 1$ we define
\[ Q(i) = \left< U_\alpha \, \middle| \, \alpha = \sum_{j \in \Pi} c_j \alpha_j \mathrm{ \ where \ } \sum_{j \in \Pi \setminus J} c_j \geq i \right>, \]
which is a subgroup of $Q$. The \emph{$i$-th level} of $Q$ is $Q(i)/Q(i+1)$, and this is central in $Q/Q(i+1)$. Moreover, by \cite[Theorem 2]{ABS} each level of $Q$ has the structure of a completely reducible $L$-module.

\section{Preliminaries}

Before proving Theorems 1--3 we present some of the results needed in the proofs. Let $G$ be a simple algebraic group over an algebraically closed field of characteristic $p$. The first of these results gives us a starting point for our strategy, which is described in Section \ref{strat}. When we say a reductive, maximal closed connected subgroup we mean a subgroup that is maximal among all closed connected subgroups and is reductive. 
   
\begin{thm}\textup{\cite[Corollary 2]{LS1}} \label{max}
The following tables give the reductive, maximal closed connected subgroups $M$ of $G= E_6$, $E_7$ and $E_8$ with each simple component having rank at least $2$. We also give the composition factors of the restrictions to $M$ of $V_{27}$, $V_{56}$ and $L(G)$. 
\end{thm}

$G = E_6$
\setlength\LTleft{0.04\textwidth}
\begin{longtable}{p{0.14\textwidth - 2\tabcolsep}>{\raggedright\arraybackslash}p{0.3\textwidth-2\tabcolsep}>{\raggedright\arraybackslash}p{0.48\textwidth-\tabcolsep}@{}}

\hline \noalign{\smallskip}

$M$ & Comp. factors of $V_{27} \downarrow M$ & Comp. factors of $L(E_6) \downarrow M$ \\

\hline \noalign{\smallskip}

 $F_4$ & $W(0001) /$ $\! 0000$ & $W(1000) /$ $\! W(0001)$ \\
 $C_4$ $(p \neq 2)$ & $0100$ & $2000 /$ $\! 0001$ \\
 $A_2^3$ & $(10,01,00) /$ $\! (00,10,01) /$ $\! (01,00,10)$ & $(W(11),00,00) /$ $\! (00,W(11),00) /$ $\! (00,00,W(11)) /$ $\! (10,10,10) /$ $\! (01,01,01)$ \\
 $A_2 G_2$ & $(10,W(10)) /$ $\! (W(02),00)$ & $(W(11),W(10)) /$ $\! (W(11),00) /$ $\! (00,W(01))$  \\
 $G_2$ $(p \neq 7)$ & $W(20)$ & $W(01) /$ $\! W(11)$ \\
 $A_2$ $(p \geq 5)$ & $W(22)$ & $11 /$ $\! 41 /$ $\! 14$ \\

\hline

\end{longtable}

$G = E_7$
\begin{longtable}{p{0.14\textwidth - 2\tabcolsep}>{\raggedright\arraybackslash}p{0.3\textwidth-2\tabcolsep}>{\raggedright\arraybackslash}p{0.48\textwidth-\tabcolsep}@{}}

\hline \noalign{\smallskip}

$M$ & Comp. factors of $V_{56} \downarrow M$ & Comp. factors of $L(E_7) \downarrow M$ \\

\hline \noalign{\smallskip}

 $A_7$ & $0100000 /$ $\! 0000010$ & $W(1000001) /$ $\! 0001000$ \\
 $A_2 A_5$ & $(10,10000) /$ $\! (01,00001) /$ $\! (00,00100)$ & $(W(11),00000) /$ $\! (00,W(10001)) /$ $\! (10,00010) /$ $\! (01,01000)$ \\
 $G_2 C_3$ & $(W(10),100) /$ $\! (00,W(001))$ & $(W(10),W(010)) /$ $\! (W(01),000) /$ $\! (00,W(200))$  \\
 $A_2$ $(p \geq 5)$ & $W(60) /$ $\! W(06)$ & $W(44) /$ $\! 11$ \\

\hline

\end{longtable}

$G = E_8$
\begin{longtable}{p{0.14\textwidth - 2\tabcolsep}>{\raggedright\arraybackslash}p{0.78\textwidth-\tabcolsep}@{}}

\hline \noalign{\smallskip}

$M$ & Comp. factors of $L(E_8) \downarrow M$ \\

\hline\noalign{\smallskip}

 $D_8$ & $W(0100000) /$ $\! 00000010$  \\
 $A_8$ & $W(1000001) /$ $\! 00100000 /$ $\! 00000100$ \\
 $A_2 E_6$  & $(W(11),000000) /$ $\!(00,W(010000) /$ $\! (10,000001) /$ $\! (01,100000)$  \\
 $A_4^2$  & $(W(1001),0000) /$ $\! (0000,W(1001)) /$ $\! (1000,0100) /$ $\! (0001,0010) /$ $\! (0100,0001) /$ $\! (0010,1000)$ \\
 $G_2 F_4$  & $(W(10),W(0001)) /$ $\! (W(01),0000) /$ $\! (00,W(1000))$ \\
 $B_2$ $(p \geq 5)$  & $02 /$ $\! W(06) /$ $\! W(32)$ \\ 

\hline
\end{longtable}

Note that in the cases in Theorem \ref{max} where $M$ is of maximal rank, the composition factors are not given in \cite{LS1} but are straightforward to calculate; moreover, for $(G,M) = (E_6,A_2^3)$, $(E_7, A_2 A_5)$, $(E_8,D_8)$ and $(E_8,A_4^2)$ we have made a choice of simple system within each factor.  

Next we state some results which allow us to find the $M$-irreducible subgroups of each $M$ in Theorem \ref{max}. If $M$ is a classical simple group we can use the following. 
\pagebreak
\begin{lem} \textup{\cite[Lemma 2.2]{LT}} \label{class}
Suppose $G$ is a classical simple algebraic group, with natural module $V = V_G(\lambda_1)$. Let $X$ be a semisimple connected closed subgroup of $G$. If $X$ is $G$-irreducible then one of the following holds:

\begin{enumerate}[leftmargin=*,label=\normalfont(\roman*),widest=iii, align=left]

\item $G = A_n$ and $X$ is irreducible on $V$; 

\item $G = B_n, C_n$ or $D_n$ and $V \downarrow X = V_1 \perp \ldots \perp V_k$ with the $V_i$ all non-degenerate, irreducible and inequivalent as $X$-modules;

\item $G = D_n$, $p=2$, $X$ fixes a non-singular vector $v \in V$, and $X$ is a $G_v$-irreducible subgroup of $G_v = B_{n-1}$.
\end{enumerate}
\end{lem}

The next two results give us an explicit list of $G$-irreducible subgroups of rank at least $2$, for $G_2$ and $F_4$. The first is clear: the only reductive maximal connected subgroups without rank 1 factors are isomorphic to $A_2$ and therefore there are no further subgroups to consider. The composition factors of $V_{G_2}(10)$ and $L(G)$ can be found by considering $G_2 < D_4$.  

\begin{lem} \label{simpleg2}
Suppose $X$ is a simple, connected irreducible subgroup of $G_2$ of rank at least $2$. Then $X$ is $G_2$-conjugate to exactly one subgroup of Table \ref{G2tab}. 
\end{lem}

\begin{thm} \textup{\cite[Theorem 4]{dav}} \label{simplef4} 
Suppose $X$ is a simple, connected irreducible subgroup of $F_4$ of rank at least $2$. Then $X$ is $F_4$-conjugate to exactly one subgroup of Table \ref{F4tab}. 
\end{thm}

We now describe some elementary results about $G$-irreducible subgroups. 

\begin{lem} \textup{\cite[Lemma 2.1]{LT}} \label{semirr}
If $X$ is a connected $G$-irreducible subgroup of $G$, then $X$ is semisimple and $C_G(X)$ is finite.   
\end{lem}

\begin{lem} \label{easy}
Suppose a $G$-irreducible subgroup $X$ is contained in $K_1 K_2$, a commuting product of connected non-trivial subgroups $K_1$, $K_2$ of $G$. Then $X$ must have a non-trivial projection to both $K_1$ and $K_2$. Moreover, each projection must be a $K_i$-irreducible subgroup. 
\end{lem}

\begin{proof}
The first assertion is clear by Lemma \ref{semirr}. For the second statement, suppose the projection to $K_1$ is contained in a parabolic, $P$, of $K_1$. Then $X < P K_2$ which is a parabolic subgroup of $K_1 K_2$ and therefore by the Borel-Tits Theorem \cite{BT}, $X$ is contained in a parabolic subgroup of $G$, a contradiction.     
\end{proof}
\vspace{-0.15cm}
For the next two results, recall the definition from the introduction of two algebraic groups having the same composition factors. 

\begin{lem} [{\cite[Lemma 3.4.3]{dav}}] \label{unip}
Let $H$ be a reductive algebraic group, $Q$ be a unipotent group on which $H$ acts on algebraically, and $X$ be a complement to $Q$ in the semidirect product $H Q$. Suppose $V$ is a rational $H Q$-module. Then the composition factors of $H$ on $V$ are the same as the composition factors of $X$ on $V$.
\end{lem}

\begin{lem} \label{wrongcomps}
Suppose $X < G$ is semisimple and $V$ is a $G$-module. Assume the composition factors of $V \downarrow X$ are not the same as those of $V \downarrow H$ for any group $H$ such that 
\begin{enumerate}[label=\normalfont(\roman*),leftmargin=*,widest=ii, align=left]

\item $H$ is of the same type as $X$, or $p=2$ and $X \cong B_n$, $H \cong C_n$, and

\item $H \leq L'$ and is $L'$-irreducible, for some Levi subgroup $L$. 

\end{enumerate}
Then $X$ is $G$-irreducible.  

\end{lem}

\begin{proof}
Suppose $X < P$ for some parabolic subgroup of $G$, minimal with respect to containing $X$. Let $P = QL$ be the Levi decomposition, so $X < QL'$. Hence, there exists some subgroup $H \leq L'$, with $QH = QX$. Furthermore, $H$ is $L'$-irreducible (by minimality) and \cite[Lemma 3.6.1]{ste1} shows $H$ is of the same type as $X$, or if $p=2$, $X \cong B_n$ and $H \cong C_n$. This is a contradiction because Lemma \ref{unip} shows that the composition factors of $V \downarrow X$ and $V \downarrow H$ are the same.   
\end{proof}

\begin{cor} \label{notrivs}
Suppose $X < G$ is semisimple and $L(G) \downarrow X$ has no trivial composition factors. Then $X$ is $G$-irreducible. 
\end{cor}

\begin{proof}
Suppose $X$ is $G$-reducible. Then by Lemma \ref{wrongcomps} (with $V = L(G)$) there exists a subgroup $H$ of some Levi subgroup $L$ such that the composition factors of $L(G) \downarrow H$ are the same as $L(G) \downarrow X$.  But $H < L$, so $L(G) \downarrow H$ has trivial composition factors coming from $L(Z(L))$, a contradiction. 
\end{proof}

\begin{lem} \textup{\cite[Lemma 1.3]{se2}} \label{fix}
Let $0 \neq l \in L(G)$ and $C = C_G(l)$. Then:

\begin{enumerate}[label=\normalfont(\roman*),leftmargin=*,widest=ii, align=left]
\item if $l$ is semisimple, then $C$ contains a maximal torus of $G$;

\item if $l$ is nilpotent, then $R_u(C) \neq 1$ and hence $C$ is contained in a proper parabolic subgroup of $G$.  
\end{enumerate}
\end{lem}

The next result is \cite[Prop. 1.4]{LS7} with $X$ allowed to be semisimple rather than simple; the proof is the same. 

\begin{lem} \textup{\cite[Prop. 1.4]{LS7}} \label{subspaces}
Let $X$ be a semisimple, connected algebraic group over $K$ and let $S$ be a finite subgroup of $X$. Suppose $V$ is a finite-dimensional $X$-module satisfying the following conditions:

\begin{enumerate}[leftmargin=*,label=\normalfont(\roman*),widest=iii, align=left]
\item every $X$-composition factor of $V$ is an irreducible $S$-module;

\item for any $X$-composition factors $M$, $N$ of $V$, the restriction map $\mathrm{Ext}^1_X(M,N) \rightarrow \mathrm{Ext}^1_S(M,N)$ is injective;

\item for any $X$-composition factors $M,N$ of $V$, if $M \downarrow S \cong N \downarrow S$, then $M \cong N$ as $X$-modules.  

\end{enumerate}
Then $X$ and $S$ fix precisely the same subspaces of $V$. 
\end{lem}

The following result is a slight generalisation of \cite[Lemma 1.2 (ii)]{LST}, proved with a small modification of the proof of \cite[Prop. 3.6(iii)]{lit}.

\begin{lem} \label{litlem} 
Let $X$ be a semisimple algebraic group and $M$ a finite-dimensional $X$-module. Let $W_1, \dots, W_r$ be the composition factors of $M$, of which $m$ are trivial, and set $n = \sum {\rm dim\ } H^{1}(X,W_i)$. Assume that $H^{1}(X,K)=0$ and for each $i$ we have
\[H^{1}(X,W_i) = \left\{0\right\} \Longleftrightarrow H^{1}(X,W_i^{*}) = \left\{0\right\}.\] If $m \geq n>0$, then $M$ has either a trivial submodule or a trivial quotient. In particular, if $M$ is self-dual then it contains a trivial submodule.
\end{lem}

\begin{lem} \label{b4cohom}

Let $X = B_4$ and $p=2$. Then the following Weyl modules for $X$ have the given socle series and are uniserial. 

\begin{enumerate}[label=\normalfont(\roman*),leftmargin=*,widest=iii, align=left]
\item $W(\lambda_1) = \lambda_1 | 0$.
\item $W(\lambda_2) = \lambda_2 | 0 | \lambda_1 | 0$.
\item $W(\lambda_3) = \lambda_3 | \lambda_2 | 0 | \lambda_1 | 0$. 

\end{enumerate}
Furthermore, letting $M$ be the uniserial module $0 | \lambda_2 | 0$, we have the following cohomology groups. 
\begin{enumerate}[label=\normalfont(\arabic*),leftmargin=*]
\item $H^1(X,\lambda_1) = K$.
\item $H^1(X,M) = H^1(X,\lambda_3) = 0$.
\item $H^2(X, \lambda_1) = H^2(X, M) = 0$. 
\end{enumerate}

\end{lem}

\begin{proof}
The structure of the Weyl modules is given by \cite[Lemma 7.2.2]{LS1}. The assertions for $H^1(X,\lambda_1)$ and $H^1(X, \lambda_3)$ follow from these. Now, consider the short exact sequence of modules $0 \rightarrow A \rightarrow B \rightarrow C \rightarrow 0$. This gives us a long exact sequence of cohomology groups: $$0 \rightarrow H^0(X,A) \rightarrow H^0(X,B) \rightarrow H^0(X,C) \rightarrow H^1(X,A) \rightarrow H^1(X,B) \rightarrow H^1(X,C) \rightarrow  \cdots$$ We can apply this to the short exact sequence $0 \rightarrow K \rightarrow M \rightarrow M/K \rightarrow 0$ (where we use $K$ for the 1-dimensional trivial module and $M / K$ is the uniserial module $0 |\lambda_2$). We obtain the long exact sequence $$0 \rightarrow K \rightarrow K \rightarrow 0 \rightarrow 0 \rightarrow H^1(X,M) \rightarrow 0 \rightarrow 0 \rightarrow H^2(X,M) \rightarrow H^2(X,M/K) \rightarrow 0 \rightarrow \cdots$$ Immediately $H^1(X,M)$ must be 0 by exactness and $H^2(X,M) \cong H^2(X,M/K)$. Similarly, considering the short exact sequence $0 \rightarrow V_{B_4}(\lambda_2) \rightarrow M/K \rightarrow K \rightarrow 0$ we find $H^2(X,M/K) \cong H^2(X,\lambda_2)$. Using the dimension-shifting identity \cite[II. 4.14]{Jan}, $H^2(X,\lambda_2) = H^1(X,0|\lambda_1|0)$ and since $0 | \lambda_1 | 0$ is tilting, it follows from \cite[II. 4.13]{Jan} that $H^1(X,0|\lambda_1|0) = 0$. Hence $H^2(X,M) \cong H^2(X,\lambda_2) = 0$. Finally, using the dimension-shifting identity again, $H^2(X,\lambda_1) = H^1(X,K) = 0$.   
\end{proof}

\begin{lem} \label{a2ext}

Let $X = A_2$ and $p=3$. Then the following hold: 
\begin{enumerate}[label=\normalfont(\roman*),leftmargin=*,widest=vi, align=left]
\item $\mathrm{Ext}^1_X(22,00) \cong \mathrm{Ext}^1_X(22,11) \cong \mathrm{Ext}^1_X(22,30) \cong \mathrm{Ext}^1_X(22,03) = 0$;
\item $\mathrm{Ext}^1_X(11^{[i]},00) \cong K$ for all $i \geq 0$;
\item $\mathrm{Ext}^1_X(30^{[i]},11) \cong \mathrm{Ext}^1_X(30^{[i]},11^{[i]}) \cong \mathrm{Ext}^1_X(03^{[i]},11) \cong \mathrm{Ext}^1_X(03,^{[i]},11^{[i]}) \cong K$ for all $i \geq 0$;   
\item $\mathrm{Ext}^1_X(30^{[i]},00) \cong \mathrm{Ext}^1_X(03^{[i]},00) = 0$ for all $i \geq 0$;
\item $\mathrm{Ext}^1_X(11 \otimes 11^{[j]},11) \cong \mathrm{Ext}^1_X(11 \otimes 11^{[j]},11^{[j]}) \cong K$ for all $j > 0$;
\item $\mathrm{Ext}^1_X(11 \otimes 11^{[j]},00) \cong  \mathrm{Ext}^1_X(11 \otimes 11^{[j]},30) \cong \mathrm{Ext}^1_X(11 \otimes 11^{[j]},30^{[j]})\cong \mathrm{Ext}^1_X(11 \otimes 11^{[j]},03) \cong \mathrm{Ext}^1_X(11 \otimes 11^{[j]},03^{[j]}) = 0$ for all $j > 0$.
\end{enumerate}
\end{lem}

\begin{proof}
All of these can be directly deduced from \cite[Lemma 2.7]{ste3}.
\end{proof}

\begin{lem} \label{repa2}
Let $X = A_2$, $p=3$ and $M$ be a self-dual $X$-module. Suppose the composition factors of $M$ are isomorphic to $11$ or $00$, with at least one trivial composition factor. Then $M$ has a trivial submodule. 

\end{lem}

\begin{proof}
Suppose $M$ has no trivial submodule. Then since $\mathrm{Ext}_X^1(11,11) = 0$, there must exist an indecomposable submodule, $N$ with structure $00 | 11$. If $N$ is a direct summand of $M$ then so is $N^* = 11 | 00$ and $M$ has a trivial submodule, a contradiction. So $N$ is not a direct summand of $M$. Since $\mathrm{Ext}_X^1(11,00) \cong \mathrm{Ext}_X^1(00,11) \cong K$ (using Lemma \ref{a2ext}(ii)), it follows that $M$ must have a submodule isomorphic to $11|00|11$. But \cite[Lemma 4.2.4(i)]{LS1} shows there is no module with socle series $11 | 00 | 11$ (it must split as $(11|00) + 11$), a contradiction.   
\end{proof}

\section{Strategy for the proofs of Theorems 1--3} \label{strat}

We describe the strategy used to prove Theorems 1--3. It relies on the following lemma. 

\begin{lem} \label{irrstrat}
Let $G$ be a simple exceptional algebraic group. Suppose $X$ is a simple, connected $G$-irreducible subgroup of rank at least $2$. Then there exists a reductive, maximal connected closed subgroup $M$ of $G$ each of whose simple components has rank at least $2$, such that $X \leq M$ and $X$ is $M$-irreducible. 
\end{lem}    

\begin{proof}
Let $M$ be a maximal connected subgroup of $G$ containing $X$. As $X$ is $G$-irreducible, $M$ must be reductive. Moreover, $X$ is $M$-irreducible, as any non-trivial parabolic subgroup of $M$ is contained in a parabolic of $G$ by \cite{BT}. Finally, by Lemma \ref{easy}, $M$ can have no rank 1 simple components. 
\end{proof}

Theorem \ref{max} gives us all reductive, maximal connected closed subgroups of $G$ with no rank 1 simple components for $G = E_6, E_7$ and $E_8$. For each such $M$ we must find all simple, connected $M$-irreducible subgroups of rank at least $2$. To avoid repeating the term ``simple, connected subgroup of rank at least $2$'' we introduce the following definition. 

\begin{defn}

Let $G$ be $E_6, E_7$ or $E_8$. We call a subgroup $X$ of $G$ a \textit{$G$-candidate} (or just a \textit{candidate}) if the following hold: 

\begin{enumerate}[label=\normalfont(\arabic*),leftmargin=*]

\item $X$ is connected and simple of rank at least $2$;

\item there exists a reductive, maximal connected subgroup $M$ of $G$  containing $X$ such that $X$ is $M$-irreducible.

\end{enumerate}

\end{defn} 

In Theorems 1--3 we are aiming to find the irreducible subgroups up to $G$-conjugacy. The strategy is as follows: for each reductive, maximal connected subgroup $M$ (from Theorem \ref{max}) we find all $G$-candidate subgroups, up to $M$-conjugacy, contained in $M$. To do this we use Lemma \ref{class} and \cite{lubeck} for classical simple components of $M$, and Lemma \ref{simpleg2} and Theorem \ref{simplef4} for exceptional simple components of $M$ when $G = E_6$, as well as Theorems \ref{Thm1} and \ref{Thm2} when $G = E_8$ for $ M = A_2 E_6$ and $A_1 E_7$, respectively. We then find all $G$-conjugacies between the candidate subgroups contained in the different reductive, maximal connected subgroups.

The last step is to check whether each $G$-conjugacy class of candidate subgroups is $G$-irreducible or not. To do this we  heavily use Lemma \ref{wrongcomps} and Corollary \ref{notrivs}. To apply these results we must find the composition factors of the action of the $G$-candidate on the minimal or adjoint module. These can be found by restricting the composition factors of $M$. This can be done for all $G$-candidate subgroups and the composition factors for the $G$-irreducible subgroups can be found in Section \ref{tabs}, Tables \ref{E6tab}--\ref{E8tab}. To apply Lemma \ref{wrongcomps} we also need the composition factors for the action of the Levi subgroups of $G$ on the minimal and adjoint modules. These can be found in Appendix \ref{app}, Tables \ref{levie6}--\ref{levie8}. In most cases, a $G$-candidate subgroup is $G$-irreducible. Corollary \ref{nongcr} lists the interesting examples of candidate subgroups which are irreducible in every reductive, maximal connected overgroup yet are $G$-reducible. 

To prove a candidate subgroup is $G$-reducible can be difficult. For example, in Lemmas \ref{badA2} and \ref{badB4} we use non-abelian cohomology (applied to the unipotent radical of a parabolic subgroup) to show an $A_2$ subgroup of $A_2 A_5$ is $E_7$-reducible and a $B_4$ subgroup of $D_8$ is $E_8$-reducible, respectively.   

\section{Proof of Theorem 1: $E_6$-irreducible subgroups} \label{e6proof}

 We start by considering each of the reductive, maximal connected subgroups of $E_6$ that do not have an $A_1$ factor, as listed in Theorem \ref{max}. They are $F_4$, $C_4$ $(p \neq 2)$, $A_2^3$, $A_2 G_2$, $G_2$ $(p \neq 7)$ and $A_2$ $(p\geq 5)$. For each maximal subgroup we determine all $E_6$-candidate subgroups (definition in Section \ref{strat}) up to $E_6$-conjugacy. To check we have found all such conjugacies is straightforward; no two subgroups listed in Table \ref{E6tab} have the same composition factors on $L(E_6)$. We then prove those and only those listed in Table~\ref{E6tab} are $E_6$-irreducible.

\subsection{Maximal $M = F_4$} \label{maxf4}

We use Theorem \ref{simplef4} to find all $E_6$-candidate subgroups contained in $F_4$. This maximal $F_4$ in $E_6$ can be obtained as the fixed points of the standard graph automorphism of $E_6$. This allows us to make a number of observations. Suppose $X \leq B_4 < F_4$. Then we must have that $X < D_5$ and therefore $X$ is not $E_6$-irreducible (since $D_5$ is a Levi subgroup of $E_6$). Therefore $B_4$, $\bar{D}_4$ and any $B_2 \hookrightarrow B_2^2$ $(p=2)$ are not $E_6$-irreducible. Now suppose $X < \bar{A}_2 \tilde{A}_2 < F_4$. The long $\bar{A}_2$ subgroup inside $F_4$ is generated by root subgroups of $E_6$ and therefore $X < \bar{A}_2 C_{E_6}(\bar{A}_2) = \bar{A}_2^3$. We will study the candidate subgroups of $\bar{A}_2^3$ later. We are left with the following $E_6$-candidate subgroups (from Table \ref{F4tab})  to consider: $C_4$ $(p=2)$, $\tilde{D}_4$ $(p=2)$, $G_2$ $(p=7$). 

\begin{lem}

The subgroups $C_4$ $(p=2)$, $\tilde{D}_4$ $(p=2)$ and $G_2$ $(p=7)$ are all $E_6$-irreducible. 

\end{lem}

\begin{proof}

Each of the candidate subgroups has a 26-dimensional composition factor on $V_{27}$. Therefore we can use Lemma \ref{wrongcomps} on $V_{27}$, since Table \ref{levie6} shows that no Levi subgroup has a composition factor of dimension at least 26.
\end{proof}

\subsection{Maximal $M = C_4$ $(p \neq 2)$}

Using Lemma \ref{class} and \cite{lubeck} we see the only reductive maximal connected subgroup of $C_4$ without an $A_1$ simple factor is $C_2^2$ when $p \neq 2$. We note that this $C_4$ is generated by root subgroups of $E_6$ and therefore the $C_2^2$ is generated by root subgroups. However, $C_{E_6}(C_2) = C_2 T_1$ from \cite[Table 3]{LS6}. Therefore any subgroup of $C_2^2 < C_4$ centralises a $T_1$ and is therefore not $E_6$-irreducible. 

\subsection{Maximal $M = A_2^3$}

All $A_2$ diagonal subgroups that have a non-trivial projection onto each factor are $A_2^3$-irreducible and form all of the candidate subgroups. The conjugacy classes of these diagonal subgroups are found in \cite[Table 8.3]{LS3} (noting that finding the conjugacy classes is unaffected by the restriction imposed on the characteristic there), up to $\text{Aut}(E_6)$-conjugacy. Note that we fix a copy of $A_2^3$ in Theorem \ref{max} by giving its composition factors on $V_{27}$ and $L(E_6)$. The following lemma shows all but one candidate subgroup is $E_6$-irreducible.

\begin{lem} \label{a2e6}

The subgroup $X = A_2 \hookrightarrow A_2^3$ via $(10,10,10)$ is not $E_6$-irreducible. Every other $A_2$ candidate subgroup contained in $A_2^3$ is $E_6$-irreducible. 

\end{lem}

\begin{proof}

Suppose $p \neq 3$. Consider $Y = A_2 < D_4$ (the Levi $D_4$ in $E_6$) embedded via $V_{A_2}(11)$. We claim $Y$ is conjugate to $X$. Indeed, \cite[Table 8.3]{LS3} shows that this is the case for $p > 3$ and the argument given in \cite[p. 64]{LS3} extends to $p = 2$. Therefore $X$ is contained in a parabolic subgroup when $p \neq 3$. When $p=3$, we show that $X < F_4$ and is not $F_4$-irreducible. Indeed, in subsection \ref{maxf4} we showed $A_2 \tilde{A}_2 < A_2^3$ and by comparing composition factors we see $X$ is conjugate to $A_2 \hookrightarrow A_2 \tilde{A}_2$ via $(10,01)$, which Theorem \ref{simplef4} shows is contained in a parabolic subgroup of $F_4$.  

Now consider the other candidate subgroups in $A_2^3$. If $p \neq 3$ then we can apply Corollary \ref{notrivs} (the restrictions are in Table \ref{E6tab}). Now suppose $p=3$. We apply Lemma \ref{wrongcomps}. Let $Z$ be any of the candidate subgroups other than $X$. Then we have to check whether the composition factors of $Z$ match those of an $L'$-irreducible subgroup of type $A_2$ for some Levi subgroup $L$. The possibilities for $L'$ are $D_5$, $D_4$, $A_5$, $A_4$, $A_3$, $A_2^2$ and $A_2$. The subgroups $D_5$ and $A_5$ contain every subgroup in that list between them. So if we show $Z$ does not match the composition factors of any $A_2$ subgroup of $D_5$ or $A_5$ (not necessarily irreducible) on $V_{27}$ then that is enough. 

From Table \ref{levie6}, $V_{27} \downarrow D_5 = \lambda_1 / \lambda_4 / 0$ and $V_{27} \downarrow A_5 = \lambda_1^2 / \lambda_4$. The dimensions for the composition factors are $16, 10, 1$ for $D_5$ and $15, 6, 6$ for $A_5$.  Depending on which embedding we take into $A_2^3$, the list of dimensions of composition factors of $V_{27} \downarrow Z$ is one of the following: 7,6,6,3,3,1,1 or 9,9,7,1,1 or 9,9,6,3 or 9,9,9 (from Table \ref{E6tab}). We need to show it is not possible for any of these to correspond with an $A_2$ subgroup of $D_5$ or $A_5$. The latter two, 9,9,6,3 and 9,9,9 cannot match; they have no 1-dimensional composition factor, ruling out $D_5$ and have two composition factors of dimension 9, ruling out $A_5$. Now suppose $Z$ has composition factors of dimensions 9,9,7,1,1. Then $A_5$ is ruled out by the previous reasoning. The only way for an $A_2$ to be contained in $D_5$ with matching composition factors is if $V_{D_5}(\lambda_1) \downarrow A_2$ has a 9-dimensional and a trivial composition factor. This is a contradiction because none of the 9-dimensional composition factors  of $Z$ is self-dual. Finally, suppose $Z$ has the list 7,6,6,3,3,1,1 for its dimensions of composition factors, so $V_{27} \downarrow Z = 10/01/20/02/11/00^2$. If $Z$ is contained in $A_5$ then $V_{A_5}(\lambda_1) \downarrow Z$ must be $V_{A_2}(20)$ or $V_{A_2}(02)$. But then $V_{A_5}(\lambda_4) \downarrow Z = V_{A_2}(12)$ or $V_{A_2}(21)$ (\cite[Prop. 2.10]{LS3}) which is impossible. Similarly, any $A_2$ contained in $D_5$, with the same composition factors as $Z$, has composition factors of $V_{D_5}(\lambda_1) \downarrow A_2$ of dimensions 7,3 or 6,3,1. But such an $A_2$ does not preserve an orthogonal form, a contradiction. Hence $Z$ is $E_6$-irreducible. 
\end{proof}

\subsection{Maximal $M = A_2 G_2$}

The only possible candidate subgroups are of type $A_2$, so we need to consider $G_2$-irreducible subgroups of type $A_2$. The factor $G_2$ of $M$ is contained in a Levi $D_4$, hence the maximal $\bar{A}_2$ generated by long root subgroups of $G_2$ is in fact generated by root subgroups of $E_6$. So for any candidate subgroup $X < A_2 \bar{A}_2$, we will have $X < \bar{A}_2 C_{E_6}(\bar{A}_2) = \bar{A}_2^3$. We have already considered these in the previous section. By Lemma \ref{simpleg2} we are left to consider the maximal $\tilde{A}_2$ when $p=3$, generated by short root subgroups of $G_2$. As $N_{G_2}(\tilde{A}_2)=\tilde{A}_2.2$ a diagonal subgroup of $A_2 \tilde{A}_2$ with a graph automorphism applied to the second factor is conjugate to one without. Therefore the following two lemmas finish the case of $M = A_2 G_2$. 

\begin{lem}
The subgroup $X = A_2 \hookrightarrow A_2 \tilde{A}_2 < A_2 G_2$ via $(10,10)$ when $p=3$  is not $E_6$-irreducible.
\end{lem}

\begin{proof}
Firstly we note that when $p=3$, $L(E_6)'$ is irreducible of dimension 77 for an adjoint $E_6$. We need to consider the restriction of $L(E_6)'$ to $M = A_2 G_2$, given in \cite[Table 10.1]{LS1}: $$\mfra{L(E_6)' \downarrow A_2 G_2 =} \mfra{(11,00) \ +} \cfrac{(00,10)}{\cfrac{(00,01) + (11,10)}{(00,10)}}$$ where the tower of modules is a socle series. The restrictions of the $G_2$-modules $V_{G_2}(10)$ and $V_{G_2}(01)$ to $\tilde{A}_2$ are as follows (from Table \ref{G2tab}): $V_{G_2}(10) \downarrow \tilde{A}_2 = V_{A_2}(11)$ and $V_{G_2}(01) \downarrow \tilde{A}_2 = V_{A_2}(30) + V_{A_2}(03) + V_{A_2}(00)$. We also note that $V_{A_2}(11) \otimes V_{A_2}(11)$ has a trivial submodule. Therefore, the tower of modules restricted to $X$ has a submodule $ 00^2 | 11$. But $\text{Ext}^1_{A_2}(11,00)$ is 1-dimensional so at least one of those trivial modules must actually occur as a submodule in $L(E_6)' \downarrow X$. Therefore $X$ fixes a non-zero vector of $L(E_6)'$ and we may apply Lemma \ref{fix}. The stabilizer of this vector in $L(E_6)'$ is contained in either a parabolic subgroup or maximal rank subgroup. Then, assuming $X$ is $E_6$-irreducible, it must be contained in $A_2^3$. However, no subgroup of $A_2^3$ has a composition factor $V_{A_2}(22)$ on $L(E_6)'$ (the restriction $L(E_6) \downarrow A_2^3$ is given in Theorem \ref{max}). Therefore $X$ is not contained in $A_2^3$ and is not $E_6$-irreducible.  
\end{proof}

In the next result, recall that $r$ and $s$ are assumed to be distinct non-negative integers and therefore $(10^{[r]},10^{[s]})$ is not equal to $(10,10)$. 

\begin{lem}
The subgroups $A_2 \hookrightarrow A_2 \tilde{A}_2 < A_2 G_2$ via $(10^{[r]},10^{[s]})$ $(rs=0)$ when $p=3$ are all $E_6$-irreducible.
\end{lem}

\begin{proof}
To see this we can apply Lemma \ref{wrongcomps} to $V_{27}$. Let $X$ be one of the subgroups in the statement. As $X$ is of type $A_2$ we just need to show it does not have the same composition factors on $V_{27}$ as any $A_2$ subgroup of $D_5$ or $A_5$. But $V_{27} \downarrow X = 10^{[r]} \otimes 11^{[s]} / 02^{[r]}$ (from Table \ref{E6tab}) so the dimensions of the composition factors are 21,6. This is incompatible with any subgroup of either $D_5$ or $A_5$ as neither has a composition factor of dimension at least 21 on $V_{27}$.  
\end{proof}
\vspace{-0.4cm}
\subsection{Maximal $M = G_2$ $(p \neq 7)$}

The only possible candidate subgroups which are proper in $M$ are of type $A_2$. First we consider the $A_2$ generated by long root subgroups of the $G_2$ (but not of the $E_6$).

\begin{lem}

Let $X$ be the $A_2$ subgroup generated by long root subgroups of a maximal $G_2$ $(p \neq 7)$. Then $X$ is conjugate to $Y =  A_2 \hookrightarrow A_2^3$ via $(10,10,01)$. (So $X$ is $E_6$-irreducible by Lemma \ref{a2e6}.)

\end{lem}

\begin{proof}
If $p \neq 3$ this is straightforward because $X$ is an $SL_3$ inside $G_2$ and has a centre of order 3. The full connected centraliser in $E_6$ of a generator of this centre is $A_2^3$. Therefore $X < A_2^3$ and must be conjugate to $Y$ as it is the only conjugacy class with those composition factors on $V_{27}$. 

Now suppose $p=3$. The composition factors of $X$ and $Y$ agree on $L(E_6)$. Therefore, the proof of Lemma \ref{a2e6} shows that both $X$ and $Y$ must be $E_6$-irreducible. From \cite[p. 215]{LS1} we have that $L(E_6)' \downarrow G_2 = 10|01|11|01|10$, a uniserial module and from Table \ref{G2tab}, $V_{G_2}(10) \downarrow X = 10 + 01 + 00$. This implies that $X$ fixes a non-zero vector in $L(E_6)'$ and we can apply Lemma \ref{fix}. As $X$ is $E_6$-irreducible, it must be contained in a maximal rank subgroup of $E_6$, which must be $A_2^3$. Comparing composition factors shows that $X$ must be conjugate to $Y$. 
\end{proof}

Finally, we must consider $\tilde{A}_2$ when $p=3$. 

\begin{lem}

Let $Z = \tilde{A}_2 < G_2$ when $p=3$. Then $Z$ is $E_6$-irreducible. 

\end{lem}

\begin{proof}
We can apply Lemma \ref{wrongcomps} to $V_{27}$. From Table \ref{E6tab}, $V_{27} \downarrow G_2 = 20$ and we therefore deduce that $V_{27} \downarrow Z = V_{A_2}(22)$. No Levi subgroup has just one composition factor on $V_{27}$, hence $Z$ is $E_6$-irreducible.  
\end{proof}

\subsection{Maximal $M = A_2$ $(p \geq 5)$}

Here there is nothing to prove as any proper simple subgroup of $A_2$ has rank 1 and so the only candidate subgroup is $M$ itself, which is $G$-irreducible by maximality.  
 
This completes the proof of Theorem 1.

\section{Proof of Theorem 2: $E_7$-irreducible subgroups}

We use the same approach as in Section \ref{e6proof} to prove Theorem 2. We start by listing the reductive, maximal connected subgroups of $E_7$ (with no $A_1$ simple factor) from Theorem \ref{max}. These are $A_7$, $A_2 A_5$, $G_2 C_3$ and $A_2$ $(p \geq 5)$. We must consider the $E_7$-candidates contained in each of them in the following subsections.

\subsection{Maximal $M = A_7$}

First, applying Lemma \ref{class} and \cite{lubeck} we find all candidate subgroups of $M$. These are $A_7$, $C_4$, $D_4$, $B_3$ embedded irreducibly via $V_{B_3}(001)$ and $A_2$ $(p \neq 3)$ embedded irreducibly via $V_{A_2}(11)$. The following lemma handles all of these subgroups.  

\begin{lem} \label{a7e7}

The only $E_7$-candidates contained in $A_7$ that are $E_7$-irreducible are $A_7$, $D_4$ $(p > 2)$ and $A_2$ $(p>3)$. Furthermore, $X = A_2$ $(p>3)$ is conjugate to $Y = A_2 \hookrightarrow \bar{A}_2 {A_2}^{(\star)} < \bar{A}_2 A_5$ via $(10,10)$ (where ${A_2}^{(\star)}$ is embedded in $A_5$ via $V_{A_2}(20)$). 

\end{lem}

\begin{proof}

First consider $C_4$. This $C_4$ is generated by long root subgroups of $A_7$, hence of $E_7$. Therefore, the connected centraliser of this $C_4$ is $T_1$ ($p \neq 2$) or a 1-dimensional connected unipotent subgroup ($p =2$) by \cite[Lemma 4.9]{LS6}. As the centraliser is infinite it follows from Lemma \ref{semirr} that this candidate $C_4$ is not $E_7$-irreducible.

Now consider $D_4$. When $p=2$ this is contained in the $C_4$ and therefore is $E_7$-reducible. When $p \neq 2$, Corollary \ref{notrivs} shows it is $E_7$-irreducible (the restrictions are in Table \ref{E7tab}). 

There are no $E_7$-irreducible subgroups isomorphic to $B_3$ however. This is because the $A_7$-irreducible $B_3$ is embedded via $V_{B_3}(001)$ and is therefore contained in $D_4$. The normaliser in $E_7$ of $D_4$ induces a triality automorphism on this $D_4$ (\cite[Lemma 2.15]{clss}), which means this $B_3$ is $E_7$-conjugate to the $B_3$ embedded in $A_7$ via $V_{A_7}(\lambda_1) \downarrow B_3 = W_{B_3}(100) / V_{B_3}(000)$ which is contained in a parabolic subgroup of $A_7$.

Finally, consider $A_2$ $(p \neq 3)$ embedded in $A_7$ via $V_{A_2}(11)$. This is contained in $D_4$. When $p=2$, the $D_4$ is contained in a parabolic subgroup of $E_7$ and hence the candidate $A_2$ is not $E_7$-irreducible. When $p > 3$, $X = A_2$ is the group of fixed points of a triality automorphism of $D_4$ induced by an element $t \in E_7$ of order 3. To find the full centraliser of $t$ in $E_7$ we calculate the dimension of $C_{L(E_7)}(t)$. By restricting from $L(E_7) \downarrow A_7$, it follows that $L(E_7) \downarrow D_4 = 2000 + 0020 + 0002 + 0100$. The triality element $t$ fixes a subgroup $A_2$ of $D_4$ hence fixes a dimension 8 subspace of $L(D_4) = \lambda_2$. It fixes a diagonal submodule of $2000 + 0020 + 0002$ of dimension 35. It follows that \linebreak $\text{dim}(C_{L(E_7)}(t)) = 43$. Therefore, using \cite[Prop. 1.2]{LS8}, the full centraliser of $t$ in $E_7$ is $A_2 A_5$. Hence $X$ is conjugate to a subgroup of $A_2 A_5$ and comparing composition factors shows $X$ is conjugate to $Y$. Corollary \ref{notrivs} shows that $X$ is $E_7$-irreducible. 
\end{proof} 

\subsection{Maximal $M = A_2 A_5$} \label{a2a2starnotation}

We need to find all $A_5$-irreducible subgroups of $A_5$ that are isomorphic to $A_2$. This is straightforward from Lemma \ref{class}. There is just one $A_5$-conjugacy class of irreducible $A_2$ subgroups when $p \neq 2$, with $V_{A_5}(\lambda_1) \downarrow A_2 = 20$ and none if $p=2$. Let ${A_2}^{(\star)}$ be the $A_2$ embedded in $A_5$ via \linebreak $V_{A_5}(\lambda_1) \downarrow A_2^{(\star)} = V_{A_2}(20)$. The normaliser in $E_7$ of $M$ is $M.2$ where the involution on top acts as a graph automorphism on both simple factors, by \cite[Table 10]{car}. A graph automorphism of $A_5$ induces a graph automorphism of ${A_2}^{(\star)}$. Therefore $N_{E_7}(\bar{A}_2 {A_2}^{(\star)}) / C_{E_7}(\bar{A}_2 {A_2}^{(\star)}) \geq \mathbb{Z}_2$ where the involution acts as a graph automorphism on both of the $A_2$ factors. Considering the composition factors of $L(E_7) \downarrow \bar{A}_2 {A_2}^{(\star)}$ shows this is in fact an equality. Therefore the conjugacy classes of candidate subgroups are $A_2 \hookrightarrow \bar{A_2} A_2^{(\star)}$ $(p \neq 2)$ via $(10,10)$, $(10^{[r]},10^{[s]})$ $(rs=0)$, $(10,01)$, $(10^{[r]},01^{[s]})$ $(rs=0)$. We require a technical lemma, before considering which candidate subgroups are $E_7$-irreducible.  

\begin{lem} \label{techa2}

Let $p=3$. Then any subgroup $X \cong A_2$ of $E_7$ with $L(E_7) \downarrow X = 22^3 / 11^6 / 30 / 03 / 00^4$ has a trivial submodule on $L(E_7)$. 

\end{lem}

\begin{proof}
Assume there is no trivial submodule. We use the proof and notation of \cite[Lemma 4.2.6]{LS1}. Define $C = C_{L(E_7)}(L(X))$, an $X$-submodule of $L(E_7)$. Suppose $C$ is non-zero. Then it can only contain the totally twisted $X$-composition factors on $L(E_7)$ (and trivial ones). So $C = 30^x/03^y/00^z$. The module $V_{A_2}(22)$ does not extend any of the other composition factors by Lemma \ref{a2ext}(i) and hence the $22^3$ forms a direct summand of $L(E_7)$. So we consider its complement, call it $M_1$. Now $M_1$ is self-dual and has composition factors $30/03/11^6/00^4$. Suppose $z \neq 0$. Then the trivial composition factors of $C$ are direct summands of $C$ because neither $30$ nor $03$ extends the trivial module by Lemma \ref{a2ext}(iv). Therefore $C$ has a trivial submodule, a contradiction. Hence $z=0$. Now suppose $x=1$ and $y=0$ (or vice versa). There must be an indecomposable direct summand $M_2$ of $M_1$ of the form $03 | (11^{6-a} / 00^{4-b}) | 30$, which is self-dual. Therefore, the complement to $M_2$, call it $M_3$, is self-dual and has composition factors $11^a / 00^b$. Lemma \ref{repa2} shows that if $b \neq 0$ then $M_3$ has a trivial submodule, a contradiction. Hence $b=0$. We can apply Lemma \ref{repa2} again to show that the module in the middle of the indecomposable $M_2$ must contain a trivial submodule. But 30 does not extend the trivial module so we have a trivial submodule, a contradiction. Finally, suppose $x=y=1$. Then by self-duality $C$ must split off as a direct summand and again we apply Lemma \ref{repa2} to the complement to show there is a trivial submodule. 

So we have shown that $C = 0$. Now we can apply \cite[ Lemma 4.2.6]{LS1}. This shows that the multiplicities of composition factors of $L(E_7) \downarrow X$ force $X$ to have a trivial submodule, a final contradiction.
\end{proof}

Before the next lemma, we require the following definition. Suppose $Q$ is a unipotent subgroup of $G$ and $X$ is a reductive algebraic group which acts algebraically on $Q$. Then we say a complement $Y$ to $Q$ in the semidirect product $Q X$ is \emph{non-standard} if $Y$ is not $Q$-conjugate to $X$.

In the proofs of Lemmas \ref{badA2} and \ref{badB4} we use a result from non-abelian cohomology (\cite[Lemma 3.2.11]{dav}). This allows us to deduce results on complements to a non-abelian unipotent radical $Q$ of a parabolic subgroup. To do this, we consider the filtration of $Q$ by levels and then calculate certain abelian cohomology groups for each level. This method is introduced by Stewart in \cite[3.2]{dav} and is subsequently used throughout. Also, in \cite{littho} Litterick and the author use these methods to classify the non-$G$-cr subgroups of exceptional algebraic groups in good characteristic, with an introduction given in Section $3$ there. 

\begin{lem} \label{badA2}
The $E_7$-candidate subgroups contained in $A_2 A_5$ that are $E_7$-irreducible are: $A_2 \hookrightarrow \bar{A}_2  {A_2}^{(\star)} < \bar{A}_2 A_5$ via $(10,10)$ $(p > 3)$, $(10,01)$ $(p >2)$, $(10^{[r]},10^{[s]})$ $(rs=0, p>2)$ and $(10^{[r]},01^{[s]})$ $(rs=0, p>2)$.
\end{lem}

\begin{proof}
The only possible candidate subgroups are the diagonal subgroups of $\bar{A}_2 A_2^{(\star)}$ $(p \neq 2)$ by the discussion before Lemma \ref{techa2}. We must prove they are all $E_7$-irreducible except $A_2 \hookrightarrow \bar{A}_2  {A_2}^{(\star)}$ via $(10,10)$ when $p=3$. If $p \geq 5$ it suffices to use Corollary \ref{notrivs}, with the restrictions given in Table \ref{E7tab}. Now assume $p=3$. First we will use Lemma \ref{wrongcomps} to show $A_2 \hookrightarrow \bar{A}_2 A_2^{(*)}$ via $(10^{[r]},10^{[s]})$, $(10^{[r]},01^{[s]})$, $(10,01)$ are $E_7$-irreducible. Take $Z$ to be any one of these subgroups. Then $Z$ has only two trivial composition factors on $L(E_7)$. Therefore if $Z$ is contained in a parabolic subgroup it has to match the composition factors of an irreducible $A_2$ subgroup of $E_6$, $A_6$ or $A_2 A_4$ (using Table \ref{levie7}). We immediately rule out $A_2 A_4$, since there is no $A_4$-irreducible $A_2$ subgroup. Since $E_6$ has two trivial composition factors on $V_{56}$ and $Z$ has none, we rule out $E_6$ and so we are left with just $A_6$ as a possibility. There is only one $A_6$-irreducible $A_2$ subgroup. Call it $Y$, where $V_{A_6}(\lambda_1) \downarrow Y = V_{A_2}(11)$. The composition factors of $L(E_7) \downarrow Y$ are $22^3/11^6/30/03/00^4$. Therefore $Y$ has more trivial composition factors on $L(E_7)$ than $Z$, and so $Z$ is not a conjugate of $Y$ and hence does not lie in an $A_6$-parabolic.  

Now we must consider $X = A_2 \hookrightarrow \bar{A}_2 A_2^{(\star)}$ via $(10,10)$ (still with $p=3$). Then $L(E_7) \downarrow X = 22^3/11^6/30/03/00^4$, which matches the restriction of $Y$. The composition factors of $X$ and $Y$ match on $V_{56}$ also, both being $30^2/03^2/11^6/00^2$. This suggests $X$ could be contained in an $A_6$-parabolic, which we now prove. Let $P = Q A_6 T_1$, a standard $A_6$-parabolic. Then \cite[Theorem 2]{ABS} shows that $A_6$ acts on the levels of $Q$ and that each level is a direct sum of irreducible $A_6$-modules (see Section \ref{nota} for the definition of a level). Furthermore, $Q$ has two levels, each level being just a single irreducible $A_6$-module. The high weights of these irreducible modules are given below.
\begin{align*}
Q/Q(2) \downarrow A_6 &= \lambda_4 \\
Q(2) \downarrow A_6 &= \lambda_1
\end{align*}
If we write $L$ for the Levi subgroup of $P$ and restrict each level to $Y < L' = A_6$ we obtain the following.
\begin{align*}
Q/Q(2) \downarrow Y &=  22/11/00 \\
Q(2) \downarrow Y &= 11
\end{align*}
We need to know the structure of level 1 more precisely. To do this we note that $A_2$ is contained in a $G_2$ in $L'$. We have $\bigwedge^3(V_{G_2}(10)) = V_{G_2}(20) / V_{G_2}(10)/V_{G_2}(00)$ by \cite[Prop. 2.10]{LS3}. Moreover, this module must be completely reducible as there are no non-trivial extensions between any of the composition factors ($W_{G_2}(20)$ and $W_{G_2}(10)$ are irreducible when $p=3$). Therefore level 1 is completely reducible when restricted to $Y$. Since $W_{A_2}(22) = V_{A_2}(22)$ and $W_{A_2}(11) = V_{A_2}(11)|V_{A_2}(00)$ it follows that $H^1(A_2,22) = 0$ and $H^1(A_2,11) = K$. Therefore $\mathrm{dim} (H^1(A_2,Q(i)/Q(i+1)))=1$ for $i=1, 2$. By \cite[Theorem 1]{ste3}, $H^2(A_2,M_1)=0$ for each direct summand $M_1$ of level 1 and 2. Therefore we can apply \cite[Lemma 3.2.11]{dav}. This shows that every complement to $Q / Q(i+1)$ in $Y Q / Q(i+1)$ lifts to a complement to $Q$ in $Y Q$ for $i = 1,2$. 

All complements have the same composition factors on $L(E_7)$, namely those of $Y$ (by Lemma \ref{unip}). Therefore Lemma \ref{techa2} shows that each complement fixes some non-zero vector $ l \in L(E_7)$. By Lemma \ref{fix}, if $l$ is semisimple then $C_{E_7}(l)$ contains a maximal torus and if it is nilpotent then $R_u(C_{E_7}(l)) \neq 1$. Suppose we have a non-standard $A_2$ complement to $Q$ in $Y Q$ which fixes a non-zero nilpotent vector $l_1$ of $L(E_7)$, call it $Z_1$, and let $C := C_{E_7}(l_1)$. We use \cite[Table 22.1.2]{LS9}, which lists all centralisers of nilpotent elements in $L(E_7)$, to show that $C$ is equal to the centraliser of an element in the nilpotent class $A_2 A_1^3$. To show this, we consider each centraliser in  \cite[Table 22.1.2]{LS9}  and check whether it contains a subgroup $A_2$ with the same composition factors as $Y$. We find that there is only one possibility, the centraliser of an element in the nilpotent class $A_2 A_1^3$, which lies in an $A_6$-parabolic subgroup of $E_7$. From \cite[Table 22.1.2]{LS9} we also have that $C / R_{u}(C) = G_2$ and the dimension of $R_u(C)$ is 35.

We now show that $C$ must lie in $P$ and $R_u(C) < Q$. Suppose $C < P^{g}$ for some $g \in E_7$. Since $Z_1 < C$, we have $Z_1 < P \cap P^g$. As $P$ is a maximal parabolic subgroup we can deduce from \cite[2.8.7, 2.8.8]{carter1} that $P \cap P^{g} = P$ or $L$. But $Z_1$ was chosen to not be $Q$-conjugate to $Y$ and $YQ \cap L = Y$. Hence $Z_1$ is not contained in $L$, and therefore $P = P^{g}$ and $C$ must lie in $P$. The projection of $R_u(C)$ to $P/Q \cong A_6 T_1$ is a closed unipotent subgroup. By the Borel-Tits Theorem \cite{BT}, its normaliser is contained in a parabolic subgroup of $A_6 T_1$; as it contains $G_2$ which is $A_6$-irreducible, this parabolic subgroup must be the whole of $A_6 T_1$. Since the projection of $R_u(C)$ is contained in the unipotent radical of this parabolic, which is trivial, we conclude that $R_u(C) < Q$. 

Suppose $Z_2$ is another non-standard complement to $Q$ in $Y Q$ fixing a non-zero nilpotent vector $l_2$ of $L(E_7)$. Then from the arguments in the previous two paragraphs, we have that $C_{E_7}(l_2) < P$ and $C_{E_7}(l_2) = C^g$ for some $g \in E_7$. Moreover, since $C^g < P^g \cap P = P$ it follows that $g \in N_{E_7}(P) = P$. Therefore, we conclude that $Z_2$ is contained in $C^g$ for some $g \in P$. 

We now prove that there exists a non-standard complement to $Q$ in $Y Q$ not contained in any $P$-conjugate of $C$ and hence fixing a non-zero semisimple vector of $L(E_7)$. Let $Z_3$ be any non-standard complement to $Q$ in $Y Q$ and suppose $Z_3 < C^g$ where $g \in P$. We saw above that the levels of $Q$ restrict to $G_2$ as the following.
\begin{align*}
Q/Q(2) \downarrow G_2 &= 20 + 10 + 00 \\
Q(2) \downarrow G_2 &= 10 
\end{align*}
We now claim that $Q(2)$ is contained in $R_u(C^g)$. We know that $R_u(C^g)$ is a 35-dimensional \linebreak {$G_2$-invariant} subgroup of $Q$ and therefore has a filtration by $G_2$-modules which contains all of the above factors bar one copy of $V_{G_2}(10)$. Suppose $R_u(C^g) \cap Q(2) = 1$. Then $R_u(C^g)$ is isomorphic to $Q / Q(2)$ and hence abelian. Thus $R_u(C^g) Q(2)$ is an abelian 42-dimensional subgroup of $Q$. This implies that $Q$ is abelian, a contradiction. Therefore, $R_u(C^g) \cap Q(2) \neq 1$ and so $R_u(C^g) \cap Q(2) = Q(2)$, proving the claim. We may now consider $R_u(C^g) / Q(2)$. Since the composition factors of $R_u(C^g) / Q(2) \downarrow Y$ are $22 / 00$, we have $H^1(Z_3 Q(2)/Q(2), R_u(C^g) / Q(2)) = 0$ and so $Z_3 Q(2) / Q(2)$ is $R_u(C^g) / Q(2)$-conjugate to $Y Q(2) / Q(2)$. Therefore, $Z_3$ is contained in $Y Q(2)$. It remains to show that there exists a non-standard complement to $Q$ in $YQ$ that is not contained in $Y Q(2)$. Let $\widehat{W}$ be a non-standard $A_2$ complement to $Q/Q(2)$ in $Y Q/ Q(2)$ and let $W$ be such that $Q(2) < W < YQ$ and $\widehat{W} = W / Q(2)$. The definition of $\widehat{W}$ being non-standard means that $\widehat{W}$ is not $Q / Q(2)$-conjugate to $Y Q(2) / Q(2)$. Now, as shown above, we can lift $\widehat{W}$ to a non-standard complement to $Q$ in $Y$. Let $V$ be such a lift, so $W = V Q(2)$. Suppose $V$ is $Q$-conjugate to a subgroup of $Y Q(2)$. Then $V^q < Y Q(2)$ for some $q \in Q$ and so $W^q = (V Q(2))^q = V^q Q(2) = Y Q(2)$. However, $\widehat{W}^q = W^q / Q(2) = Y Q(2) / Q(2)$, contradicting the fact that $\widehat{W}$ was non-standard. 

By the previous paragraph, there exists a non-standard complement to $Q$ in $Y Q$ fixing a semisimple non-zero vector $l$. Let $Z$ be such a complement. Then $C_{E_7}(l)$ contains a maximal torus. We claim that $C_{E_7}(l)^\circ = A_2 A_5$. Indeed, the only centralisers of semisimple elements that contain an $A_2$ with the same composition factors as $Z$ are $A_2 A_5$ and $A_6 T_1$. Suppose $Z < A_6 T_1$. Then $Z < L^g$ for some $g \in E_7$ and we may assume $Z = Y^g$. Therefore $Z < C^g$ and so $Z$ fixes a non-zero nilpotent vector of $L(E_7)$, a contradiction. Hence $Z < A_2 A_5$ and by comparing composition factors we see that $Z$ must be conjugate to $X$. Therefore $X$ is contained in a parabolic subgroup of $E_7$.  
\end{proof}

\subsection{Maximal  $M = G_2 C_3$}

The only subgroups of rank at least $2$ contained in $G_2$ are $G_2$ and $A_2$ (2 classes if $p=3$). However, there are no $C_3$-irreducible subgroups of type $A_2$ and only if $p=2$ is there one conjugacy class of $G_2$ subgroups in $C_3$. 

\begin{lem} \label{g2c3e7} 

The $E_7$-irreducible candidate subgroups contained in $G_2 C_3$ are $G_2 \hookrightarrow G_2 G_2 < G_2 C_3$ via $(10^{[r]},10^{[s]})$ $(rs=0)$ with $p=2$.  

\end{lem}

\begin{proof}
To see the subgroups listed in the statement of the lemma are all $E_7$-irreducible we use Lemma \ref{wrongcomps} on $L(E_7)$. All of the subgroups in the lemma have an 84-dimensional composition factor on $L(E_7)$ (see Table \ref{E7tab}) but no Levi subgroup has a composition factor of dimension 84 or higher, by Table \ref{levie7}. 

The only candidate subgroup we are claiming is not $E_7$-irreducible is the $G_2 \hookrightarrow G_2 G_2$ via $(10,10)$. Call this $X$. To prove $X$ is contained in a parabolic subgroup of $E_7$ we show $V_{56} \downarrow X$ has a trivial submodule. This is sufficient because the full centraliser of a non-zero vector in $V_{56}$ has dimension at least 77, hence is either $E_6$ or contained in a parabolic subgroup. So this implies $X$ is contained in a parabolic because $E_6$ is a Levi subgroup. From \cite[Table 10.2]{LS1} we get the exact structure of $V_{56}$ restricted to $G_2 C_3$: 
$$\mfra{V_{56} \downarrow G_2 C_3 = } \cfrac{(00,100)}{\cfrac{(10,100)+(00,001)}{(00,100)}} \mfra{.}$$
Also, $V_{C_3}(100) \downarrow G_2 = 10$ and $V_{C_3}(001) \downarrow G_2 = 20/00^2$, with a trivial submodule by self-duality. There is a trivial submodule in $V_{G_2}(10) \otimes V_{G_2}(10)$ as well. So the restriction of $(10,100) + (00,001)$ to $X$ has a trivial submodule of dimension at least 2. Since $(00,100)$ restricted to $X$ is $V_{G_2}(10)$ and $\mathrm{Ext}^1_{G_2}(10,00)$ is 1-dimensional, it follows that $(00,100)$ can only block a 1-dimensional trivial module. Therefore there is a trivial submodule for $X$ on $V_{56}$, as required.  
\end{proof}

\subsection{Maximal  $M = A_2$ $(p \geq 5)$}

The maximal $A_2$ is the only $E_7$-irreducible subgroup from this maximal subgroup. 

This completes the proof of Theorem 2. 

\section{Proof of Theorem 3: $E_8$-irreducible subgroups}

We now move on to the proof of Theorem 3, finding the conjugacy classes of simple, connected irreducible subgroups of $E_8$ of rank at least $2$. As before, we use Theorem \ref{max} which lists the reductive, maximal connected subgroups (with no simple $A_1$ factor) of $E_8$. These are $D_8$, $A_8$, $A_2 E_6$, $A_4^2$, $G_2 F_4$ and $B_2$ $(p \geq 5)$. We take each maximal subgroup in turn, finding all $E_8$-irreducible subgroups up to $E_8$-conjugacy.  

\subsection{Maximal  $M  = D_8$} \label{b4notation}

We start by finding all $E_8$-candidate subgroups contained in $D_8$. 

\begin{lem} \label{d8irr}
The simple, connected $D_8$-irreducible subgroups of $D_8$ of rank at least $2$ are listed in the following Table (each $E_8$-conjugacy class is listed exactly once). 
\end{lem} \pagebreak
\setlength\LTleft{0.05\textwidth}

\begin{longtable}{p{0.29\textwidth - 2\tabcolsep}>{\raggedright\arraybackslash}p{0.08\textwidth-2\tabcolsep}>{\raggedright\arraybackslash}p{0.53\textwidth-\tabcolsep}@{}}

\hline \noalign{\smallskip}

Irreducible subgroup & $p$ & Comments \\

\hline \noalign{\smallskip}

$D_8$ & any & \\

\hline \noalign{\smallskip}

$B_7$ & any & Maximal subgroup of $D_8$.  \\

\hline \noalign{\smallskip}

$B_4 (\dagger)$ & any & Maximal with $V_{D_8}(\lambda_1) \downarrow B_4 (\dagger) = W(0001)$ and $V_{D_8}(\lambda_7) \downarrow B_4 (\dagger) = W(1001)$. \\

\hline \noalign{\smallskip}

$B_4 (\ddagger)$ & any & Maximal with $V_{D_8}(\lambda_1) \downarrow B_4 (\ddagger) = W(0001)$ and $V_{D_8}(\lambda_7) \downarrow B_4 (\ddagger) = W(2000) / W(0010)$. \\

\hline \noalign{\smallskip}

$A_3$ & $p >2$ & $A_3 < B_7$ with $V_{D_8}(\lambda_1) \downarrow A_3 = 101 + 000$. \\

\hline \noalign{\smallskip}

$D_4 \hookrightarrow D_4^2$ via: \newline 
$(1000,1000^{[r]})$ $(r \neq 0)$ & \newline any & $D_4^2$ is maximal, see Section \ref{nota} for notation of the diagonal subgroups.   \\
$(1000,0010)$ & any & \\
$(1000,0010^{[r]})$ $(r \neq 0)$ & any & \\
$(1000,1000^{[\iota r]})$ $(r \neq 0)$ & any & \\

\hline \noalign{\smallskip}

 $B_3 \hookrightarrow  B_3^2$ via: \newline $(100,100^{[r]})$ $(r \neq 0)$ & \newline any & $B_3^2 < D_4^2$ with $V_{D_8}(\lambda_1) \downarrow B_3^2 = (100,000)$ $ + $ $(000,001)$ $+$ $(000,000)^2$.  \\

\hline \noalign{\smallskip}

 $A_2 \hookrightarrow A_2^2$ via: \newline $(10,10^{[r]})$ $(r \neq 0)$  & \newline $p \neq 3$ & $A_2^2 < D_4^2$ $(p \neq 3)$ where $V_{D_4}(\lambda_1) \downarrow A_2 = 11$. \\

$(10,01^{[r]})$ $(r \neq 0)$ & $p \neq 3$ &  \\

\hline \noalign{\smallskip}

 $B_2 \hookrightarrow B_2^2 (\dagger)$ via: & & $B_2^2 (\dagger)$ is maximal if $p \neq 2$, while $B_2^2 (\dagger) < B_4 (\dagger)$ if   \\

$(10,10^{[r]})$ $(r \neq 0)$ & any & $p=2$. In both cases $V_{D_8}(\lambda_1) \downarrow B_2^2 (\dagger) = (01,01)$  \\

$(10,02)$ & $p=2$ & and $V_{D_8}(\lambda_7) \downarrow B_2^2 (\dagger) = (01,W(11)) / (W(11),01)$. \\

$(10,02^{[r]})$ $(r \neq 0)$ & $p=2$ & \\

\hline \noalign{\smallskip}

 $B_2 \hookrightarrow B_2^2 (\ddagger)$ via: & & $B_2^2 (\ddagger)$ is maximal if $p \neq 2$, while $B_2^2 (\ddagger) < B_4 (\ddagger)$ if   \\

$(10,10^{[r]})$ $(r \neq 0)$ & any & $p=2$. In both cases $V_{D_8}(\lambda_1) \downarrow B_2^2 (\ddagger) = (01,01)$ \\

$(10,02)$ & $p=2$ & and $V_{D_8}(\lambda_7) \downarrow B_2^2 (\ddagger) = (W(20),00) / (00,W(20)) /$    \\

$(10,02^{[r]})$ $(r \neq 0)$ & $p=2$ & $ (W(10),W(02)) / (W(02),W(10))$. \\

\hline \noalign{\smallskip}

 $B_2 \hookrightarrow B_2 B_2$ via: & & $B_2 B_2 < A_3 D_5$ with $V_{D_8}(\lambda_1) \downarrow B_2 B_2 = (10,00)$ $+$  \\

 $(10,10)$ & $p > 2$ & $(00,02) + (00,00)$. \\

 $(10^{[r]}, 10^{[s]})$ $(rs = 0)$ & $p > 2$ & \\

\hline \noalign{\smallskip}

$B_2 \hookrightarrow B_2^3$ via: & & $B_2^3 < B_7, A_3 D_5$ with $V_{D_8}(\lambda_1) \downarrow B_2^3 = (10,00,00)$ $+$  \\

 $(10,10^{[r]},10^{[s]})$ $(0<r<s)$ & $p > 2$ & $(00,10,00) + (00,00,10) + (00,00,00)$. \\

\hline

\end{longtable}

\begin{proof}
This is a mainly routine task of using Lemma \ref{class} and the tables in \cite{lubeck} to calculate the possibilities for $V_{D_8}(\lambda_1) \downarrow X$, with $X$ irreducible. We note a few technical details. Firstly, there are two conjugacy classes of $B_4$ in $D_8$ embedded via $V_{B_4}(0001)$ that are interchanged by the graph automorphism of $D_8$. We distinguish between them by using $B_4 (\dagger)$ and $B_4 (\ddagger)$, which have different composition factors on $V_{D_8}(\lambda_7)$ (as given in \cite[Prop. 2.12]{LS3}). Similarly there are two conjugacy classes of $B_2^2$ embedded via $V_{B_2}(01) \otimes V_{B_2}(01)$ and we use $B_2^2 (\dagger)$ and $B_2^2 (\ddagger)$ to represent the two classes. If $p=2$ then $B_2^2 (\dagger)$ is contained in $B_4 (\dagger)$ and $B_2^2 (\ddagger)$ is contained in $B_4 (\ddagger)$. Indeed, the action of the $B_2^2$ subgroup of $B_4$ on $V_{B_4}(0001)$ is $01 \otimes 01$. There are only two conjugacy classes in $D_8$ of stabilisers of the tensor product decomposition of $V_{16} = V_{4} \otimes V_{4}$, namely $B_2^2 (\dagger)$ and $B_2^2 (\ddagger)$. By checking the composition factors on $L(E_8)$ we see that the $B_2^2$ subgroup of $B_4 (\dagger)$ is conjugate to $B_2^2 (\dagger)$ and similarly $B_2^2 (\ddagger) < B_4 (\ddagger)$. 

For $p > 2$, the module $V_{A_3}(101)$ is 15-dimensional and self-dual, therefore $A_3$ embeds into $B_7$ and is $D_8$-irreducible. When $p=2$, the module $V_{A_3}(101)$ is 14-dimensional and \cite[Table 1]{se1} shows $A_3$ embeds into $D_7$. Therefore there is no $D_8$-irreducible $A_3$ when $p=2$.  

Finally, consider $D_4^2$. The $E_8$-conjugacy classes of diagonal $D_4$ subgroups are found in \cite[p. 59]{LS3} and only those conjugacy classes of $D_8$-irreducible subgroups are given in the conclusion of the lemma.     
\end{proof}

Now we take each candidate subgroup in Lemma \ref{d8irr} which is proper in $M$ and either show it is $E_8$-irreducible or prove it lies in a parabolic subgroup of $E_8$. 

\begin{lem}
The candidate subgroups $B_7$ and $A_3$ $(p \neq 2)$ are $E_8$-irreducible.
\end{lem}

\begin{proof}
The $B_7$ is $E_8$-irreducible since no Levi subgroup has a subgroup of type $B_7$ (or $C_7$ if $p=2$). The composition factors in Table \ref{E8tab} show Corollary \ref{notrivs} applies to prove that $A_3$ $(p \neq 2)$ is $E_8$-irreducible.
\end{proof}

The following technical lemma is required to prove $B_4 (\ddagger)$ is not $E_8$-irreducible when $p=2$.  

\begin{lem} \label{b4hastriv}

Suppose $p=2$, $X \cong B_4$ and $L(E_8) \downarrow X = 2000/1000^4/0100^4/0010^2/0000^8$. Then $L(E_8) \downarrow X$ has a trivial submodule. 

\end{lem}

\begin{proof}
We use the proof of \cite[Lemma 7.2.3]{LS1}. In the hypothesis of the lemma it is assumed that there is a subgroup isomorphic to $B_4$ with the same composition factors on $L(E_8)$ as $X$. Let $C = C_{L(E_8)}(L(X))$. If $C = 0$ then the proof of \cite[Lemma 7.2.3]{LS1} shows that $L(E_8) \downarrow X$ has a trivial submodule. So we may assume $C \neq 0$ and that $X$ has no trivial submodule on $L(E_8)$. The composition factors of $C$ are among the totally twisted composition factors of $L(E_8) \downarrow X$ so $C$ must have $V_{X}(2000)$ as a submodule (we are assuming no trivial submodules of $L(E_8)$ and therefore of $C$). This submodule of $C$ must be a submodule of $L(E_8)$, but $L(E_8)$ is self-dual and has only one composition factor isomorphic to $V_{X}(2000)$. Therefore $C$ splits off as a direct summand of $L(E_8)$. Hence $L(E_8) \downarrow X = 2000 \oplus M_1$ where $M_1$ has composition factors $1000^4/0100^4/0010^2/0000^8$. We claim that $M_1$ has a trivial submodule. To prove the claim we use Lemma \ref{litlem}. The conditions of the lemma hold because only $V_{X}(1000)$ and $V_{X}(0100)$ extend the trivial module, as shown in Lemma \ref{b4cohom}(i)-(iii). Because $M_1$ is self-dual it therefore has a trivial submodule. This is a contradiction, completing the proof.  
\end{proof}

\begin{lem} \label{badB4}
Suppose $p=2$. Let $B_4 (\ddagger) < D_8$ be as in Lemma \ref{d8irr}. Then $B_4(\ddagger)$ is contained in an $A_7$-parabolic of $E_8$. 
\end{lem}

\begin{proof}
We use the same method as in the proof of Lemma \ref{badA2}. We find a $B_4$ in an $A_7$-parabolic and show it is contained in $D_8$ and conjugate to $B_4(\ddagger)$.

Let $P = Q A_7 T_1$ be a standard $A_7$-parabolic. The centraliser of the standard graph automorphism of $A_7$ is a $C_4$. So $C_4$ acts on $Q$ and we may look for complements to $Q$ in $Q C_4$. However, as $p=2$, the group $B_4$ can also act on $Q$. This is due to the special isogeny from $B_4$ to $C_4$. We want to find a $B_4$ complement to $Q$ in $Q C_4$ and need to use the same tools as in the  proof of Lemma \ref{badA2}. We first find the structure of the levels of $Q$ under the action of $B_4$. Using \cite[Theorem 2]{ABS} (and the notation defined in Section \ref{nota}) we find $Q$ has three levels and that under the action of $L' = A_7$ they have the following $A_7$-module structure. 
\begin{align*}
Q/Q(2) \downarrow A_7 &= \lambda_5 \\
Q(2)/Q(3) \downarrow A_7 &= \lambda_2 \\
Q(3) \downarrow A_7 &= \lambda_7
\end{align*}
Now under the action of $B_4$ we find the levels of $Q$ have the following structure. 
\begin{align*}
Q/Q(2) \downarrow B_4 &= \lambda_1 \oplus \lambda_3 \\
Q(2)/Q(3) \downarrow B_4 &= 0|\lambda_2|0 \\
Q(3) \downarrow B_4 &= \lambda_1
\end{align*}
By Lemma \ref{b4cohom}, $H^1(B_4,\lambda_1)$ is 1-dimensional and $H^1(B_4,0 | \lambda_2 | 0)$, $H^1(B_4,\lambda_3)$ are both 0. Also, $H^2(B_4,0 | \lambda_2 | 0) = H^2(B_4,\lambda_1) = 0$. Therefore \cite[Lemma 3.2.11]{dav} applies. This shows that every $B_4$ complement to $Q / Q(i+1)$ in $C_4 Q / Q(i+1)$ lifts to a $B_4$ complement to $Q$ in $Q C_4$ for $i = 1$, $2$, $3$. 

We know that any $B_4$ complement to $Q$ has the composition factors $L(E_8) \downarrow B_4 = 2000/$ $\! 1000^4 /$ $\! 0100^4/$ $\! 0010^2 /$ $\! 0000^8$ (these are images of the composition factors of the $C_4 < A_7$ under the special isogeny). Therefore, Lemma \ref{b4hastriv} shows that any $B_4$ complement to $Q$ has a trivial submodule on $L(E_8)$. Let $0 \neq l \in L(E_8)$ be one such fixed vector. Lemma \ref{fix} shows $C_{E_8}(l)$ either contains a maximal torus of $E_8$ or has non-trivial unipotent radical. 

Assume we have a non-standard $B_4$ complement to $Q$, call it $X$, which fixes $l_1$, a non-zero nilpotent vector of $L(E_8)$. Then using \cite[Table 22.1.1]{LS9}, which gives all possible centralisers in $E_8$ of nilpotent elements of $L(E_8)$, and an argument analogous to that given in the proof of Lemma \ref{badA2}, we conclude that $C_{E_8}(l_1)$ is in fact $P$-conjugate to $C$ where $C$ is the centraliser of an element of the nilpotent class $A_1^4$. Moreover, $C / R_u(C) = C_4$, the dimension of $R_u(C)$ is 84 and $C < P$ with $R_u(C) < Q$. 

It similarly follows as in the proof of Lemma \ref{badA2} that any non-standard $B_4$ complement contained in a $P$-conjugate of $C$ is contained in $Q(3) C_4$ and that there exists a non-standard $B_4$ complement to $Q$ in $Q C_4$ that is not contained in any $P$-conjugate of $C$. Let $Y$ be such a complement. Then $Y$ fixes a non-zero semisimple element $l_2$ and $D := C_{E_7}(l_2)$ contains a maximal torus. It follows, from considering the composition factors of $Y$, that $D = A_7 T_1, A_8$ or $D_8$. We rule out the first possibility as $A_7 T_1$ does not contain a $B_4$ subgroup. By Theorem \ref{max}, we know that $A_8$ does not fix a non-zero vector on $L(E_8)$ when $p=2$ and so $D = D_8$. Therefore $Y$ is contained in $D_8$ and has the same composition factors as $B_4(\ddagger)$ by construction. To prove $Y$ is conjugate to $B_4(\ddagger)$ we need to show it is not conjugate to any other $B_4$ subgroup in $D_8$. The only other possibility for a $B_4$ subgroup in $D_8$ with the same composition factors as $Y$ is a $B_4$ in the $A_7$-parabolic of $D_8$ (whose Levi factor is a Levi subgroup of $E_8$). Let this parabolic be $Q_1 A_7 T_1 < D_8$. Then $Q_1 \downarrow A_7 = \lambda_2$. Therefore under the action of $B_4$ the structure of $Q_1$ is  $Q_1 \downarrow B_4 = 0|\lambda_2|0$. Since $H^1(B_4, 0 | \lambda_2 | 0) = 0$ there are no $B_4$ complements to $Q_1$ in $Q_1 C_4$ in such an $A_7$-parabolic of $D_8$. Hence $Y$ is conjugate to $B_4(\ddagger)$ and $B_4(\ddagger)$ is contained in an $A_7$-parabolic of $E_8$. 
\end{proof}

\begin{lem}
The candidate subgroups in Lemma \ref{d8irr} isomorphic to $B_4$ are $E_8$-irreducible, except for $B_4(\ddagger)$ when $p=2$. 
\end{lem}

\begin{proof}
If $p \neq 2$ then the results of \cite[p. 97]{LS3} apply, showing $B_4 (\dagger)$ and $B_4 (\ddagger)$ are $E_8$-irreducible.  

Now suppose $p=2$. Lemma \ref{badB4} shows $B_4 (\ddagger)$ is contained in an $A_7$-parabolic. We claim $B_4 (\dagger)$ is $E_8$-irreducible. We have $L(E_8) \downarrow B_4 (\dagger) = 0100^2 / 1000^2 / 0010 / 1001 / 0000^4$. The module $V_{B_4}(1001)$ is 128-dimensional. Therefore the only possible Levi subgroup that could contain an irreducible $B_4$ with the same composition factors is $E_7$. But $E_7$ has no irreducible subgroups of type $B_4$ by Theorem 2. Therefore $B_4 (\dagger)$ is $E_8$-irreducible.
\end{proof}

\begin{lem} \label{e8a2d4}
All of the candidate subgroups in Lemma \ref{d8irr} isomorphic to $A_2$ are $E_8$-irreducible. 
\end{lem}

\begin{proof}
If $p>3$ then Corollary \ref{notrivs} shows they are all $E_8$-irreducible. When $p=3$ there are no $A_2$ candidate subgroups from Lemma \ref{d8irr}, so let $p=2$ and $X$ be a candidate $A_2$. Then $L(E_8) \downarrow X = 11/ 11^{[r]} / 30 / 03 / 30^{[r]} / 03^{[r]} / (11 \otimes 11^{[r]})^3 / 00^4$ with $r \neq 0$. As before, we use Lemma \ref{wrongcomps}. By considering the number of trivial composition factors, the only possible Levi subgroups containing an irreducible $A_2$ with the same composition factors as $X$ are $L' = E_7, D_7, A_7$ or $A_6$. We can rule out all but $D_7$ because $X$ has three 64-dimensional composition factors. There is no $D_7$-irreducible $A_2$ that shares the composition factors of $X$ because there is no way of making a 14-dimensional self-dual module out of the composition factors of $X$ (let alone two: $V_{D_7}(\lambda_1)$ occurs twice in $L(E_8) \downarrow D_7$). Therefore $X$ is $E_8$-irreducible as claimed.
\end{proof}

\begin{lem}
All of the candidate subgroups in Lemma \ref{d8irr} isomorphic to $D_4$ or $B_3$ are $E_8$-irreducible. 
\end{lem}

\begin{proof}
If $p \neq 2$ then Corollary \ref{notrivs} shows each $D_4$ and $B_3$ candidate subgroup is $E_8$-irreducible. So suppose $p=2$. First consider candidate subgroups isomorphic to $D_4$. If the embedding of $D_4$ has a non-zero field twist then it contains an $E_8$-irreducible $A_2$ (by Lemma \ref{e8a2d4}) and must be $E_8$-irreducible itself. So consider $X = D_4 \hookrightarrow D_4^2$ via $(1000,0010)$. We will use Lemma \ref{wrongcomps} and therefore need to do the usual analysis for $X$, where $L(E_8) \downarrow X = 1000^2 / 0100^2 / 0010^2 / 0001^2 / 1010 / 1001 / 0011 / 0000^4$, from Table \ref{E8tab}. Using Table \ref{levie8}, the only possible Levi subgroups containing a $D_4$ with the same composition factors as $X$ are $L' = E_7, D_7$ or $A_7$. We use Theorem 2 to rule out $E_7$, as it has no irreducible subgroup $D_4$ when $p=2$. There are no $D_7$-irreducible $D_4$ subgroups in $D_7$. That leaves just $L' = A_7$. If $D \cong D_4 < A_7$ then $V_{A_7}(\lambda_1) \downarrow D = V_{D_4}(\lambda_1)$. But $L(E_8) \downarrow D$ has $V_{D_4}(2 \lambda_1)$ as a composition factor, which means $D$ does not have the same composition factors as $X$. Hence $X$ is $E_8$-irreducible.    

Now consider (still with $p=2$) a candidate $B_3$ as in Lemma \ref{d8irr}, call it $Y$. Then $$L(E_8) \downarrow Y = 010 / 100^2 / (100^{[r]})^2 / 010^{[r]} / 001^2 / (001^{[r]})^2 / 100 \otimes 001^{[r]} / 001 \otimes 100^{[r]} / 001 \otimes 001^{[r]} / 000^4$$ with $r \neq 0$ (from Table \ref{E8tab}). As there are four trivial composition factors, the only possible Levi subgroups containing an irreducible $B_3$ with the same composition factors are $L' = E_7, D_7, A_7$ and $A_6$. We know from Theorem 2 that $E_7$ has no irreducible $B_3$ subgroups. Both $A_6$ and $A_7$ have no composition factors of dimension at least 64 so they are ruled out. Therefore $L' = D_7$ is the only remaining possibility. Suppose $B_3 \cong Z < D_7$ has the same composition factors as $Y$ on $L(E_8)$. Since $V_{D_7}(\lambda_6)^* = V_{D_7}(\lambda_7)$ it follows that the only 64-dimensional composition factor of $L(E_8) \downarrow Y$, namely $001 \otimes 001^{[r]}$, is contained in $V_{D_7}(\lambda_2) \downarrow Z$. But then $V_{D_7}(\lambda_6) \downarrow Z$ must contain $100 \otimes 001^{[r]}$ and $V_{D_7}(\lambda_7) \downarrow Z$ must contain $001 \otimes 100^{[r]}$ (or vice versa). This implies that $V_{D_7}(\lambda_6) \downarrow Z$ is not the dual of $V_{D_7}(\lambda_7) \downarrow Z$, which is a contradiction. So no such $Z$ exists and $Y$ is $E_8$-irreducible. 
\end{proof}

\begin{lem} \label{b2ind8}
Suppose $X$ is a candidate subgroup in Lemma \ref{d8irr} isomorphic to $B_2$. Then $X$ is $E_8$-irreducible unless $p=2$ and $X$ is contained in $B_2^2 (\ddagger) < D_8$. 
\end{lem}

\begin{proof}
If $p \neq 2$ then Corollary \ref{notrivs} is enough for most of the candidate $B_2$ subgroups. There are two exceptions, both occurring when $p=5$, namely $X = B_2 \hookrightarrow B_2 B_2$ via $(10,10)$ and $Y = B_2 \hookrightarrow B_2^2 (\ddagger)$ via $(10,10^{[r]})$ ($r \neq 0)$. We show both are $E_8$-irreducible using Lemma \ref{wrongcomps}. From Table \ref{E8tab}, $L(E_8) \downarrow X = 02^6 / 10^4 / 20^2 / 12^4 /00^2$. There are only two trivial composition factors which means only $L' = D_7$ or $A_7$ could possibly contain a $B_2$ with the same composition factors as $X$. There is no way of making an 8-dimensional module from the composition factors of $L(E_8) \downarrow X$ which rules out $A_7$. Similarly there is no way to make two copies of the same 14-dimensional module without using both trivial composition factors, but $L(E_8) \downarrow D_7 = \lambda_2 / \lambda_1^2 / \lambda_6 / \lambda_7 / 0$, from Table \ref{levie8}. So $V_{D_7}(\lambda_1)^2$ cannot contain any trivial composition factors in a restriction to a $B_2$ sharing the composition factors of $X$. This proves $X$ is $E_8$-irreducible.

Now consider $Y$, still with $p=5$. Again, from Table \ref{E8tab}, $$L(E_8) \downarrow Y = 02 / 02^{[r]} / 20 / 20^{[r]} / (10 \otimes 02^{[r]})^2 / (10^{[r]} \otimes 02)^2 / 00^2.$$ Therefore the only Levi subgroup that could contain an $L'$-irreducible $B_2$ with the same composition factors as $Y$ is $D_7$. But $L(E_8) \downarrow Y$ has four 50-dimensional composition factors which, using Table \ref{levie8}, rules out $D_7$. Hence $Y$ is $E_8$-irreducible.  

If $p=2$ then the only candidate subgroups are diagonal subgroups of $B_2^2 (\dagger)$ and $B_2^2 (\ddagger)$. Any subgroup of $B_2^2 (\ddagger)$ will not be $E_8$-irreducible. This is because $B_2^2 (\ddagger) < B_4 (\ddagger)$ and Lemma \ref{badB4} shows $B_4 (\ddagger)$ is contained in a parabolic subgroup of $E_8$ when $p=2$. Now we consider diagonal subgroups of $B_2^2 (\dagger)$ with $p=2$. 

First, let $Z_1 = B_2 \hookrightarrow B_2^2 (\dagger)$ via $(10,10^{[r]})$ $(r \neq 0)$. Then $$L(E_8) \downarrow Z_1 = 02^2 / (02^{[r]})^2 / 10^4 / (10^{[r]})^4 / 10 \otimes 02^{[r]} / (10 \otimes 10^{[r]})^2 / 10^{[r]} \otimes 02 / 01 \otimes 11^{[r]} / 01^{[r]} \otimes 11 / 00^8.$$ Since there are two 64-dimensional composition factors it follows (from Table \ref{levie8}) that the only Levi subgroups that could contain an irreducible $B_2$ with these composition factors are $E_7$ and $D_7$. Theorem 2 rules out $E_7$, so assume there is a $B_2$ contained $D_7$-irreducibly in $D_7$ with the same composition factors as $Z_1$. The 14-dimensional $V_{D_7}(\lambda_1)$ occurs in $L(E_8) \downarrow D_7$ and so $V_{D_7}(\lambda_1) \downarrow B_2$ includes two trivial composition factors because that is the only way to make a 14-dimensional module from the composition factors of $L(E_8) \downarrow Z_1$. The only possibility for a $D_7$-irreducible $B_2$ with two trivial composition factors is one contained irreducibly in $B_6$. From \cite[Table 8.1]{LS3}, $L(E_8) \downarrow B_6 = \lambda_1^4 / \lambda_2 / \lambda_6^2 / 0^8$. This means the two, non-isomorphic, 64-dimensional composition factors of $L(E_8) \downarrow Z_1$ are contained in $V_{B_6}(\lambda_6)^2$ (note $V_{B_6}(\lambda_2)$ is 64-dimensional but  on restricting to $B_2$ it cannot be equal to $01 \otimes 11^{[r]}$ or $01^{[r]} \otimes 11$). Clearly this cannot happen, which proves $Z_1$ is $E_8$-irreducible. 

The same argument works for $B_2 \hookrightarrow B_2^2 (\dagger)$ via $(10,02^{[r]})$ $(r \neq 0)$. It remains to consider $Z_2 = B_2 \hookrightarrow B_2^2 (\dagger)$ via $(10,02)$. From Table \ref{E8tab}, $$L(E_8) \downarrow Z_2 = 10^4 / 01^4 / 20^4 / 02^{6} / 21 / 12^2 / 30 / 03 / 13 / 04 / 00^{12}.$$ Using Table \ref{levie8} (noting we have one 64-dimensional and 12 trivial composition factors), the possible Levi subgroups that could contain an irreducible $B_2$ with the same composition factors as $Z_2$ are $D_7, E_7$ and $E_6$. Theorems 1 and 2 rule out $E_6$ and $E_7$. As in the previous argument, it follows that such an irreducible $B_2$ contained in $D_7$ is contained in $B_6$. The one 64-dimensional composition factor, $V_{B_2}(13)$, is then equal to $V_{B_6}(\lambda_2) \downarrow B_2$. But there is no possible $B_2$ in $B_6$ that has $V_{B_6}(\lambda_2) \downarrow B_2 = V_{B_2}(13)$. Hence $Z_2$ is $E_8$-irreducible.    
\end{proof}

\subsection{Maximal  $M = A_8$}

We need to find all simple subgroups, of rank at least $2$, that have an irreducible 9-dimensional module. By \cite{lubeck} we see that these are $A_8$, $B_4$ $(p \neq 2)$ embedded via $V_{B_4}(1000)$ and $A_2 < A_2^2$ where $V_{A_8} \downarrow A_2^2 = (10,10)$. It is shown in \cite[p. 58]{LS3} that $B_4 < A_8$ is conjugate to $B_4(\ddagger) < D_8$. Similarly, if $p \neq 3$, then $A_2^2$ is conjugate to $A_2^2 < D_4^2 < D_8$ (each factor $A_2$ irreducibly embedded) \cite[pp. 66-67]{LS3}. The following lemma shows that when $p=3$, the subgroup $A_2^2$ is contained in a parabolic subgroup of $E_8$. Therefore all diagonal subgroups of $A_2^2$ are contained in a parabolic subgroup of $E_8$. The proof is different in flavour from all of the arguments so far. We use finite subgroups and computations in Magma \cite{magma} to show $A_2^2$ is contained in a $D_7$-parabolic.   

\begin{lem} \label{bada2a2}
Let $p=3$. The subgroup $X = A_2^2 < A_8$, embedded via $V_{A_8}(\lambda_1) \downarrow A_2^2 = (10,10)$, is contained in a $D_7$-parabolic of $E_8$. 
\end{lem}

\begin{proof}
Lemma \ref{wrongcomps}, along with Table \ref{levie8} shows the only parabolic subgroup $X$ can be contained in is a $D_7$-parabolic subgroup. To prove $X$ is contained in a parabolic subgroup we first use Lemma \ref{subspaces} to show that the finite group $S = A_2(3) \times A_2(3) < X$ fixes the same subspaces as $X$ on $L(E_8)$. To show Lemma \ref{subspaces} applies we have to check the three conditions. We have $L(E_8) \downarrow X = (11,11)^3 / (11,00)^6 / (00,11)^6 / (30,00) / (00,30) / (03,00) / (00,03) / (00,00)^5$ and hence conditions (i) and (iii) hold. Condition (ii) holds for all pairs of composition factors for which there are no non-trivial extensions between them. Using the K\"{u}nneth formula \cite[10.85]{rot} we see that the only pairs of composition factors that have a non-trivial extension between them are $\{M_1,M_2\} = \{(11,11),(11,00)\},$ $\{(11,00),(00,00)\}$ and $\{(30,00),(11,00)\}$ up to duals and swaps. For all but the last pair, \cite[Theorem 7.4]{cps} shows immediately that $\mathrm{Ext}^1_X(M_1,M_2) \rightarrow \mathrm{Ext}_S^1(M_1,M_2)$ is injective. To show the restriction map $\mathrm{Ext}^1_X((30,00),(11,00)) \rightarrow \mathrm{Ext}^1_S((10,00),(11,00))$ is injective it suffices to show $\mathrm{Ext}^1_{A_2}(30,11) \rightarrow \mathrm{Ext}^1_{A_2(3)}(10,11)$ is injective. We know $\mathrm{Ext}^1_{A_2}(30,11)$ is 1-dimensional and the tilting module $T(30) = 11|(30 + 00)|11$ is indecomposable. We construct $T(30)$ as a direct summand of $10 \otimes 10 \otimes 10$. Therefore, if we show a direct summand of $10 \otimes 10 \otimes 10$ for $A_2(3)$ contains a non-trivial extension of $10$ by $11$, the restriction map must be injective. This last check is easily done using Magma \cite{magma}.

We now show that $S$ fixes a 14-dimensional abelian subalgebra of $L(E_8)$ that is ad-nilpotent of exponent 3 i.e. $(\text{ad } a)^3 = 0$ for all $a$.  To do this we construct $S$ as a normal subgroup of the maximal subgroup $(SL(3,3) \otimes SL(3,3)).2$ in $SL(9,3)$. Doing this in Magma gives 11 generators for $S$, as $9 \times 9$ matrices over GF(3). We then write these 11 generators as words in the Magma generators of $A_8$. Finally, we write the words of $A_8$ generators in terms of the generators of $E_8$. This allows us to use Magma to find all 14-dimensional $S$-submodules of $L(E_8)$. There is a unique such $S$-submodule that is an abelian subalgebra, and it is ad-nilpotent of exponent 3. 

So $S$ and therefore $X$ fixes a 14-dimensional abelian subalgebra of $L(E_8)$ that is ad-nilpotent of exponent 3. Exponentiating this subalgebra yields a 14-dimensional unipotent subgroup of $E_8$, normalised by $X$. Therefore $X$ is contained in a parabolic subgroup of $E_8$, as required. 
\end{proof}

\subsection{Maximal  $M = A_2 E_6$} \label{maxa2e6}

From Theorem 1, we have all of the $E_6$-irreducible subgroups of type $A_2$. They are diagonal subgroups of $A_2^3$, diagonal subgroups of $A_2 \tilde{A}_2 < A_2 G_2$ when $p=3$ or conjugates of $A_2 \cong Y$, where $Y$ acts on $V_{27}$ as $W(22)$ with $p \neq 2$. Therefore all candidate subgroups contained in $A_2 E_6$ are diagonal subgroups of $A_2^4$, $\bar{A}_2 A_2 \tilde{A}_2$ or $\bar{A}_2 Y$. We consider these separately in the next three lemmas. 

We fix the composition factors of $L(E_8) \downarrow A_2^4$ as follows (which is consistent with the restriction of $L(E_8)$ to $A_2 E_6$ and of $V_{27}$ and $L(E_6)$ to $A_2^3$ in Theorem \ref{max}):  \begin{align*} L(E_8) \downarrow A_2^4 =& (W(11),00,00,00)/ (00,W(11),00,00) / (00,00,W(11),00) / (00,00,00,W(11)) /\\ & (00,10,10,10) / (00,01,01,01) / (10,01,10,00) / (10,00,01,10) / (10,10,00,01) /\\ & (01,10,01,00) / (01,00,10,01) / (01,01,00,10).\\ \end{align*}  
\vspace{-1.4cm}
\begin{lem} \label{A2E8}
Suppose $X \cong A_2$ is a diagonal subgroup of $A_2^4$ (with non-trivial projection to each simple factor). Then $X$ is $E_8$-irreducible unless $X$ is conjugate to $A_2 \hookrightarrow A_2^4$ via $(10,10,10,10)$, $(10^{[r]},10,10,10)$ or $(10,10^{[r]},10^{[r]},10^{[r]})$. 
\end{lem}

\begin{proof}
Firstly, note that  $A_2 \hookrightarrow A_2^4$ via $(10,10,10,10)$, $(10^{[r]},10,10,10)$ or $(10,10^{[r]},10^{[r]},10^{[r]})$ is contained in a parabolic subgroup of $A_2 E_6$. This follows from Theorem 1.  It remains for us to show all other diagonal subgroups are indeed $E_8$-irreducible. The conjugacy classes are determined in \cite[pp. 64-65]{LS3}. If $p \neq 3$, we apply Corollary \ref{notrivs} (the composition factors are listed in Table \ref{E8tab}). Now let $p=3$. Then any $X \hookrightarrow A_2^4$ (except those we already know to be $E_8$-reducible) has four trivial composition factors on $L(E_8)$. Therefore the possible Levi subgroups that can contain an irreducible $A_2$ subgroup sharing the same composition factors are $E_7$, $D_7$, $A_6$ and $D_4 A_2$ (using Table \ref{levie8}, Lemma \ref{class} and Theorem 2). We use Table \ref{E7tab} to see that any $E_7$-irreducible $A_2$ does not have the same composition factors as $X$. The only $D_7$-irreducible $A_2$ subgroup is embedded via $V_{A_2}(11) + V_{A_2}(11)^{[r]}$ $(r \neq 0)$. But this does not have the same composition factors as $X$ either. Similarly, the only $A_6$-irreducible $A_2$ is embedded via $V_{A_2}(11)$ and this has six 7-dimensional composition factors, which is more than $X$ has. Finally, the only $D_4$-irreducible $A_2$ subgroup is embedded via $V_{A_2}(11) + 00$. Therefore any $D_4 A_2$-irreducible $A_2$ subgroup of $D_4 A_2$ has nine trivial composition factors (using the restriction $L(E_8) \downarrow D_4 A_2$ in Table \ref{levie8}), which is more than $X$ has. Hence $X$ is $E_8$-irreducible.           
\end{proof}

\begin{lem} \label{bara2a2tildea2}
Suppose $X \cong A_2$ is a diagonal subgroup of $\bar{A}_2 A_2 \tilde{A}_2 < \bar{A}_2 A_2 G_2 < \bar{A}_2 E_6$ $(p=3)$. Then $X$ is $E_8$-irreducible if and only if it is conjugate to one of the subgroups in Table \ref{E8tab}.
\end{lem}

\begin{proof}
First we must find the conjugacy classes of $A_2 \hookrightarrow \bar{A}_2 A_2 \tilde{A}_2$. We have $C_{E_8}(G_2) = F_4$, hence $\bar{A}_2 A_2 < F_4$. It follows that there is an involution acting on $\bar{A}_2 A_2$, inducing a graph automorphism on both factors. We also have that $\tilde{A}_2.2 < G_2$. Therefore all of the conjugacy classes can be represented by $(10^{[a]},10^{[b]},10^{[c]})$ or $(10^{[d]},01^{[e]},10^{[f]})$ (with $a,b,c$ and $d,e,f$ not necessarily distinct for the moment). For $X$ to be $E_8$-irreducible it must be both $A_2 E_6$-irreducible and $G_2 F_4$-irreducible. To satisfy the first of these conditions we must have $b \neq c$ and $e \neq f$ (from Theorem 1). To satisfy the second condition we must have $a \neq b$ (there is only one $F_4$-reducible $A_2$ contained diagonally in $A_2 A_2$ when $p=3$, and by considering composition factors we see it forces $a \neq b$ rather than $d \neq e$). Therefore, the subgroups in Table \ref{E8tab} represent all of the conjugacy classes of candidate subgroups which might be $E_8$-irreducible. By considering their composition factors on $L(E_8)$ we see there are no more conjugacies between them. It remains to show they are all $E_8$-irreducible. 

Suppose $X$ is in one of the conjugacy classes of $A_2$ subgroups in Table \ref{E8tab}. Then $L(E_8) \downarrow X$ has only three trivial composition factors. Therefore among Levi subgroups, only $E_7$ and $D_7$ can contain an irreducible $A_2$ subgroup having the same composition factors as $X$ (the proof of the previous lemma shows an irreducible $A_2$ subgroup of $A_2 D_4$ has more than three trivial composition factors). Theorem 2 rules out $E_7$. Any $D_7$-irreducible $A_2$ does not have the same composition factors as $X$ (in fact it has the same composition factors as $A_2 \hookrightarrow \bar{A}_2 A_2 \tilde{A}_2$ via $(10,10,10^{[r]})$ $(r \neq 0)$ which was $G_2 F_4$-reducible). Hence $X$ is $E_8$-irreducible.   
\end{proof}

\begin{lem} \label{bara2y}

Suppose $X \cong A_2$ is a diagonal subgroup of $\bar{A}_2 Y < \bar{A}_2 E_6$ (where $Y \cong A_2$ acts on $V_{27}$ via $W_{A_2}(22)$, $p>2$). Then $X$ is $E_8$-irreducible.

\end{lem}

\begin{proof}
The conjugacy classes in Table \ref{E8tab} follow from \cite[p. 67]{LS3} (noting that even when $p=3$, $E_6$ contains $Y.2$ because $Y < G_2$ and $Y.2 < G_2$). If $p>3$, the $E_8$-irreducibility follows from Corollary \ref{notrivs}. So suppose $p=3$ and let $X \hookrightarrow \bar{A}_2 Y$. Then $L(E_8) \downarrow X$ has four trivial composition factors. Therefore the possible Levi subgroups that could contain an irreducible $A_2$ having the same composition factors as $X$ are $E_7$, $D_7$, $A_6$ and $D_4 A_2$ (from the proof of Lemma \ref{A2E8}). We may rule all of these out, using the same ideas as in the proof of Lemma \ref{A2E8}, by considering the irreducible $A_2$ subgroups in each Levi subgroup and noting that they do not have the same composition factors as $X$ on $L(E_8)$.  
\end{proof}

Using the composition factors listed in Table \ref{E8tab} we see that there are no conjugacies between any of the $E_8$-irreducible subgroups in the three different overgroups $A_2^4$, $\bar{A}_2 A_2 \tilde{A}_2$ and $\bar{A}_2 Y$.   

\subsection{Maximal  $M = A_4^2$}

By considering which simple groups have an irreducible 5-dimensional module we see that the only $A_4$-irreducible simple subgroups (of rank at least $2$) are $A_4$ and $B_2$ $(p \neq 2)$. So the candidate subgroups contained in $A_4^2$ are diagonal subgroups of type $A_4$ and $B_2$ $(p \neq 2)$. There is no prime restriction on $A_4$ subgroups in $E_8$ in \cite{LS3} so we immediately use \cite[Table 8.1]{LS3} to see that all $A_4$ diagonal subgroups (with a non-trivial projection to both $A_4$ factors) are $E_8$-irreducible, and the conjugacy classes of such subgroups are as in Table \ref{E8tab}. For the $B_2$ subgroups we note that $B_2^2 < A_4^2$ is actually conjugate to $B_2^2 (\ddagger) < D_8$ $(p \neq 2)$, as is shown in \cite[p. 63]{LS3}. Therefore, we have already considered them in Lemma \ref{b2ind8}.    

\subsection{Maximal $M = G_2 F_4$}

The only possible candidate subgroups contained in $G_2 F_4$ are of type $A_2$ and $G_2$. Theorem \ref{simplef4} lists all $F_4$-irreducible subgroups of type $A_2$ and $G_2$. All such subgroups of type $A_2$ are contained in $\bar{A}_2 A_2$ in $F_4$. Therefore any $A_2$ candidate subgroup contained in $G_2 F_4$ will be conjugate to one in $\bar{A}_2 C_{E_8}(\bar{A}_2) = \bar{A}_2 E_6$. Therefore we have already considered them in Section \ref{maxa2e6}.

There is only one $F_4$-irreducible subgroup of type $G_2$, namely the maximal $G_2$ when $p=7$. Therefore we finish this section by proving the following lemma. 

\begin{lem} \label{badG2}
Suppose $X \cong G_2$ is a candidate subgroup contained in $G_2 F_4$. Then $X$ is conjugate to $G_2 \hookrightarrow G_2 G_2 < G_2 F_4$ $(p=7)$ via $(10^{[r]},10^{[s]})$ $(rs=0)$ or $(10,10)$. All such candidate subgroups are $E_8$-irreducible except for $G_2 \hookrightarrow G_2 G_2$ $(p=7)$ via $(10,10)$.  
\end{lem}

\begin{proof}
The previous discussion shows that a candidate subgroup must be conjugate to a diagonal subgroup of $G_2 G_2$. Corollary \ref{notrivs} shows that the subgroups $G_2 \hookrightarrow G_2 G_2 < G_2 F_4$ $(p=7)$ via $(10^{[r]},10^{[s]})$ $(rs=0)$ are all $E_8$-irreducible. We need to prove $X = G_2 \hookrightarrow G_2 G_2$ via $(10,10)$ $(p=7)$ is contained in a parabolic subgroup of $E_8$. We use the same approach as in the proof of Lemma \ref{bada2a2}. First, let $S = G_2(7) < X$. We show Lemma \ref{subspaces} applies to $S < X$ acting on $L(E_8)$. We have $L(E_8) \downarrow X = 30 / 11^2 / 20^2 / 01^3 / 00$, so conditions (i) and (iii) of the lemma hold. Condition (ii) holds directly from \cite[Theorem 7.4]{cps}. 

We now show that $S$ fixes a 14-dimensional, abelian subalgebra of $L(E_8)$ that is ad-nilpotent of exponent 3. We use Magma \cite{magma} to check this. First we need to write down generators of $S$ in terms of root group elements. We will use the notation $a_1 \ldots a_8$ for the root $a_1 \alpha_1 + \cdots + a_8 \alpha_8$ in $E_8$. Consider $\bar{A}_2 E_6 < E_8$. The generators for $\bar{A}_2$ are $x_{\pm \alpha_8}(t)$ and $x_{\pm(23465432)}(t)$, where $t \in K$. The generators for $E_6$ are $x_{\pm \alpha_1}$(t), \ldots, $x_{\pm \alpha_6}(t)$ and the generators for $F_4$ are $x_{\alpha_1}(t)x_{\alpha_6}(t)$,  $x_{-\alpha_1}(t)x_{-\alpha_6}(t)$, $x_{\alpha_3}(t)x_{\alpha_5}(t)$,  $x_{-\alpha_3}(t)x_{-\alpha_5}(t)$, $x_{\pm \alpha_2} (t)$ and $x_{\pm \alpha_4}(t)$. Let $AB = G_2(7) G_2(7) < G_2 F_4$. The following elements are a set of generators for $A$: 
 \begin{align*}
x_{\gamma_1} & = x_{22343221}(1) x_{12343321}(6) x_{12244321}(1),\\
x_{-\gamma_1} & = x_{-(22343221)}(1) x_{-(12343321)}(6) x_{-(12244321)}(1),\\
x_{\gamma_2} & = x_{23465432}(1),\\
x_{-\gamma_2} & = x_{-(23465432)}(1). 
\end{align*}

 From \cite[Prop. G.1]{testerman}, the following elements are a set of generators for $B$:
\begin{align*} 
x_{\gamma'_1} & = x_{\alpha_1}(1) x_{\alpha_3}(1) x_{\alpha_1 + \alpha_3}(3) x_{\alpha_2}(1) x_{\alpha_5}(1) x_{\alpha_6}(1) x_{\alpha_5 + \alpha_6}(3),\\
x_{-\gamma'_1} & = x_{-\alpha_1}(2) x_{-\alpha_3}(2) x_{-\alpha_1 - \alpha_3}(2) x_{-\alpha_2}(1) x_{-\alpha_5}(2) x_{-\alpha_6}(2) x_{-\alpha_5 - \alpha_6}(2),\\
x_{\gamma'_2} & = x_{\alpha_3 + \alpha_4}(1) x_{\alpha_2 + \alpha_4}(3) x_{\alpha_4 + \alpha_5}(6),\\
x_{-\gamma'_2} & = x_{-\alpha_3 - \alpha_4}(1) x_{-\alpha_2 - \alpha_4}(5) x_{-\alpha_4 - \alpha_5}(6). 
\end{align*}

The map $x_{\pm \gamma_i} \rightarrow x_{\pm \gamma'_i}$ gives an isomorphism $A \rightarrow B$. Therefore \[S = \langle x_{\gamma_1} x_{\gamma'_1}, x_{-\gamma_1} x_{-\gamma'_1}, x_{\gamma_2} x_{\gamma'_2}, x_{-\gamma_2} x_{-\gamma'_2} \rangle.\] We check that $S \cong G_2(7)$ and that there is a unique 14-dimensional abelian subalgebra of $L(E_8)$ that is ad-nilpotent of exponent 3, which is fixed by $S$. Therefore, by Lemma \ref{subspaces} $X$ fixes this subalgebra and exponentiating gives a unipotent group normalised by $X$. Hence $X$ is contained in a parabolic subgroup of $E_8$.
\end{proof}

\subsection{Maximal  $M = B_2$ $(p \geq 5)$}

The maximal $B_2$ is the only candidate subgroup and is $E_8$-irreducible by maximality. 

This completes the proof of Theorem 3. 

\section{Corollaries} \label{cors}

Corollary 1 follows from Tables \ref{G2tab}--\ref{E8tab}, noting that all $G$-irreducible conjugacy class representatives have a unique set of composition factors on $L(G)$, apart from the one exception in the statement. Corollary 3 also follows immediately from Tables \ref{G2tab}--\ref{E8tab}. Corollary 4 follows from the proofs of Lemma \ref{simpleg2}, Theorem \ref{simplef4} and Theorems 1--3.  

Corollary 2 requires more work. First we need a technical lemma for $A_2$ subgroups when $p=3$. 

\begin{lem} \label{weirda2}
Let $p=3$ and $A_2^2 < A_8$ be embedded via $V_{A_8}(\lambda_1) \downarrow A_2^2 = (10,10)$. If $Y_1 = A_2 \hookrightarrow A_2^2$ via $(10,10^{[r]})$ and $Y_2 = A_2 \hookrightarrow A_2^2$ via $(10,01^{[r]})$ (with $r \neq 0$ in both) then the following modules have the indicated socle series: 

\begin{enumerate}[label=\normalfont(\arabic*),leftmargin=*]
\item $V_{A_8}(\lambda_1 + \lambda_8) \downarrow Y_1 \cong V_{A_8}(\lambda_1 + \lambda_8) \downarrow Y_2$ = $(11 +11^{[r]}) | (11 \otimes 11^{[r]} + 00^2) | (11 + 11^{[r]})$,

\item $V_{A_8}(\lambda_3) \downarrow Y_1$ = $( (11 + 11^{[r]}) | (11 \otimes 11^{[r]} + 30 + 30^{[r]} + 00^i) | (11 + 11^{[r]}) )+ 00^{1-i}$ with $i=0$ or $1$,

\item $V_{A_8}(\lambda_3) \downarrow Y_2$ = $( (11 + 11^{[r]}) | (11 \otimes 11^{[r]} + 30 + 03^{[r]} + 00^i) | (11 + 11^{[r]}) ) + 00^{1-i}$ with $i=0$ or $1$.
\end{enumerate}

\end{lem}

\begin{proof}
For part (1) note that $V_{A_8}(\lambda_1) \otimes V_{A_8}(\lambda_8) = V_{A_8}(0) | V_{A_8}(\lambda_1 + \lambda_8) | V_{A_8}(0)$ and $(V_{A_8}(\lambda_1) \downarrow Y_1) \otimes (V_{A_8}(\lambda_8) \downarrow Y_1) \cong (V_{A_8}(\lambda_1) \downarrow Y_2) \otimes (V_{A_8}(\lambda_8) \downarrow Y_2) = 10 \otimes 10^{[r]} \otimes 01 \otimes 01^{[r]}$. It is therefore enough to prove for an $A_2$ that the self-dual module $M := 10 \otimes 10^{[r]} \otimes 01 \otimes 01^{[r]}$ has socle series $00 | (11 +11^{[r]}) | (11 \otimes 11^{[r]} + 00^2) | (11 + 11^{[r]}) | 00$. The composition factors are as indicated, so we need to show the structure of the module is as claimed. Firstly, $M = (10 \otimes 01) \otimes (10^{[r]} \otimes 01^{[r]}) = (00|11|00) \otimes (00|11^{[r]}|00)$ and hence $M$ only has a $1$-dimensional trivial submodule and a $1$-dimensional trivial quotient. If there is a trivial direct summand then $M$ does not have enough composition factors that extend the trivial (see Lemma \ref{a2ext}) to block all of the other three trivial composition factors, leading to a trivial submodule or quotient of larger dimension. Hence $M$ has no trivial direct summands. Furthermore, the composition factor $11 \otimes 11^{[r]}$ does not occur as a submodule or quotient of $M$. This is because if we restrict to $S \cong SL(3,3)$ inside $Y_1$ acting as $10 \otimes 10$ on $V_{A_8}(\lambda_1)$ then $V_{S}(11) \otimes V_{S}(11)$ does not occur as a submodule or quotient of $(V_S(00)|V_S(11)|V_S(00)) \otimes (V_S(00)|V_S(11)|V_S(00))$ (checked using Magma \cite{magma}). Therefore $\text{Socle}(M) = 11^a + (11^{[r]})^b + 00$, and since $M$ is self-dual it follows that $a$ and $b$ are at most 1. Moreover, swapping the field twists $0$ and $r$ in $M$ does not change $M$ so it is also symmetrical, hence $a=b$. Suppose $a = b = 1$. Then $M = (11 + 11^{[r]} + 00) | (11 \otimes 11^{[r]} + 00^2) |  (11 + 11^{[r]} + 00)$ since $11 \otimes 11^{[r]}$ does not extend the trivial module. But no such $M$ exists since the module in the middle of $M$ does not extend the trivial module, hence there are three trivial modules for the socle to block. This is impossible because $\text{Ext}^1_{A_2}(11,00) \cong K \cong \text{Ext}^1_{A_2}(11^{[r]},00)$ by Lemma \ref{a2ext}. Hence $a=b=0$ and the socle is $00$. Now consider the socle of $M / 00$. Since $11 \otimes 11^{[r]}$ is not in the socle of $M$ and does not extend the trivial module it is not in the socle of $M/ 00$. Similarly, the socle does not have a composition factor isomorphic to $00$. Hence the socle must be $11 + 11^{[r]}$, again by self-duality and symmetry. It therefore follows that $M$ has the required socle series, since there are no non-trivial extensions between $11 \otimes 11^{[r]}$ and $00^2$.     

For part (2) we need to prove that $M := \bigwedge^3(10 \otimes 10^{[r]})$ has one of the two indicated socle series ($i = 0$ or $1$). The composition factors are easily checked to be correct and we again consider $S \cong SL(3,3)$ inside $Y_1$, acting as $10 \otimes 10$ on $V_{A_8}(\lambda_1)$. Using Magma, we check that $\bigwedge^3(10 \otimes 10) \downarrow S = 22 + 00 + ( 11 | (30 + 03) | 11 ) + ( 11 | (30 + 00) | 11 )^2$ and $11 \otimes 11 \downarrow S = 22 + 00  + ( 11 | (30 + 03 + 00) | 11 )$. Hence we see that none of $11 \otimes 11^{[r]}$, $30$ or $30^{[r]}$ are submodules or quotients of the $Y_1$-module $M$. Lemma \ref{a2ext} shows that  only $11$ and $11^{[r]}$ extend these three $Y_1$-modules and so $M$ must have the structure $( (11 + 11^{[r]}) | (11 \otimes 11^{[r]} + 30 + 03^{[r]}) | (11 + 11^{[r]}) ) + 00$ or $ (11 + 11^{[r]}) | (11 \otimes 11^{[r]} + 30 + 03 + 00) | (11 + 11^{[r]})$.

The last part is similar to the previous one. Let $M := \bigwedge^3(10 \otimes 01^{[r]})$. In this case we consider the subgroup $S' \cong SL(3,3)$ of $Y_2$ acting as $10 \otimes 01$ on $V_{A_8}(\lambda_1)$. Then $\bigwedge^3(10 \otimes 01) \downarrow S' = 22 + 00 + ((11 + 11) | (30 + 03 + 11) | 11 ) + ( 11 | (30 + 03 + 11) | (11 + 11) )$. It follows that $11 \otimes 11^{[r]}$ could occur as a submodule or a quotient but not a direct summand of the $Y_2$-module $M$. Furthermore, $30$ and $03^{[r]}$ cannot occur as submodules or quotients of the $Y_2$-module $M$. Using Lemma \ref{a2ext} again, we conclude that $M$ must have one of the structures indicated, completing the proof. 
\end{proof}

\begin{lem} \label{weirdG2C3}
Let $G = E_7$, $p=2$ and $X = G_2 \hookrightarrow G_2 G_2 < G_2 C_3$ via $(10,10)$. Then $V_{56} \downarrow X$ is indecomposable.  
\end{lem}

\begin{proof}
It is enough to prove that $V_{56} \downarrow S$ is indecomposable where $S = G_2(2) < X$. By \cite[Table 10.2]{LS1}, $V_{56} \downarrow G_2 C_3 = (00,100) | ( (10,100) + (00,001)) | (00,100)$. This module is constructed from $M_1 := (00,100) | (10,100) | (00,100)$ and $M_2 := \bigwedge^3((00,100)) = (00,100) | (00,001) | (00,100)$ as follows: take a maximal submodule of $M_1 + M_2$, call it $M_3$ and then quotient $M_3$ by a diagonal submodule of $\textrm{Soc}(M_3) = (00,100) + (00,100)$. We construct such a module for $G_2(2) \times C_3(2)$ in Magma. This module is still indecomposable when restricted to $S$ and in fact has socle series $(00 + 10 + 01) | (20 + 00) | (20 + 00) | (10 + 01 + 00)$.  
\end{proof}

\begin{pf}{Proof of Corollary 2}
The strategy for the proof is as follows. For each exceptional algebraic group $G$ and each reductive, maximal connected subgroup $M$ of $G$ we find all simple $M$-irreducible connected subgroups of rank at least $2$ that are not $G$-irreducible from the proofs of Lemma \ref{simpleg2}, Theorem \ref{simplef4} and Theorems 1--3. Given such a subgroup $X$ we then check whether it satisfies the hypothesis of Corollary \ref{nongcr}. That is to say, we check whether $X$ is contained reducibly in another reductive, maximal connected subgroup $M'$.

The result is trivial for $G = G_2$. 

For $G = F_4$, we follow the proof of Theorem \ref{simplef4} in \cite{dav} noting that the only subgroup that is $M$-irreducible for some $M$, but not $F_4$-irreducible is $X = A_2 \hookrightarrow A_2 \tilde{A}_2$ via $(10,01)$ when $p=3$. Suppose $X$ is contained reducibly in another reductive, maximal connected subgroup $M'$ or contained in a Levi subgroup. Consideration of the composition factors of $X$ on $V_{26}$ shows that the only possibility is $M' = B_4$ (which contains a Levi subgroup  $B_3 T_1$). The possibilities for a subgroup $A_2$ with the same composition factors as $X$ contained in $B_4$ are $V_{B_4}(1000) \downarrow A_2 = 11 + 00^2$ or $00|11|00$. By \cite[Prop 4.2.2]{dav}, neither of these subgroups is conjugate to $X$. Hence $X$ is not contained in $B_4$ and satisfies the hypothesis of the corollary. It is therefore listed in Table \ref{cortab}. 

Now let $G=E_6$. The only simple connected subgroups which are $M$-irreducible for some $M$ but not $G$-irreducible can immediately be found from the proof of Theorem 1. Most are contained in a $D_5$ Levi subgroup because they are contained in a maximal subgroup $B_4$ of $F_4$. These subgroups do not satisfy the hypothesis of Corollary \ref{nongcr}. The remaining possibilities are $A_2 \hookrightarrow A_2^3$ via $(10,10,10)$ and $A_2 \hookrightarrow A_2 \tilde{A}_2 < A_2 G_2$ $(p=3)$ via $(10,10)$. By the proof of Lemma \ref{a2e6}, the first subgroup is contained in a $D_4$ Levi subgroup when $p \neq 3$ and is contained $F_4$-reducibly in $F_4$ when $p=3$. Therefore it does not satisfy the hypothesis. Now let $p=3$ and consider $X = A_2 \hookrightarrow A_2 \tilde{A}_2 < A_2 G_2$ via $(10,10)$. Suppose $X$ is contained reducibly in another reductive, maximal connected subgroup or contained in a Levi subgroup. Consideration of the composition factors of $X$ on $V_{27}$ shows that $X$ is contained in a Levi $A_5$. We have $L(E_6)' \downarrow A_5 = V_{A_5}(\lambda_1 + \lambda_5) + V_{A_5}(\lambda_3)^2 + 0^3$ (when $p=3$ the centres of $E_6$ and $A_5$ coincide). By considering the finite subgroup $A_2(9) < X$ and using the same method as in the proof of Lemma \ref{weirdG2C3}, we see that $X$ has two direct summands of dimension 21 on $L(E_6)'$ (and only one trivial direct summand). Therefore $X$ is not contained in a Levi $A_5$ and satisfies the hypothesis of Corollary \ref{nongcr} and $X$ is listed in Table \ref{cortab}.    

For $G = E_7$ the result is checked in the same way as for $E_6$. By Lemmas \ref{a7e7}, \ref{badA2} and \ref{g2c3e7} the only subgroups which are $M$-irreducible for some $M$ but not $E_7$-irreducible are $X_1 = C_4 < A_7$, $X_2 = D_4 < A_7$ ($p=2$), $X_3 = A_2 < A_7$ $(p=2)$, $Y = A_2 \hookrightarrow A_2 A_2^{(*)} < A_2 A_5$ $(p=3)$ via $(10,10)$ and $Z = G_2 \hookrightarrow G_2 G_2 < G_2 C_3$ $(p=2)$ via $(10,10)$. 

First consider $X_1$. When $p \neq 2$, $X_1$ is contained in a Levi $E_6$ subgroup since its centraliser in $E_7$ contains a torus (see the proof of Lemma \ref{a7e7}) and hence $X_1$ does not satisfy the hypothesis of the corollary. Now let $p=2$. Suppose $X_1$ is contained in another reductive, maximal connected subgroup or contained in a Levi subgroup. Consideration of the composition factors of $X_1$ on $V_{56}$ shows that the only possibilities are that $X_1$ is contained in $E_6$ or $A_1 F_4$. The connected component of the centraliser in $E_7$ of $X_1$ is a 1-dimensional connected unipotent subgroup (see the proof of Lemma \ref{a7e7}). Therefore, $X_1$ is not contained in $E_6$ or $A_1 F_4$ since the connected component of the centraliser of $X_1$ would contain $T_1$ or an $A_1$, respectively. Hence $X_1$ satisfies the hypothesis of Corollary \ref{nongcr} and is listed in Table \ref{cortab}. 

Now consider $X_2$. Suppose $X_2$ is contained in another reductive, maximal connected subgroup or contained in a Levi subgroup. As with $X_1$, consideration of the composition factors of $X_2$ on $V_{56}$ shows that the only possibilities are that $X_2$ is contained in $E_6$ or $A_1 F_4$. By \cite[Table 8.6]{LS3}, $V_{56} \downarrow A_7 = \lambda_2 \oplus \lambda_6$, and therefore $V_{56} \downarrow X_2 = (0 | \lambda_2 | 0)^2$. The restrictions of $V_{56}$ to $E_6$ and $A_1 F_4$ are as follows: $V_{56} \downarrow E_6 = \lambda_1 \oplus \lambda_6 \oplus 0^2$, by \cite[Table 8.6]{LS3} and $V_{56} \downarrow A_1 F_4 = (1,0001) + (3,0000)$, by \cite[Table 10.2]{LS1}. Therefore, $X_2$ is not contained in $E_6$ or $A_1 F_4$ and hence $X_2$ satisfies the hypothesis of Corollary \ref{nongcr}. 

The last subgroup of $A_7$ to consider is $X_3$. This $A_2$ subgroup is contained in $D_4$ and is the group of fixed points of a triality automorphism of $D_4$ induced by an element $t$ of $E_7$. By \cite[Prop. 1.2]{LS8}, the full centraliser of $t$ is either a Levi subgroup or $A_2 A_5$. Since there are no $A_2 A_5$-irreducible $A_2$ subgroups when $p=2$, it follows that $X_3$ is contained either in a Levi subgroup or $M'$-reducibly in a reductive, maximal connected subgroup $M'$. Hence $X_3$ does not satisfy the hypothesis of Corollary~\ref{nongcr}.  

Next consider $Y$. Suppose $Y$ is contained in another reductive, maximal connected subgroup or contained in a Levi subgroup. Consideration of the composition factors of $Y$ on $V_{56}$ shows that $Y$ is contained in $A_7$. The proof of Lemma \ref{badA2} shows $Y$ is not contained in $A_7$ and so $Y$ satisfies the hypothesis of the corollary and is in Table \ref{cortab}. 

Finally, consider $Z$. If $Z$ does not satisfy the hypothesis of the corollary then $Z$ is contained in an $E_6$ Levi subgroup. From Table \ref{levie7} it follows that $V_{56} \downarrow E_6 = \lambda_1 + \lambda_6 + 0^2$. Lemma \ref{weirdG2C3} shows that $Z$ does not have any trivial direct summands on $V_{56}$ and hence $Z$ is not contained in $E_6$. Therefore, $Z$ satisfies the hypothesis and is listed in Table \ref{cortab}.      

Finally, let $G=E_8$. The proof of Theorem 3 yields the candidate subgroups that are not \linebreak $G$-irreducible. They are those listed in Table \ref{cortab}. It remains to prove that they satisfy the hypothesis of Corollary \ref{nongcr}, that is that they are not contained reducibly in another reductive, maximal connected subgroup or contained in a Levi subgroup. 

First, consider $X = B_4 (\ddagger) < D_8$, with $p=2$. The proof of Lemma \ref{badB4} shows $X$ is only contained in $D_8$ and not $A_8$ or a Levi $A_7$, hence it satisfies the hypothesis of the corollary. 

Next, consider the diagonal subgroups of $B_2^2 (\ddagger)$ $(p=2)$. Consideration of their composition factors on $L(E_8)$ shows that the only possibilities for another maximal subgroup or Levi subgroup containing them are $A_8$ or $A_4^2$ (or a Levi $A_7$, $A_3 A_4$ or $A_3 A_3$ but they are contained in $A_8$ or $A_4^2$). First, let $X_1 = B_2 \hookrightarrow B_2^2 (\ddagger)$ via $(10,10^{[r]})$ with $r \neq 0$. Then $X_1$ contains a subgroup $S_1 \cong Sp(4,2)$ embedded in $D_8$ via $01 \otimes 01$. Using Magma \cite{magma}, we check that $\bigwedge^2(V_{S_1}(01) \otimes V_{S_1}(01))$ has an indecomposable direct summand of dimension 88. Therefore $X_1$ has an indecomposable direct summand of $L(E_8)$ of dimension at least 88. But the largest dimension of a direct summand of $L(E_8) \downarrow A_4^2$ is 50 and  of $L(E_8) \downarrow A_8$ is 84, hence $X_1$ is not a subgroup of $A_4^2$ or $A_8$. To prove $X_2 = B_2 \hookrightarrow B_2^2 (\ddagger)$ via $(10,02^{[r]})$ (including $r=0$) is not contained in $A_4^2$ we consider the parity of $r$. If $r$ is even then $X_2$ contains a subgroup $S_2 \cong Sp(4,4)$ embedded in $D_8$ via $01 \otimes 10 = 11$. If $r$ is odd then $X_2$ contains a subgroup $S_3 \cong Sp(4,4)$ embedded in $D_8$ via $01 \otimes 20 = 21$.  Again using Magma, we find that both $\bigwedge^2(V_{S_2}(11))$ and $\bigwedge^2(V_{S_3}(21))$ are indecomposable (of dimension 120). Therefore, for all $r$, the subgroup $X_2$ has a 120-dimensional indecomposable direct summand on $L(E_8)$ and is not a subgroup of $A_4^2$ or $A_8$. Hence both $X_1$ and $X_2$ satisfy the hypothesis of Corollary \ref{nongcr}. 

Now let $p=3$ and consider $Y_1 = A_2 \hookrightarrow A_2^2$ via $(10,10^{[r]})$ and $Y_2 = A_2 \hookrightarrow A_2^2$ via $(10,01^{[r]})$ ($r \neq 0$ in both cases). Consideration of their composition factors on $L(E_8)$ shows that the only reductive, maximal connected subgroup that can contain $Y_1$ or $Y_2$ (other than $A_8$) is $D_8$ and the only Levi subgroup is $D_7$. So it suffices to show that $Y_1$ and $Y_2$ are not contained in $D_8$. By Theorem \ref{max}, we have $L(E_8) \downarrow D_8 = V_{D_8}(\lambda_2) + V_{D_8}(\lambda_7)$ where $V_{D_8}(\lambda_2)$ is 120-dimensional and $V_{D_8}(\lambda_7)$ is 128-dimensional. But Lemma \ref{weirda2} shows $Y_1$ and $Y_2$ have a 79-dimensional indecomposable summand and two indecomposable direct summands of dimension at least 83 on $L(E_8)$. Hence neither $Y_1$ nor $Y_2$ is contained in $D_8$ and they both satisfy the hypothesis of the corollary. 

Finally, let $p=7$ and $Z = G_2 \hookrightarrow G_2 G_2 < G_2 F_4$ via $(10,10)$. Other than $G_2 F_4$, the only reductive, connected maximal subgroup that can contain $Z$ is $D_8$ and the only Levi subgroup is $D_7$. There is only one $D_8$-conjugacy class of $G_2$ subgroups with the same composition factors as $Z$ on $L(E_8)$, with $V_{D_8}(\lambda_1) \downarrow G_2 = V_{G_2}(01) + V_{G_2}(00)^2$. This $G_2$ is contained in a $D_7$ Levi subgroup of $D_8$ (and $E_8$) and therefore $L(E_8) \downarrow G_2$ has a trivial submodule. However, restricting from $L(E_8) \downarrow G_2 F_4$ in \cite[Table 10.1]{LS1}, we find that $L(E_8) \downarrow Z = 10 \otimes 20 + 01^2 + 11$. Since $\mathrm{Hom}_{G_2}(00,10 \otimes 20) = \mathrm{Hom}_{G_2}(10,20) = 0$, we see that $L(E_8) \downarrow Z$ has no trivial submodules. Hence $Z$ is not contained in $D_7$ and does satisfy the hypothesis of Corollary \ref{nongcr}. \qed
\end{pf}

\section{Variations of Steinberg's Tensor Product Theorem}\label{vars} 

We need some background for the next set of corollaries. Let $X$ be a simple, simply connected algebraic group over an algebraically closed field $K$ of characteristic $p < \infty$. We recall Steinberg's tensor product theorem \cite{stein}. It states that if $\phi : X \rightarrow SL(V)$ is an irreducible rational representation, then we can write $V = V_1^{[r_1]} \otimes \ldots \otimes V_k^{[r_k]}$, where the $V_i$ are restricted $X$-modules and the $r_i$ are distinct. The main result of \cite{LS4} generalises this conclusion to the situation where $\phi$ is a rational homomorphism from $X$ to an arbitrary simple algebraic group $G$. We describe this generalisation now. Recall the definition of a subgroup being $G$-cr and of being restricted from Section \ref{intro}.  

\begin{thm} \textup{\cite[Corollary 1]{LS4}} \label{stein}
Assume $p$ is good for $G$. If $X$ is a connected simple $G$-cr subgroup of $G$, then there is a uniquely determined commuting product $Y_1 \dots Y_k$ with $X \leq Y_1 \dots Y_k \leq G$, such that each $Y_i$ is a simple restricted subgroup of the same type as $X$, and each of the projections $X \rightarrow Y_i/Z(Y_i)$ is non-trivial and involves a different field twist. 
\end{thm}

Using our classification of $G$-irreducible subgroups, we can investigate to what extent Theorem \ref{stein} is true in bad characteristics for simple, connected $G$-irreducible subgroups of rank at least $2$. For $G= G_2, F_4, E_6$ and $E_7$ the bad characteristics are $2, 3$ and for $G = E_8$ they are $2, 3, 5$. To save repeating ourselves, we will say a subgroup $X$ satisfies the conclusion of Theorem \ref{stein} if there is a uniquely determined commuting product $Y_1 \dots Y_k$ with $X \leq Y_1 \dots Y_k \leq G$, such that each $Y_i$ is a simple restricted subgroup of the same type as $X$, and each of the projections $X \rightarrow Y_i/Z(Y_i)$ is non-trivial and involves a different field twist.

\begin{cor}

Let $G = G_2$ and $X$ be a simple, connected irreducible subgroup of rank at least $2$. Then either $X$ satisfies the conclusion of Theorem \ref{stein} or $p=3$ and $X = \tilde{A}_2$. 

\end{cor}

\begin{proof}
From Lemma \ref{simpleg2} we know the only $G_2$-irreducible subgroups which are simple and of rank at least $2$ are $A_2$ and $\tilde{A}_2$ $(p=3)$. Both are maximal subgroups, so we only have to check whether they are restricted for $p=2, 3$. To do this we use Table \ref{G2tab} which shows $A_2$ is restricted for both $p=2, 3$ but $\tilde{A}_2$ is not restricted for $p=3$. 
\end{proof}

\begin{cor}

Let $G = F_4$ and $X$ be a simple, connected irreducible subgroup of rank at least $2$. Then either $X$ satisfies the conclusion of Theorem \ref{stein} or $p=2$ and $X$ is conjugate to one of the following subgroups:

\begin{enumerate}[leftmargin=*,label=\normalfont(\arabic*)]

\item \label{c4p2lab} $C_4$;

\item \label{d4p2lab} $\tilde{D}_4$;

\item \label{a2a2p2lab} $A_2 \hookrightarrow A_2 \tilde{A}_2$ via any $G$-irreducible embedding;

\item \label{b2p2label} $B_2 \hookrightarrow B_2^2$ via $(10,02)$. 

\end{enumerate}

\end{cor}

\begin{proof}
Using the proof of Theorem \ref{simplef4} we find that for each simple $F_4$-irreducible connected subgroup $X$ there is at most one commuting product of restricted groups, of the same type as $X$, containing $X$ as a diagonal subgroup with distinct field twists. We have to check if the subgroups in the possible commuting product are indeed restricted (that is the only obstruction to the conclusion of Theorem \ref{stein}).  

If the possible commuting product is just $X$ itself, then we use Table \ref{F4tab} to check whether $X$ is restricted or not and hence whether $X$ satisfies the conclusion of Theorem \ref{stein}. This leads to the subgroups in \ref{c4p2lab}, \ref{d4p2lab} and \ref{b2p2label} as well as $A_2 \hookrightarrow A_2 \tilde{A}_2$ via $(10,10)$ and $(10,01)$ in \ref{a2a2p2lab}. Note that the subgroup in \ref{b2p2label}, namely $B_2 \hookrightarrow B_2^2$ via $(10,02)$, does not satisfy the conclusion of Theorem \ref{stein}: firstly, it is not itself $2$-restricted; secondly, despite appearances the twists in the two factors are not distinct, as $02$ is just our notation for an endomorphism of $B_2$ involving a graph automorphism, which is untwisted although it happens to induces a field twist when applied to $V_{B_2}(10)$. 

Now let $X \cong A_2$. We see that the only commuting product of $A_2$ subgroups containing $X$ as a diagonal subgroup with distinct field twists is $A_2 \tilde{A}_2$ (or just $X$ itself which is covered above). Since $\tilde{A}_2$ is not 2-restricted ($L(F_4) \downarrow \tilde{A}_2 = W(11) / W(20)^3 / W(02)^3 / 00^8$) it follows that when $p=2$, the subgroup $X$ does not satisfy the conclusion of Theorem \ref{stein}. When $p=3$, both $A_2$ and $\tilde{A}_2$ are 3-restricted and so $X$ satisfies the conclusion of Theorem \ref{stein}. 

Let $X \cong B_2$. Then $p=2$ and the possible commuting product is $B_2^2$, containing $X$ diagonally with distinct field twists. The two $B_2$ factors are restricted ($L(F_4) \downarrow B_2 = 10^5 / 01^5 / 00^{12}$) and hence $X$ satisfies the conclusion of Theorem \ref{stein}. 
\end{proof}

\begin{cor}

Let $G = E_6$ and $X$ be a simple, connected irreducible subgroup of rank at least $2$. Then either $X$ satisfies the conclusion of Theorem \ref{stein} or $p=2$ or $3$ and $X$ is conjugate to one of the following subgroups:

\begin{enumerate}[leftmargin=*,label=\normalfont(\arabic*)]

\item  \label{a2p3lab} maximal $A_2$ $(p = 3)$;

\item $A_2 \hookrightarrow A_2 \tilde{A}_2 < A_2 G_2$ $(p=3)$ via $(10^{[r]},10^{[s]})$ $(rs=0)$;

\item \label{c4p2e6lab} $C_4$ $(p=2)$;

\item \label{d4p2e6lab} $D_4$ $(p=2)$;

\item $A_2 \hookrightarrow A_2^3$ $(p=2)$ via:
\begin{enumerate}
\item \label{a2a2a2e6lab} $(10,10,01)$

\item $(10,10,10^{[r]})$ $(r \neq 0)$

\item $(10,10,01^{[r]})$ $(r \neq 0)$

\item $(10,10^{[r]},10^{[r]})$ $(r \neq 0)$

\item $(10,01^{[r]},01^{[r]})$ $(r \neq 0)$.

\end{enumerate}

\end{enumerate}

\end{cor}

\begin{proof}
We proceed as we did for $F_4$. The proof of Theorem 1 shows that for each simple $E_6$-irreducible connected subgroup $X$, there is at most one commuting product of restricted groups, of the same type as $X$, containing $X$ as a diagonal subgroup with distinct field twists. When this is just $X$, Table \ref{E6tab} is used to determine whether $X$ is restricted or not and hence whether $X$ satisfies the conclusion of Theorem \ref{stein}. This yields subgroups \ref{a2p3lab}, \ref{c4p2e6lab}, \ref{d4p2e6lab} and \ref{a2a2a2e6lab} in the conclusion of the corollary.  

Now let $X \cong A_2$, and first suppose $X \hookrightarrow A_2 \tilde{A}_2 < A_2 G_2$ via $(10^{[r]},10^{[s]})$ when $p=3$. Because $\tilde{A}_2$ is not 3-restricted ($L(G) \downarrow \tilde{A}_2 = 30 / 03 / 11^9 / 00^9$) it follows that $X$ does not satisfy the conclusion of Theorem \ref{stein}. Now suppose $X < \bar{A}_2^3$. Since $\bar{A}_2$ is both 2-restricted and 3-restricted, if $X$ is a diagonal subgroup with three distinct field twists then $X$ satisfies the conclusion of Theorem \ref{stein} for $p=2,3$. If $X$ has two distinct field twists then it is contained in $\bar{A}_2 A < \bar{A}_2^3$ or $\bar{A}_2 B < \bar{A}_2^3$ where $A = A_2 \hookrightarrow \bar{A}_2^2$ via $(10,10)$ and $B = A_2 \hookrightarrow \bar{A}_2^2$ via $(10,01)$. We have $L(E_6) \downarrow A = W(11)^2 / W(20)^3 / W(02)^3 / 10^3 / 01^3 / 00^8$  and $L(E_6) \downarrow B = W(11)^8 / 00^{14}$. It follows that $B$ is both 2-restricted and 3-restricted, but $A$ is only 3-restricted. Therefore $X$ satisfies the conclusion of Theorem \ref{stein} unless $p=2$ and it is contained in $\bar{A}_2 A$.  
\end{proof}

\begin{cor}

Let $G = E_7$ and $X$ be a connected, simple irreducible subgroup of rank at least $2$. Then either $X$ satisfies the conclusion of Theorem \ref{stein} or $p=2$ or $3$ and $X$ is conjugate to one of the following subgroups: 

\begin{enumerate}[leftmargin=*,label=\normalfont(\arabic*)]

\item $A_2 \hookrightarrow A_2 A_2^{(\star)} < A_2 A_5$ $(p = 3)$ via $(10,01)$ (where $A_2^{(\star)}$ is embedded in $A_5$ via $V_{A_2}(20)$);

\item $G_2 \hookrightarrow G_2 G_2 < G_2 C_3$ $(p=2)$ via $(10^{[r]},10^{[s]})$ $(rs=0)$.

\end{enumerate}

\end{cor}

\begin{proof}
Studying the proof of Theorem 2, specifically Lemmas \ref{a7e7}, \ref{badA2} and \ref{g2c3e7}, we find that for each subgroup $X$ there is at most one commuting product of restricted groups, of the same type as $X$, containing $X$ as a diagonal subgroup with distinct field twists. As before, it remains to check whether the groups in the product are restricted. 

From Table \ref{E7tab}, we see that $A_7$ and $D_4$ $(p>2)$ are restricted for $p=2,3$. Let $X = A_2 \hookrightarrow \bar{A}_2 A_2^{(\star)}$ via any $G$-irreducible embedding. The only possibility for a commuting product of $A_2$ restricted subgroups containing $X$ as a diagonal subgroup with distinct field twists is $\bar{A}_2 A_2^{(\star)}$. Both $\bar{A}_2$ and $A_2^{(\star)}$ are 3-restricted ($L(G) \downarrow A_2 = W(11) / 10^{15} / 01^{15} / 00^{35}$ and $L(G) \downarrow A_2^{(\star)} = 22 / 21^3 / 12^3 / W(11) / 00^8$). It follows that if $X$ has two distinct field twists then it satisfies the conclusion of Theorem \ref{stein}. If $X$ is embedded via $(10,01)$ then Table \ref{E7tab} shows that $X$ is not 3-restricted and hence does not satisfy the conclusion of Theorem \ref{stein}.  

Now consider $X \cong G_2$, with $p=2$. The only  possibility for a commuting product of $G_2$ subgroups containing $X$ as a diagonal subgroup with distinct field twists is $G_2 G_2 < G_2 C_3$. The subgroup $G_2 < C_3$ is not 2-restricted ($L(G) \downarrow G_2 = 20 / 01^8 / 00^{15}$) and hence $X$ does not satisfy the conclusion of Theorem \ref{stein}. 
\end{proof}

\begin{cor}

Let $G = E_8$ and $X$ be a connected, simple irreducible subgroup of rank at least $2$. Then either $X$ satisfies the conclusion of Theorem \ref{stein} or $p=2, 3$ or $5$ and $X$ is conjugate to one of the following subgroups:

\begin{enumerate}[leftmargin=*,label=\normalfont(\arabic*)]

\item maximal $B_2$ $(p = 5)$;

\item $A_2 \hookrightarrow \bar{A}_2 A_2 < \bar{A}_2 E_6$ $(p=3)$ via any $G$-irreducible embedding;

\item $A_2 \hookrightarrow \bar{A}_2 A_2 \tilde{A}_2 < \bar{A}_2 A_2 G_2$ $(p=3)$ via any $G$-irreducible embedding;

\item $A_2 \hookrightarrow A_2^4$ $(p=2)$ via any $G$-irreducible embedding that does not have four distinct field twists;

\item $A_2 \hookrightarrow A_2^2 < D_4^2$ $(p=2)$ via $(10,10^{[r]})$ $(r \neq 0)$ or $(10,01^{[r]})$ $(r \neq 0)$;

\item $B_2 \hookrightarrow B_2^2 (\dagger)$ $(p=2)$ via any $G$-irreducible embedding.

\end{enumerate}

\end{cor}

\begin{proof}
As for $E_7$, we use the proof of Theorem 3 to check that for each subgroup $X$ there is only one possibility for a commuting product of groups, of the same type as $X$, containing $X$ as a diagonal subgroup with distinct field twists. 

When this is just $X$, Table \ref{E8tab} is used to determine whether or not $X$ is restricted and hence whether it satisfies the conclusion of Theorem \ref{stein}. 

If $X \cong D_4$, $A_4$ or $B_3$ then $X$ satisfies the conclusion of Theorem \ref{stein} because the simple factors of \linebreak $D_4^2$, $A_4^2$ and $B_3^2 < D_4^2$ are all restricted for $p=2,3,5$. Indeed, $L(G) \downarrow D_4 = 0100 / 1000^8 / 0010^8 /$ $\! 0001^8 / 0000^{28}$, $L(G) \downarrow A_4 = W(1001) / 1000^{10} / 0100^5 / 0010^5 / 0001^{10} / 0000^{24}$ and $L(G) \downarrow B_3 = 010 /$ $\!100^9 / 001^{16} / 000^{36}$. 

Now suppose $X \cong A_2$. When $X$ is contained diagonally in $A_2^2 < D_4^2$ $(p \neq 3)$ (each factor $A_2$ irreducibly embedded) it satisfies the conclusion of Theorem \ref{stein} only when $p=5$. Indeed, $A_2 < D_4$ is $5$-restricted but not $2$-restricted since $L(G) \downarrow A_2 = 11^{25} / W(30) / W(03) / 00^{28}$. If $X$ is contained in $\bar{A}_2 A_2 \tilde{A}_2 < \bar{A}_2 A_2 G_2$ $(p=3)$ then it does not satisfy the conclusion of Theorem \ref{stein} because $\tilde{A}_2$ is not $3$-restricted ($L(G) \downarrow \tilde{A}_2 = 30 /$ $\! 03 /$ $\! 11^{27} /$ $\! 00^{53}$). If $X$ is contained in $\bar{A}_2 A_2 < A_2 E_6$ $(p \geq 3)$, where the $A_2$ acts on $V_{27}$ as $W(22)$ then $X$ satisfies the conclusion of Theorem \ref{stein} when $p=5$ but not when $p=3$. Indeed, $\bar{A}_2$ is restricted for all $p$ since $L(G) \downarrow A_2 = 11 /$ $\! 10^{27} /$ $\! 01^{27} /$ $\! 00^{78}$ but $A_2 < E_6$ is $5$-restricted but not $3$-restricted since $L(G) \downarrow A_2 = W(41) / W(14) / W(22)^6 / W(11) / 00^{8}$. 

Now let $X$ be contained in $A_2^4$. If $X$ is embedded with four distinct field twists then it satisfies the conclusion of Theorem \ref{stein} for all $p$ because $\bar{A}_2$ is restricted for all $p$, as seen in the previous paragraph. If $X$ is embedded with three distinct field twists then the possibilities for a commuting product of $A_2$ subgroups diagonally containing $X$ are $\bar{A}_2^2 A$ or $\bar{A}_2^2 B$ where $A = A_2 \hookrightarrow \bar{A}_2^2$ via $(10,10)$ and $B = A_2 \hookrightarrow \bar{A}_2^2$ via $(10,01)$. From $L(G) \downarrow A_2^4$ in subsection \ref{maxa2e6} we find that when $p=3$ or $5$, both $A$ and $B$ are restricted but neither $A$ nor $B$ is $2$-restricted and hence $X$ satisfies the conclusion of Theorem \ref{stein} only when $p>2$. Similarly, if $X$ is embedded with two distinct field twists then it satisfies the conclusion of Theorem \ref{stein} only when $p>2$. Indeed, the possibilities for a commuting product of $A_2$ subgroups diagonally containing $X$ are $A^2$, $A B$ and $\bar{A}_2 C$ where $C = A_2 \hookrightarrow \bar{A}_2^3$ via $(10,10,01)$. 

Finally, assume $X \cong B_2$. Suppose $X$ is diagonally contained in $B_2^2 (\ddagger)$, $B_2 B_2 < A_3 D_5$ or $B_2^3$ (all with $p \geq 3$) with distinct field twists. Then $X$ satisfies the conclusion of Theorem \ref{stein} because all of the simple $B_2$ factors are restricted for $p=3$ and $5$. Indeed, the factors are contained in $A_4$, $A_3$, $D_5$ and $A_3$ again, respectively and the restrictions of $L(G)$ are $20 / 02^{11} / 10^{20} / 00^{24}$, $02 / 10^{11} / 01^{32} / 00^{55}$ and $02^7 / 11^8 / 12 / 00^{15}$ respectively. Now suppose $X$ is diagonally contained in $B_2^2 (\dagger)$ with two distinct field twists. The simple $B_2$ factors are restricted for $p=3$ and $5$ but not for $p=2$ ($L(G) \downarrow B_2 = W(02)^6 / 11^4 / 10^{10} / 01^{16} / 00^{10}$). Hence $X$ satisfies the conclusion of Theorem \ref{stein} only when $p>2$.  
\end{proof}

\section{Tables} \label{tabs}

Here we give the Tables from Lemma \ref{simpleg2}, Theorem \ref{simplef4} and Theorems 1--3. The notation used is explained in Section \ref{nota}. All of the subgroups listed in Theorems 1--3 are either maximal and hence found in Theorem \ref{max} or an $M$-irreducible subgroup for one of the maximal connected subgroups $M$. Details of these can be found in Sections 4--6. For some subgroups we give a reference to the subsection they are defined in. We should also explain where all of the restrictions for $V_{26}$, $V_{27}$, $V_{56}$ and $L(G)$ come from. Theorem \ref{max} gives the composition factors for the maximal subgroups. From these it is just a case of restricting to the $M$-irreducible subgroups. This is a mainly routine calculation, and lots of them have been carried out already. Specifically, in \cite[Tables 8.1-8.7]{LS3}, the composition factors are given with some restrictions on the characteristic. However, even in the characteristics not covered, it is still possible to deduce the composition factors. Sometimes we give a composition factor as a Weyl module, which will be reducible in certain characteristics. The Weyl modules package for GAP, by S. Doty, can be used to determine the composition factors of these reducible Weyl modules. For convenience, Appendix A has a table listing all of the reducible Weyl modules that occur in Tables \ref{G2tab}--\ref{E8tab}, which have a trivial composition factor. 

\setlength\LTleft{0pt}
\begin{longtable}{p{0.26\textwidth - 2\tabcolsep}>{\raggedright\arraybackslash}p{0.36\textwidth-2\tabcolsep}>{\raggedright\arraybackslash}p{0.38\textwidth-\tabcolsep}@{}}

\caption{The simple, connected irreducible subgroups of $G_2$ of rank at least $2$. \label{G2tab}} \\

\hline \noalign{\smallskip}

Irreducible subgroup $X$ & Comp. factors of $V_7 \downarrow X$ & Comp. factors of $L(G_2) \downarrow X$ \\

\hline \noalign{\smallskip}

$A_2$ & $10 / 01 / 00$ & $W(11) / 10 / 01$  \\

\hline \noalign{\smallskip}

$\tilde{A}_2$ $(p=3)$ & $11$ & $11 / 30 / 03 / 00$ \\

\hline

\end{longtable}

\begin{longtable}{p{0.26\textwidth - 2\tabcolsep}>{\raggedright\arraybackslash}p{0.36\textwidth-2\tabcolsep}>{\raggedright\arraybackslash}p{0.38\textwidth-\tabcolsep}@{}}

\caption{The simple, connected irreducible subgroups of $F_4$ of rank at least $2$. \label{F4tab}} \\

\hline \noalign{\smallskip}

Irreducible subgroup $X$ & Comp. factors of $V_{26} \downarrow X$ & Comp. factors of $L(F_4) \downarrow X$ \\

\hline \noalign{\smallskip}

$B_4$ & $W(1000) /$ $\! 0001 /$ $\! 0000$  & $W(0100) /$ $\! 0001$ \\

\hline \noalign{\smallskip}

$D_4$ & $1000 /$ $\! 0010 /$ $\! 0001 /$ $\! 0000^2$ & $W(0100) /$ $\! 1000 /$ $\! 0010 /$ $\! 0001$ \\

\hline \noalign{\smallskip}

$C_4$ $(p=2)$ & $0100$ & $2000 /$ $\! 0100 /$ $\! 0001 /$ $\! 0000^2$ \\

\hline \noalign{\smallskip}

$\tilde{D}_4$ $(p=2)$ & $0100$ & $0100 /$ $\! 2000 /$ $\! 0020 /$ $\! 0002 /$ $\! 0000^2$ \\

\hline \noalign{\smallskip}

$G_2$ $(p=7)$ & $20$ & $01 /$ $\! 11$ \\

\hline \noalign{\smallskip}

$A_2 \hookrightarrow A_2 \tilde{A}_2$ via: & & \\

$(10,10)$ & $W(20) /$ $\! W(02) /$ $\! 10 /$ $\! 01 /$ $\! W(11)$ & $W(11)^2 /$ $\! W(21) /$ $\! W(12) /$ $\! 10 /$ $\! 01$  \\

$(10^{[r]},10^{[s]})$ $(rs=0)$ & $10^{[r]} \otimes 10^{[s]} /$ $\! 01^{[r]} \otimes 01^{[s]} /$ $\!$ $ W(11)^{[s]}$ & $W(11)^{[r]} /$ $\! W(11)^{[s]} /$ $\! 10^{[r]} \otimes W(02)^{[s]} /$ $\! 01^{[r]} \otimes W(20)^{[s]}$  \\

$(10,01)$ $(p \neq 3)$ & $11^3 /$ $\! 00^2$ & $11^4 /$ $\! W(30) /$ $\! W(03)$ \\

$(10^{[r]},01^{[s]})$ $(rs=0)$ & $10^{[r]} \otimes 01^{[s]} /$ $\! 01^{[r]} \otimes 10^{[s]} /$ $\!$ $ W(11)^{[s]}$ & $W(11)^{[r]} /$ $\! W(11)^{[s]} /$ $\! 10^{[r]} \otimes W(20)^{[s]} /$ $\! 01^{[r]} \otimes W(02)^{[s]}$  \\

\hline \noalign{\smallskip}

$B_2 \hookrightarrow B_2^2$ $(p=2)$ via: & & \\

$(10,10^{[r]})$ $(r \neq 0)$ & $10 /$ $\! 10^{[r]} /$ $\! 01 \otimes 01^{[r]} /$ $\! 00^2$ & $10 /$ $\! 10^{[r]} /$ $\! 02 /$ $\! 02^{[r]} /$ $\! 10 \otimes 10^{[r]} /$ $\! 01 \otimes 01^{[r]} /$ $\! 00^4$ \\

$(10,02)$ & $10 /$ $\! 02 /$ $\! 11 /$ $\! 00^2$  & $10 /$ $\! 20 /$ $\! 02^2 /$ $\! 11 /$ $\! 12 /$ $\! 00^4$ \\

$(10,02^{[r]})$ $(r \neq 0)$ & $10 /$ $\! 01^{[r+1]} /$ $\! 01 \otimes 10^{[r]} /$ $\! 00^2$ & $10 /$ $\! 01^{[r+1]} /$ $\! 02 /$ $\! 10^{[r+1]} /$ $\! 01 \otimes 10^{[r]} /$ $\! 10 \otimes 01^{[r+1]} /$ $\! 00^4$ \\

\hline

\end{longtable}

\pagebreak
\begin{longtable}{p{0.26\textwidth - 2\tabcolsep}>{\raggedright\arraybackslash}p{0.36\textwidth-2\tabcolsep}>{\raggedright\arraybackslash}p{0.38\textwidth-\tabcolsep}@{}}

\caption{The simple, connected irreducible subgroups of $E_6$ of rank at least $2$. \label{E6tab}} \\

\hline \noalign{\smallskip}

Irreducible subgroup $X$ & Composition factors of $V_{27} \downarrow X$ & Composition factors of $L(E_6) \downarrow X$ \\

\hline \noalign{\smallskip}

$F_4$ & $W(0001) /$ $\! 0000$ & $W(1000) /$ $\! W(0001) $ \\

\hline \noalign{\smallskip}

$C_4$ & $W(0100)$ & $W(2000) /$ $\! W(0001)$  \\

\hline \noalign{\smallskip}

$\tilde{D}_4 < F_4$ $(p=2)$ & $0100 /$ $\! 0000$ & $ 2000 /$ $\! 0020 /$ $\! 0002 /$ $\! 0100^2 /$ $\! 0000^2 $ \\ 

\hline \noalign{\smallskip}

$G_2$ &  $W(20)$ & $W(01) /$ $\! W(11)$ \\

\hline \noalign{\smallskip}

$A_2$ $(p \geq 3)$ & $W(22)$ & $W(11) /$ $\! W(14) /$ $\! W(41)$ \\

\hline \noalign{\smallskip}

\multicolumn{3}{l}{$A_2 \hookrightarrow A_2 \tilde{A}_2 < A_2 G_2$ $(p=3)$ via:} \\ \noalign{\smallskip} 
$(10^{[r]},10^{[s]})$ $(rs=0)$ & $10^{[r]} \otimes 11^{[s]} /$ $\! 02^{[r]}$ & $11^{[r]} /$ $\! 11^{[r]} \otimes 11^{[s]} /$ $\! (11^{[s]})^2 /$ $\! 30^{[s]} /$ $\! 03^{[s]} /$ $\! 00^2$    \\ 

\hline \noalign{\smallskip}

$A_2 \hookrightarrow A_2^3$ via: & & \\

$(10,10,01)$ & $10 /$ $\! 01 /$ $\! W(20) /$ $\! W(02) /$ $\! W(11) /$ $\! 00$ & $W(11)^3 /$ $\! 10^2 /$ $\! 01^2 /$ $\! W(20) /$ $\! W(02) /$ $\! W(21) /$ $\! W(12)$ \\

$(10,10,10^{[r]})$ $(r \neq 0)$ & $W(11) /$ $\! 10 \otimes 01^{[r]} /$ $\! 10^{[r]} \otimes 01 /$ $\! 00$ & $W(11)^2 /$ $\! W(11)^{[r]} /$ $\! W(20) \otimes 10^{[r]} /$ $\! 01 \otimes 10^{[r]} /$ $\! W(02) \otimes 01^{[r]} /$ $\! 10 \otimes 01^{[r]}$ \\

 $(10,10,01^{[r]})$ $(r \neq 0)$ &  $W(11) /$ $\! 10 \otimes 10^{[r]} /$ $\! 01^{[r]} \otimes 01 /$ $\! 00 $ & $W(11)^2 /$ $\! W(11)^{[r]} /$ $\! W(20) \otimes 01^{[r]} /$ $\! 01 \otimes 01^{[r]} /$ $\! W(02) \otimes 10^{[r]} /$ $\! 10 \otimes 10^{[r]}$ \\

 $(10,10^{[r]},01)$ $(r \neq 0)$ & $10 \otimes 01^{[r]} /$ $\! W(02) /$ $\! 10 /$ $\! 10^{[r]} \otimes 10 $ & $W(11)^2 /$ $\! W(11)^{[r]} /$ $\! W(11) \otimes 10^{[r]} /$ $\! 10^{[r]} /$ $\! W(11) \otimes 01^{[r]} /$ $\! 01^{[r]}$ \\

 $(10,10^{[r]},10^{[r]})$ $(r \neq 0)$ & $10 \otimes 01^{[r]} /$ $\! 10^{[r]} \otimes 01 /$ $\! W(11)^{[r]} /$ $\! 00$ & $W(11) /$ $\! (W(11)^{[r]})^2 /$ $\! W(20)^{[r]} \otimes 10 /$ $\! 01^{[r]} \otimes 10 /$ $\! W(02)^{[r]} \otimes 01 /$ $\! 10^{[r]} \otimes 01$  \\

 $(10,10^{[r]},01^{[r]})$ $(r \neq 0)$ & $10 \otimes 01^{[r]} /$ $\! 01 \otimes 01^{[r]} /$ $\! W(20)^{[r]} /$ $\! 01^{[r]}$ & $W(11) /$ $\! (W(11)^{[r]})^2 /$ $\! W(11)^{[r]} \otimes 10 /$ $\!10 /$ $\! W(11)^{[r]} \otimes 01 /$ $\! 01$ \\

 $(10,01^{[r]},01^{[r]})$ $(r \neq 0)$ & $10 \otimes 10^{[r]} /$ $\! 01 \otimes 01^{[r]} /$ $\! W(11)^{[r]} /$ $\! 00$  &  $W(11) /$ $\! (W(11)^{[r]})^2 /$ $\! W(02)^{[r]} \otimes 10 /$ $\! 10^{[r]} \otimes 10 /$ $\! W(20)^{[r]} \otimes 01 /$ $\! 01^{[r]} \otimes 01$ \\

 $(10,10^{[r]},10^{[s]})$ \newline $(0<r<s)$ & $10 \otimes 01^{[r]} /$ $\! 01 \otimes 10^{[s]} /$ $\! 10^{[r]} \otimes 01^{[s]}$  & $W(11) /$ $\! W(11)^{[r]} /$ $\! W(11)^{[s]} /$ $\! 10 \otimes 10^{[r]} \otimes 10^{[s]} /$ $\! 01 \otimes 01^{[r]} \otimes 01^{[s]} $ \\

 $(10,10^{[r]},01^{[s]})$ \newline $(0 < r < s)$ & $10 \otimes 01^{[r]} /$ $\! 01 \otimes 01^{[s]} /$ $\! 10^{[r]} \otimes 10^{[s]}$ & $W(11) /$ $\! W(11)^{[r]} /$ $\! W(11)^{[s]} /$ $\! 10 \otimes 10^{[r]} \otimes 01^{[s]} /$ $\! 01 \otimes 01^{[r]} \otimes 10^{[s]}$ \\

 $(10,01^{[r]},10^{[s]})$ \newline $(0 < r < s)$ & $10 \otimes 10^{[r]} /$ $\! 01 \otimes 10^{[s]} /$ $\! 01^{[r]} \otimes 01^{[s]}$ & $W(11) /$ $\! W(11)^{[r]} /$ $\! W(11)^{[s]} /$ $\! 10 \otimes 01^{[r]} \otimes 10^{[s]} /$ $\! 01 \otimes 10^{[r]} \otimes 01^{[s]}$ \\

 $(10,01^{[r]},01^{[s]})$ \newline $(0 < r < s)$ & $10 \otimes 10^{[r]} /$ $\! 01 \otimes 01^{[s]} /$ $\! 01^{[r]} \otimes 10^{[s]}$ & $W(11) /$ $\! W(11)^{[r]} /$ $\! W(11)^{[s]} /$ $\! 10 \otimes 01^{[r]} \otimes 01^{[s]} /$ $\! 01 \otimes 10^{[r]} \otimes 10^{[s]}$ \\

\hline

\end{longtable}

\begin{longtable}{p{0.26\textwidth - 2\tabcolsep}>{\raggedright\arraybackslash}p{0.36\textwidth-2\tabcolsep}>{\raggedright\arraybackslash}p{0.38\textwidth-\tabcolsep}@{}}

\caption{The simple, connected irreducible subgroups of $E_7$ of rank at least $2$. \label{E7tab}} \\

\hline \noalign{\smallskip}

Irreducible subgroup $X$ & Composition factors of $V_{56} \downarrow X$ & Composition factors of $L(E_7) \downarrow X$ \\

\hline \noalign{\smallskip}

$A_7$ & $0100000 /$ $\! 0000010$ & $W(1000001) /$ $\! 0001000$ \\

\hline \noalign{\smallskip}

$D_4$ $(p>2)$ & $0100^2$ & $0100 /$ $\! 2000 /$ $\! 0020 /$ $\! 0002$ \\

\hline \noalign{\smallskip}

\multicolumn{3}{l}{$A_2 \hookrightarrow \bar{A}_2 A_2^{(*)} < \bar{A}_2 A_5$ $(p \neq 2)$  (see \S \ref{a2a2starnotation}) via:} \\ \noalign{\smallskip}

$(10,10)$ $(p > 3)$ & $30^2 /$ $\! 03^2 /$ $\! 11^2$  & $11^4 /$ $\! W(22)^3 /$ $\! 30 /$ $\! 03$ \\

$(10^{[r]},10^{[s]})$ $(rs=0)$ & $10^{[r]} \otimes 20^{[s]} /$ $\! 01^{[r]} \otimes 02^{[s]} /$ $\! W(30)^{[s]} /$ $\! W(03)^{[s]}$ & $W(11)^{[r]} /$ $\! W(11)^{[s]} /$ $\! W(22)^{[s]} /$ $\! 01^{[r]} \otimes 21^{[s]} /$ $\! 10^{[r]} \otimes 12^{[s]}$  \\

$(10,01)$ & $10 /$ $\! 01 /$ $\! W(30) /$ $\! W(03) /$ $\! 21 /$ $\! 12$  & $W(11)^2 /$ $\! 20 /$ $\! 21 /$ $\! W(31) /$ $\! W(22) /$ $\! 02 /$ $\! 12 /$ $\! W(13)$ \\

$(10^{[r]},01^{[s]})$ $(rs=0)$ & $10^{[r]} \otimes 02^{[s]} /$ $\! 01^{[r]} \otimes 20^{[s]} /$ $\! W(30)^{[s]} /$ $\! W(03)^{[s]}$ & $W(11)^{[r]} /$ $\! W(11)^{[s]} /$ $\! W(22)^{[s]} /$ $\! 10^{[r]} \otimes 21^{[s]} /$ $\! 01^{[r]} \otimes 12^{[s]}$ \\

\hline \noalign{\smallskip}

\multicolumn{3}{l}{$G_2 \hookrightarrow G_2 G_2 < G_2 C_3$ $(p=2)$ via:} \\  \noalign{\smallskip}

$(10^{[r]},10^{[s]})$ $(rs=0)$ & $10^{[r]} \otimes 10^{[s]} /$ $\! 20^{[s]} /$ $\! (10^{[s]})^2 /$ $\! 00^2$ & $01^{[r]} /$ $\! (01^{[s]})^2 /$ $\! 10^{[r]} \otimes 01^{[s]} /$ $\! 20^{[s]} /$ $\! 00$ \\

\hline \noalign{\smallskip}

$A_2$ $(p \geq 5)$ & $W(60) /$ $\! W(06)$ & $W(44) /$ $\! 11$ \\

\hline

\end{longtable}

\begin{longtable}{p{0.26\textwidth - 2\tabcolsep}>{\raggedright\arraybackslash}p{0.74\textwidth-\tabcolsep}@{}}

\caption{The simple, connected irreducible subgroups of $E_8$ of rank at least $2$. \label{E8tab}} \\

\hline \noalign{\smallskip}

Irreducible subgroup $X$ & Composition factors of $L(E_8) \downarrow X$ \\

\hline \noalign{\smallskip}

$D_8$ & $W(01000000) /$ $\! 00000010$ \\

\hline \noalign{\smallskip}

$B_7$ & $W(0100000) /$ $\! W(1000000) /$ $\! 0000001$ \\

\hline \noalign{\smallskip}

 $B_4 (\dagger)$ (see \S \ref{b4notation}) & $W(0100) /$ $\! W(0010) /$ $\! W(1001)$ \\

\hline \noalign{\smallskip}

 $B_4 (\ddagger)$ $(p \neq 2)$ (see \S \ref{b4notation})  & $0100 /$ $\! W(2000) /$ $\! 0010^2$ \\ 

\hline \noalign{\smallskip}

 $A_3$ ($p \neq 2$) & $ 101^2 /$ $\! 210 /$ $\! 012 /$ $\! W(111)^2$ \\

\hline \noalign{\smallskip}

 $D_4 \hookrightarrow D_4^2$ via: & \\

 $(1000,1000^{[r]})$ $(r \neq 0)$ & $W(0100) /$ $\! W(0100)^{[r]} /$ $\! 1000 \otimes 1000^{[r]} /$ $\! 0010 \otimes 0010^{[r]} /$ $\! 0001 \otimes 0001^{[r]}$ \\

 $(1000,1000^{[\tau]})$ & $W(0100)^2 /$ $\! 1000 /$ $\! 0010 /$ $\! 0001 /$ $\! W(1010) /$ $\! W(1001) /$ $\! W(0011)$ \\

 $(1000,1000^{[\tau r]})$ $(r \neq 0)$ & $W(0100) /$ $\! W(0100)^{[r]} /$ $\! 1000 \otimes 0010^{[r]} /$ $\! 0010 \otimes 0001^{[r]} /$ $\!  0001 \otimes 1000^{[r]}$ \\

 $(1000,1000^{[\iota r]})$ $(r \neq 0)$ & $W(0100) /$ $\! W(0100)^{[r]} /$ $\! 1000 \otimes 1000^{[r]} /$ $\! 0010 \otimes 0001^{[r]} /$ $\! 0001 \otimes 0010^{[r]}$ \\

\hline \noalign{\smallskip}

 $B_3 \hookrightarrow B_3 B_3 < D_4^2$ via: & \\

 $(100,100^{[r]})$ $(r \neq 0)$ & $W(010) /$ $\! W(100) /$ $\! W(100)^{[r]} /$ $\! W(010)^{[r]} /$ $\! 001 /$ $\! 001^{[r]} /$ $\! W(100) \otimes 001^{[r]} /$ $\! 001 \otimes W(100)^{[r]} /$ $\! 001 \otimes 001^{[r]}$ \\

\hline \noalign{\smallskip}

\multicolumn{2}{l}{$A_2 \hookrightarrow A_2^2 < D_4^2$  $(p \neq 3)$ via:} \\ \noalign{\smallskip}

$(10,10^{[r]})$ $(r \neq 0)$ & $11 /$ $\! W(30) /$ $\! W(03) /$ $\! 11^{[r]} /$ $\! W(30)^{[r]} /$ $\! W(03)^{[r]} /$ $\! (11 \otimes 11^{[r]})^3$ \\

$(10,01^{[r]})$ $(r \neq 0)$ & $11 /$ $\! W(30) /$ $\! W(03) /$ $\! 11^{[r]} /$ $\! W(30)^{[r]} /$ $\! W(03)^{[r]} /$ $\! (11 \otimes 11^{[r]})^3$ \\

\hline \noalign{\smallskip}

\multicolumn{2}{l}{$B_2 \hookrightarrow B_2^2 (\dagger)$ (see \S \ref{b4notation}) via:} \\ \noalign{\smallskip} 

$(10,10^{[r]})$ $(r \neq 0)$ & $W(02) /$ $\! W(02)^{[r]} /$ $\! W(10) \otimes W(02)^{[r]} /$ $\! W(02) \otimes W(10)^{[r]} /$ $\! 01 \otimes W(11)^{[r]} /$ $\! W(11) \otimes 01^{[r]} $\\

$(10,02)$ $(p=2)$ & $10^4 /$ $\! 01^4 /$ $\! 20^4 /$ $\! 02^{6} /$ $\! 21 /$ $\! 12^2 /$ $\! 30 /$ $\! 03 /$ $\! 13 /$ $\! 04 /$ $\! 00^{12}$ \\

$(10,02^{[r]})$ $(p=2, r \neq 0)$ & $10^4 /$ $\! 02^2 /$ $\! (02^{[r]})^4 /$ $\! (20^{[r]})^2 /$ $\! 02 \otimes 02^{[r]} /$ $\! (10 \otimes 02^{[r]})^2 /$ $\! 10 \otimes 20^{[r]} /$ $\! 01 \otimes 12^{[r]} /$ $\! 11 \otimes 10^{[r]} /$ $\! 00^8$ \\

\hline \noalign{\smallskip}

\multicolumn{2}{l}{$B_2 \hookrightarrow B_2^2 (\ddagger)$ $(p \neq 2)$ (see \S \ref{b4notation}) via:} \\ \noalign{\smallskip}

$(10,10^{[r]})$ $(r \neq 0)$ & $02 /$ $\! 02^{[r]} /$ $\! W(20) /$ $\! W(20)^{[r]} /$ $\! (10 \otimes 02^{[r]})^2 /$ $\! (02 \otimes 10^{[r]})^2$\\

\hline \noalign{\smallskip}

\multicolumn{2}{l}{$B_2 \hookrightarrow B_2 B_2 < A_3 D_5$ $(p \geq 3)$ via:} \\ \noalign{\smallskip}

$(10,10)$ & $02^6 /$ $\! 10^4 /$ $\! W(20)^2 /$ $\! W(12)^4 $ \\

 $(10^{[r]},10^{[s]})$ $(rs=0)$ & $02^{[r]} /$ $\! (02^{[s]})^2 /$ $\! 10^{[r]} /$ $\! W(12)^{[s]} /$ $\! 10^{[r]} \otimes 02^{[s]} /$ $\! (01^{[r]} \otimes W(11)^{[s]})^2$ \\

\hline \noalign{\smallskip}

 $B_2 \hookrightarrow B_2^3$ $(p \geq 3)$ via: & \\
 $(10,10^{[r]},10^{[s]})$ \newline $(0<r<s)$ & $02 /$ $\! 10 /$ $\! 10^{[r]} /$ $\! 10^{[s]} /$ $\! 02^{[r]} /$ $\! 02^{[s]} /$ $\! 10 \otimes 10^{[r]} /$ $\! 10 \otimes 10^{[s]} /$ $\! 10^{[r]} \otimes 10^{[s]} /$ $\! (01 \otimes 01^{[r]} \otimes 01^{[s]})^2$ \\

\hline \noalign{\smallskip}

$A_8$ & $W(10000001) /$ $\! 00100000 /$ $\! 00000100$ \\

\hline \noalign{\smallskip}

$A_4 \hookrightarrow A_4^2$ via: & \\

$(1000,1000)$ & $W(1001)^2 /$ $\! 1000 /$ $\! 0001 /$ $\! 0100 /$ $\! 0010 /$ $\! W(1100) /$ $\! W(0011) /$ $\!  W(1010) /$ $\! W(0101)$ \\

$(1000,1000^{[r]})$ $(r \neq 0)$ & $W(1001) /$ $\! W(1001)^{[r]} /$ $\! 1000 \otimes 0100^{[r]} /$ $\! 0100 \otimes 0001^{[r]} /$ $\! 0001 \otimes 0010^{[r]} /$ $\! 0010 \otimes 1000^{[r]}$ \\

$(1000,0001^{[r]})$ $(r \neq 0)$ & $W(1001) /$ $\! W(1001)^{[r]} /$ $\! 1000 \otimes 0010^{[r]} /$ $\! 0100 \otimes 1000^{[r]} /$ $\! 0001 \otimes 0100^{[r]} /$ $\! 0010 \otimes 0001^{[r]}$  \\

\hline \noalign{\smallskip}

$A_2 \hookrightarrow A_2^4$ via: & \\

$(10^{[r]},10,10,01)$ \newline $(r \neq 0)$ & $W(11)^3 /$ $\! W(11)^{[r]} /$ $\! 10^2 /$ $\! 01^2 /$ $\! W(20) /$ $\! W(02) /$ $\! W(21) /$ $\! W(12) /$ $\! 10^{[r]} /$ $\! 01^{[r]} /$ $\! 10 \otimes 10^{[r]} /$ $\! 01 \otimes 01^{[r]} /$ $\! 10 \otimes 01^{[r]} /$ $\! 01 \otimes 10^{[r]} /$ $\! W(20) \otimes 10^{[r]} /$ $\! W(02) \otimes 01^{[r]} /$ $\! W(20) \otimes 01^{[r]} /$ $\! W(02) \otimes 10^{[r]} /$ $\! W(11) \otimes 10^{[r]} /$ $\! W(11) \otimes 01^{[r]}$  \\

$(10,10^{[r]},10^{[r]},01^{[r]})$ \newline $(r \neq 0)$ & $W(11) /$ $\! (W(11)^{[r]})^3 /$ $\! 10 /$ $\! 01 /$ $\! W(20)^{[r]} /$ $\! W(02)^{[r]} /$ $\! W(21)^{[r]} /$ $\!W(12)^{[r]} /$ $\! (10^{[r]})^2 /$ $\! (01^{[r]})^2 /$ $\! 10 \otimes 10^{[r]} /$ $\! 01 \otimes 01^{[r]} /$ $\! 10 \otimes 01^{[r]} /$ $\! 01 \otimes 10^{[r]} /$ $\! W(20)^{[r]} \otimes 10 /$ $\! W(02)^{[r]} \otimes 01 /$ $\! W(20)^{[r]} \otimes 01 /$ $\! W(02)^{[r]} \otimes 10 /$ $\! W(11)^{[r]} \otimes 10 /$ $\! W(11)^{[r]} \otimes 01$ \\

$(10,10,10^{[r]},10^{[r]})$ \newline $(r \neq 0)$ & $W(11)^2 /$ $\! (W(11)^{[r]})^2 /$ $\! 10 \otimes W(20)^{[r]} /$ $\! 10 \otimes 01^{[r]} /$ $\! 01 \otimes W(02)^{[r]} /$  $\! 01 \otimes 10^{[r]} /$ $\! 10 \otimes W(11)^{[r]} /$ $\! 10 /$ $\! 01 \otimes W(11)^{[r]} /$  $\! 01 /$  $\! W(20) \otimes 01^{[r]} /$ $\! 01 \otimes 01^{[r]} /$  $\! W(02) \otimes 10^{[r]} /$ $\! 10 \otimes 10^{[r]} /$  $\! W(11) \otimes 10^{[r]} /$ $\! 10^{[r]} /$ $\! W(11) \otimes 01^{[r]} /$ $\! 01^{[r]}$  \\

$(10,10,10^{[r]},01^{[r]})$ \newline $(r \neq 0)$ & $W(11)^2 /$ $\! (W(11)^{[r]})^2 /$ $\! 10 \otimes W(11)^{[r]} /$ $\! 10 /$ $\! 01 \otimes W(11)^{[r]} /$  $\! 01 /$ $\! 10 \otimes W(02)^{[r]} /$ $\! 10 \otimes 10^{[r]} /$ $\! 01 \otimes W(20)^{[r]} /$ $\! 01 \otimes 01^{[r]} /$ $\! W(20) \otimes 10^{[r]} /$   $\! 01 \otimes 10^{[r]} /$ $\! W(02) \otimes 01^{[r]} /$ $\! 10 \otimes 01^{[r]} /$ $\! W(11) \otimes 10^{[r]} /$ $\! 10^{[r]} /$ $\! W(11) \otimes 01^{[r]} /$ $\! 01^{[r]}$  \\

$(10,10,10^{[r]},10^{[s]})$ \newline $(0 < r < s)$ & $W(11)^2 /$ $\! W(11)^{[r]} /$ $\! W(11)^{[s]} /$  $\!10 \otimes 10^{[r]} \otimes 10^{[s]} /$  $\!01 \otimes 01^{[r]} \otimes 01^{[s]} /$  $\!10 \otimes 01^{[r]} \otimes 10^{[s]} /$  $\!01 \otimes 10^{[r]} \otimes 01^{[s]} /$  $\! W(20) \otimes 01^{[s]} /$  $\! 01 \otimes 01^{[s]} /$  $\! W(02) \otimes 10^{[s]} /$ $\! 10 \otimes 10^{[s]} /$ $\! W(11) \otimes 10^{[r]} /$ $\! 10^{[r]} /$ $\! W(11) \otimes 01^{[r]} /$ $\! 01^{[r]}$ \\

$(10,01,10^{[r]},10^{[s]})$ \newline $(0 < r < s)$ & $W(11)^2 /$ $\! W(11)^{[r]} /$ $\! W(11)^{[s]} /$ $\! 01 \otimes 10^{[r]} \otimes 10^{[s]} /$ $\! 10 \otimes 01^{[r]} \otimes 01^{[s]} /$ $\!10 \otimes 01^{[r]} \otimes 10^{[s]} /$ $\! 01 \otimes 10^{[r]} \otimes 01^{[s]} /$ $\! W(11) \otimes 01^{[s]} /$ $\! 01^{[s]} /$ $\! W(11) \otimes 10^{[s]} /$ $\! 10^{[s]} /$ $\! W(20) \otimes 10^{[r]} /$ $\! 01 \otimes 10^{[r]} /$ $\! W(02) \otimes 01^{[r]} /$ $\! 10 \otimes 01^{[r]}$ \\

$(01,10,10^{[r]},10^{[s]})$ \newline $(0 < r < s)$ & $W(11)^2 /$ $\! W(11)^{[r]} /$ $\! W(11)^{[s]} /$ $\! 10 \otimes 10^{[r]} \otimes 10^{[s]} /$ $\! 01 \otimes 01^{[r]} \otimes 01^{[s]} /$ $\!01 \otimes 01^{[r]} \otimes 10^{[s]} /$ $\! 10 \otimes 10^{[r]} \otimes 01^{[s]} /$ $\! W(11) \otimes 01^{[s]} /$ $\! 01^{[s]} /$ $\! W(11) \otimes 10^{[s]} /$ $\! 10^{[s]} /$ $\! W(02) \otimes 10^{[r]} /$ $\! 10 \otimes 10^{[r]} /$ $\! W(20) \otimes 01^{[r]} /$ $\! 01 \otimes 01^{[r]}$   \\

$(01,01,10^{[r]},10^{[s]})$ \newline $(0 < r < s)$ & $W(11)^2 /$ $\! W(11)^{[r]} /$ $\! W(11)^{[s]} /$ $\! 01 \otimes 10^{[r]} \otimes 10^{[s]} /$ $\! 10 \otimes 01^{[r]} \otimes 01^{[s]} /$ $\!01 \otimes 01^{[r]} \otimes 10^{[s]} /$ $\! 10 \otimes 10^{[r]} \otimes 01^{[s]} /$ $\! W(02) \otimes 01^{[s]} /$ $\! 10 \otimes 01^{[s]} /$ $\! W(20) \otimes 10^{[s]} /$ $\! 01 \otimes 10^{[s]} /$ $\! W(11) \otimes 10^{[r]} /$ $\! 10^{[r]} /$ $\!W(11) \otimes 01^{[r]} /$ $\! 01^{[r]} $ \\

$(10,10^{[r]},10^{[r]},10^{[s]})$ \newline $(0 < r < s)$ & $W(11) /$ $\! (W(11)^{[r]})^2 /$ $\! W(11)^{[s]} /$ $\!W(20)^{[r]} \otimes 10^{[s]} /$ $\!01^{[r]} \otimes 10^{[s]} /$  $\!W(02)^{[r]} \otimes 01^{[s]} /$ $\! 10^{[r]} \otimes 01^{[s]} /$ $\!10 \otimes 01^{[r]} \otimes 10^{[s]} /$  $\!01 \otimes 10^{[r]} \otimes 01^{[s]} /$  $\!10 \otimes 10^{[r]} \otimes 01^{[s]} /$  $\!01 \otimes 01^{[r]} \otimes 10^{[s]} /$ $\! 10 \otimes W(11)^{[r]} /$ $\! 10 /$ $\! 01 \otimes W(11)^{[r]} /$ $\! 01 $ \\

$(10,10^{[r]},01^{[r]},10^{[s]})$ \newline $(0 < r < s)$ & $W(11) /$ $\! (W(11)^{[r]})^2 /$ $\! W(11)^{[s]} /$ $\!W(11)^{[r]} \otimes 10^{[s]} /$ $\! 10^{[s]} /$  $\!W(11)^{[r]} \otimes 01^{[s]} /$ $\! 01^{[s]} /$ $\!10 \otimes 10^{[r]} \otimes 10^{[s]} /$  $\!01 \otimes 01^{[r]} \otimes 01^{[s]} /$  $\!10 \otimes 10^{[r]} \otimes 01^{[s]} /$  $\!01 \otimes 01^{[r]} \otimes 10^{[s]} /$ $\! 10 \otimes W(02)^{[r]} /$ $\! 10 \otimes 10^{[r]} /$ $\! 01 \otimes W(20)^{[r]} /$ $\! 01 \otimes 01^{[r]} $ \\

$(10,01^{[r]},10^{[r]},10^{[s]})$ \newline $(0 < r < s)$ & $W(11) /$ $\! (W(11)^{[r]})^2 /$ $\! W(11)^{[s]} /$ $\!W(11)^{[r]} \otimes 10^{[s]} /$ $\!10^{[s]} /$  $\!W(11)^{[r]} \otimes 01^{[s]} /$ $\! 01^{[s]} /$ $\!10 \otimes 01^{[r]} \otimes 10^{[s]} /$  $\!01 \otimes 10^{[r]} \otimes 01^{[s]} /$  $\!10 \otimes 01^{[r]} \otimes 01^{[s]} /$  $\!01 \otimes 10^{[r]} \otimes 10^{[s]} /$ $\! 10 \otimes W(20)^{[r]} /$ $\! 10 \otimes 01^{[r]} /$ $\! 01 \otimes W(02)^{[r]} /$ $\! 01 \otimes 10^{[r]} $ \\

$(10,01^{[r]},01^{[r]},10^{[s]})$ \newline $(0 < r < s)$ & $W(11) /$ $\! (W(11)^{[r]})^2 /$ $\! W(11)^{[s]} /$ $\!W(02)^{[r]} \otimes 10^{[s]} /$ $\!10^{[r]} \otimes 10^{[s]} /$  $\!W(20)^{[r]} \otimes 01^{[s]} /$ $\! 01^{[r]} \otimes 01^{[s]} /$ $\!10 \otimes 10^{[r]} \otimes 10^{[s]} /$  $\!01 \otimes 01^{[r]} \otimes 01^{[s]} /$  $\!10 \otimes 01^{[r]} \otimes 01^{[s]} /$  $\!01 \otimes 10^{[r]} \otimes 10^{[s]} /$ $\! 10 \otimes W(11)^{[r]} /$ $\! 10 /$ $\! 01 \otimes W(11)^{[r]} /$ $\! 01 $ \\

$(10,10^{[r]},10^{[s]},10^{[s]})$ \newline $(0 < r < s)$ & $W(11) /$ $\! W(11)^{[r]} /$ $\! (W(11)^{[s]})^2 /$ $\! 10^{[r]} \otimes W(20)^{[s]} /$ $\! 10^{[r]} \otimes 01^{[s]} /$ $\! 01^{[r]} \otimes W(02)^{[s]} /$ $\! 01^{[r]} \otimes 10^{[s]} /$ $\! 10 \otimes W(11)^{[s]} /$ $\! 10 /$ $\! 01 \otimes W(11)^{[s]}/$ $\! 01 /$ $\! 10 \otimes 10^{[r]} \otimes 01^{[s]} /$ $\! 01 \otimes 01^{[r]} \otimes 10^{[s]} /$ $\! 10 \otimes 01^{[r]} \otimes 10^{[s]} /$ $\! 01 \otimes 10^{[r]} \otimes 01^{[s]} $ \\ 

$(10,10^{[r]},10^{[s]},01^{[s]})$ \newline $(0 < r < s)$ & $W(11) /$ $\! W(11)^{[r]} /$ $\! (W(11)^{[s]})^2 /$ $\! 10^{[r]} \otimes W(11)^{[s]} /$ $\! 10^{[r]} /$ $\! 01^{[r]} \otimes W(11)^{[s]} /$ $\! 01^{[r]}/$ $\! 10 \otimes W(02)^{[s]} /$ $\! 10 \otimes 10^{[s]}/$ $\! 01 \otimes W(20)^{[s]}/$ $\! 01 \otimes 01^{[s]} /$ $\! 10 \otimes 10^{[r]} \otimes 10^{[s]} /$ $\! 01 \otimes 01^{[r]} \otimes 01^{[s]} /$ $\! 10 \otimes 01^{[r]} \otimes 10^{[s]} /$ $\! 01 \otimes 10^{[r]} \otimes 01^{[s]} $ \\ 

$(10,10^{[r]},01^{[s]},10^{[s]})$ \newline $(0 < r < s)$ & $W(11) /$ $\! W(11)^{[r]} /$ $\! (W(11)^{[s]})^2 /$ $\! 10^{[r]} \otimes W(11)^{[s]} /$ $\! 10^{[r]} /$ $\! 01^{[r]} \otimes W(11)^{[s]} /$ $\! 01^{[r]}/$ $\! 10 \otimes W(20)^{[s]} /$ $\! 10 \otimes 01^{[s]}/$ $\! 01 \otimes W(02)^{[s]}/$ $\! 01 \otimes 10^{[s]} /$ $\! 10 \otimes 10^{[r]} \otimes 01^{[s]} /$ $\! 01 \otimes 01^{[r]} \otimes 10^{[s]} /$ $\! 10 \otimes 01^{[r]} \otimes 01^{[s]} /$ $\! 01 \otimes 10^{[r]} \otimes 10^{[s]} $ \\ 

$(10,10^{[r]},01^{[s]},01^{[s]})$ \newline $(0 < r < s)$ & $W(11) /$ $\! W(11)^{[r]} /$ $\! (W(11)^{[s]})^2 /$ $\! 10^{[r]} \otimes W(02)^{[s]} /$ $\! 10^{[r]} \otimes 10^{[s]} /$ $\! 01^{[r]} \otimes W(20)^{[s]} /$ $\! 01^{[r]} \otimes 01^{[s]} /$ $\! 10 \otimes W(11)^{[s]} /$ $\! 10 /$ $\! 01 \otimes W(11)^{[s]}/$ $\! 01 /$ $\! 10 \otimes 10^{[r]} \otimes 10^{[s]} /$ $\! 01 \otimes 01^{[r]} \otimes 01^{[s]} /$ $\! 10 \otimes 01^{[r]} \otimes 01^{[s]} /$ $\! 01 \otimes 10^{[r]} \otimes 10^{[s]} $ \\ 

$(10,10^{[r]},10^{[s]},10^{[t]})$ \newline $(0 < r < s < t)$ & $W(11) /$ $\! W(11)^{[r]} /$ $\! W(11)^{[s]} /$ $\! W(11)^{[t]} /$ $\! 10^{[r]} \otimes 10^{[s]} \otimes 10^{[t]} /$ $\! 01^{[r]} \otimes 01^{[s]} \otimes 01^{[t]} /$ $\! 10 \otimes 01^{[s]} \otimes 10^{[t]} /$ $\! 01 \otimes 10^{[s]} \otimes 01^{[t]} /$ $\! 10 \otimes 10^{[r]} \otimes 01^{[t]} /$ $\! 01 \otimes 01^{[r]} \otimes 10^{[t]} /$ $\!10 \otimes 01^{[r]} \otimes 10^{[s]} /$ $\! 01 \otimes 10^{[r]} \otimes 01^{[s]}$ \\

$(10,10^{[r]},10^{[s]},01^{[t]})$ \newline $(0 < r < s < t)$ & $W(11) /$ $\! W(11)^{[r]} /$ $\! W(11)^{[s]} /$ $\! W(11)^{[t]} /$ $\! 10^{[r]} \otimes 10^{[s]} \otimes 01^{[t]} /$ $\! 01^{[r]} \otimes 01^{[s]} \otimes 10^{[t]} /$ $\! 10 \otimes 01^{[s]} \otimes 01^{[t]} /$ $\! 01 \otimes 10^{[s]} \otimes 10^{[t]} /$ $\! 10 \otimes 10^{[r]} \otimes 10^{[t]} /$ $\! 01 \otimes 01^{[r]} \otimes 01^{[t]} /$ $\!10 \otimes 01^{[r]} \otimes 10^{[s]} /$ $\! 01 \otimes 10^{[r]} \otimes 01^{[s]}$ \\

$(10,10^{[r]},01^{[s]},10^{[t]})$ \newline $(0 < r < s < t)$ & $W(11) /$ $\! W(11)^{[r]} /$ $\! W(11)^{[s]} /$ $\! W(11)^{[t]} /$ $\! 10^{[r]} \otimes 01^{[s]} \otimes 10^{[t]} /$ $\! 01^{[r]} \otimes 10^{[s]} \otimes 01^{[t]} /$ $\! 10 \otimes 10^{[s]} \otimes 10^{[t]} /$ $\! 01 \otimes 01^{[s]} \otimes 01^{[t]} /$ $\! 10 \otimes 10^{[r]} \otimes 01^{[t]} /$ $\! 01 \otimes 01^{[r]} \otimes 10^{[t]} /$ $\!10 \otimes 01^{[r]} \otimes 01^{[s]} /$ $\! 01 \otimes 10^{[r]} \otimes 10^{[s]}$ \\

$(10,10^{[r]},01^{[s]},01^{[t]})$ \newline $(0 < r < s < t)$ &  $W(11) /$ $\! W(11)^{[r]} /$ $\! W(11)^{[s]} /$ $\! W(11)^{[t]} /$ $\! 10^{[r]} \otimes 01^{[s]} \otimes 01^{[t]} /$ $\! 01^{[r]} \otimes 10^{[s]} \otimes 10^{[t]} /$ $\! 10 \otimes 10^{[s]} \otimes 01^{[t]} /$ $\! 01 \otimes 01^{[s]} \otimes 10^{[t]} /$ $\! 10 \otimes 10^{[r]} \otimes 10^{[t]} /$ $\! 01 \otimes 01^{[r]} \otimes 01^{[t]} /$ $\!10 \otimes 01^{[r]} \otimes 01^{[s]} /$ $\! 01 \otimes 10^{[r]} \otimes 10^{[s]}$ \\

$(10,01^{[r]},10^{[s]},10^{[t]})$ \newline $(0 < r < s < t)$ & $W(11) /$ $\! W(11)^{[r]} /$ $\! W(11)^{[s]} /$ $\! W(11)^{[t]} /$ $\! 01^{[r]} \otimes 10^{[s]} \otimes 10^{[t]} /$ $\! 10^{[r]} \otimes 01^{[s]} \otimes 01^{[t]} /$ $\! 10 \otimes 01^{[s]} \otimes 10^{[t]} /$ $\! 01 \otimes 10^{[s]} \otimes 01^{[t]} /$ $\! 10 \otimes 01^{[r]} \otimes 01^{[t]} /$ $\! 01 \otimes 10^{[r]} \otimes 10^{[t]} /$ $\!10 \otimes 10^{[r]} \otimes 10^{[s]} /$ $\! 01 \otimes 01^{[r]} \otimes 01^{[s]}$ \\

$(10,01^{[r]},10^{[s]},01^{[t]})$ \newline $(0 < r < s < t)$ & $W(11) /$ $\! W(11)^{[r]} /$ $\! W(11)^{[s]} /$ $\! W(11)^{[t]} /$ $\! 01^{[r]} \otimes 10^{[s]} \otimes 01^{[t]} /$ $\! 10^{[r]} \otimes 01^{[s]} \otimes 10^{[t]} /$ $\! 10 \otimes 01^{[s]} \otimes 01^{[t]} /$ $\! 01 \otimes 10^{[s]} \otimes 10^{[t]} /$ $\! 10 \otimes 01^{[r]} \otimes 10^{[t]} /$ $\! 01 \otimes 10^{[r]} \otimes 01^{[t]} /$ $\!10 \otimes 10^{[r]} \otimes 10^{[s]} /$ $\! 01 \otimes 01^{[r]} \otimes 01^{[s]}$ \\

$(10,01^{[r]},01^{[s]},10^{[t]})$ \newline $(0 < r < s < t)$ & $W(11) /$ $\! W(11)^{[r]} /$ $\! W(11)^{[s]} /$ $\! W(11)^{[t]} /$ $\! 01^{[r]} \otimes 01^{[s]} \otimes 10^{[t]} /$ $\! 10^{[r]} \otimes 10^{[s]} \otimes 01^{[t]} /$ $\! 10 \otimes 10^{[s]} \otimes 10^{[t]} /$ $\! 01 \otimes 01^{[s]} \otimes 01^{[t]} /$ $\! 10 \otimes 01^{[r]} \otimes 01^{[t]} /$ $\! 01 \otimes 10^{[r]} \otimes 10^{[t]} /$ $\!10 \otimes 10^{[r]} \otimes 01^{[s]} /$ $\! 01 \otimes 01^{[r]} \otimes 10^{[s]}$ \\

$(10,01^{[r]},01^{[s]},01^{[t]})$ \newline $(0 < r < s < t)$ &  $W(11) /$ $\! W(11)^{[r]} /$ $\! W(11)^{[s]} /$ $\! W(11)^{[t]} /$ $\! 01^{[r]} \otimes 01^{[s]} \otimes 01^{[t]} /$ $\! 10^{[r]} \otimes 10^{[s]} \otimes 10^{[t]} /$ $\! 10 \otimes 10^{[s]} \otimes 01^{[t]} /$ $\! 01 \otimes 01^{[s]} \otimes 10^{[t]} /$ $\! 10 \otimes 01^{[r]} \otimes 10^{[t]} /$ $\! 01 \otimes 10^{[r]} \otimes 01^{[t]} /$ $\!10 \otimes 10^{[r]} \otimes 01^{[s]} /$ $\! 01 \otimes 01^{[r]} \otimes 10^{[s]}$ \\

\hline \noalign{\smallskip}

\multicolumn{2}{l}{$A_2 \hookrightarrow \bar{A}_2 A_2 < \bar{A}_2 E_6$ $(p \geq 3)$ (see \S \ref{bara2y}) via:} \\ \noalign{\smallskip}

$(10,10)$ & $W(11)^2 /$ $\! W(41) /$ $\! W(14) /$ $\! W(32) /$ $\! W(23) /$ $\! W(13) /$ $\! W(31) /$ $\! W(21) /$ $\!  W(12)$ \\

$(10^{[r]},10^{[s]})$ $(rs=0)$ & $W(11)^{[r]} /$ $\! W(11)^{[s]} /$ $\! W(41)^{[s]} /$ $\! W(14)^{[s]} /$ $\! 10^{[r]} \otimes W(22)^{[s]} /$ $\! 01^{[r]} \otimes W(22)^{[s]}$ \\

\hline \noalign{\smallskip}

\multicolumn{2}{l}{$A_2 \hookrightarrow \bar{A}_2 A_2 \tilde{A}_2  < \bar{A}_2 A_2 G_2$ $(p=3)$ (see \S \ref{bara2a2tildea2}) via:} \\ \noalign{\smallskip}

$(10^{[r]},10^{[s]},10^{[t]})$ \newline $(rst=0)$ & $11^{[r]} /$ $\! 11^{[s]} /$ $\! (11^{[t]})^2 /$ $\! 11^{[s]} \otimes 11^{[t]} /$ $\! 30^{[t]} /$ $\! 03^{[t]} /$ $\! 10^{[r]} \otimes 20^{[s]} /$ $\!  01^{[r]} \otimes 02^{[s]} /$ $\! 10^{[r]} \otimes 01^{[s]} \otimes 11^{[t]} /$ $\! 01^{[r]} \otimes 10^{[s]} \otimes 11^{[t]} /$ $\! 00^3$ \\

$(10^{[r]},10,10^{[r]})$ $(r \neq 0)$ & $(11^{[r]})^3 /$ $\! 11 /$ $\! 11 \otimes 11^{[r]} /$ $\! 30^{[r]} /$ $\! 03^{[r]} /$ $\! 10^{[r]} \otimes 20 /$ $\! 01^{[r]} \otimes 02 /$ $\!  21^{[r]} \otimes 01 /$ $\! 02^{[r]} \otimes 01 /$ $\! 12^{[r]} \otimes 10 /$ $\! 20^{[r]} \otimes 10 /$ $\! 00^3$ \\

$(10,10^{[r]},10)$ $(r \neq 0)$ & $11^3 /$ $\! 11^{[r]} /$ $\! 11 \otimes 11^{[r]} /$ $\! 30 /$ $\! 03 /$ $\! 10 \otimes 20^{[r]} /$ $\! 01 \otimes 02^{[r]} /$ $\! 21 \otimes 01^{[r]} /$ $\! 02 \otimes 01^{[r]} /$ $\! 12 \otimes 10^{[r]} /$ $\! 20 \otimes 10^{[r]} /$ $\! 00^3$ \\

$(10^{[r]},01^{[s]},10^{[t]})$ \newline $(rst=0)$ & $11^{[r]} /$ $\! 11^{[s]} /$ $\! (11^{[t]})^2 /$ $\! 11^{[s]} \otimes 11^{[t]} /$ $\! 30^{[t]} /$ $\! 03^{[t]} /$ $\! 10^{[r]} \otimes 02^{[s]} /$ $\! 01^{[r]} \otimes 20^{[s]} /$ $\! 10^{[r]} \otimes 10^{[s]} \otimes 11^{[t]} /$ $\! 01^{[r]} \otimes 01^{[s]} \otimes 11^{[t]} /$ $\! 00^3$ \\

$(10^{[r]},01^{[r]},10)$ $(r \neq 0)$ & $(11^{[r]})^2 /$ $\! 11^2 /$ $\! 11 \otimes 11^{[r]} /$ $\! 30 /$ $\! 03 /$ $\! 12^{[r]} /$ $\! 01^{[r]} /$ $\! 21^{[r]} /$ $\! 10^{[r]} /$ $\! 20^{[r]} \otimes 11 /$ $\! 01^{[r]} \otimes 11 /$ $\! 02^{[r]} \otimes 11 /$ $\! 10^{[r]} \otimes 11 /$ $\! 00^3$ \\

$(10,01,10^{[r]})$ $(r \neq 0)$ & $11^2 /$ $\! (11^{[r]})^2 /$ $\! 11 \otimes 11^{[r]} /$ $\! 30^{[r]} /$ $\! 03^{[r]} /$ $\! 12 /$ $\! 01 /$ $\! 21 /$ $\! 10 /$ $\! 20 \otimes 11^{[r]} /$ $\! 01 \otimes 11^{[r]} /$ $\! 02 \otimes 11^{[r]} /$ $\! 10 \otimes 11^{[r]} /$ $\! 00^3$ \\

$(10^{[r]},01,10^{[r]})$ $(r \neq 0)$ & $(11^{[r]})^3 /$ $\! 11 /$ $\! 11 \otimes 11^{[r]} /$ $\! 30^{[r]} /$ $\! 03^{[r]} /$ $\! 10^{[r]} \otimes 02 /$ $\! 01^{[r]} \otimes 20 /$ $\! 21^{[r]} \otimes 10 /$ $\! 02^{[r]} \otimes 10 /$ $\! 12^{[r]} \otimes 01 /$ $\! 20^{[r]} \otimes 01 /$ $\! 00^3$ \\

$(10,01^{[r]},10)$ $(r \neq 0)$ & $11^3 /$ $\! 11^{[r]} /$ $\! 11 \otimes 11^{[r]} /$ $\! 30 /$ $\! 03 /$ $\! 10 \otimes 02^{[r]} /$ $\! 01 \otimes 20^{[r]} /$ $\! 21 \otimes 10^{[r]} /$ $\! 02 \otimes 10^{[r]} /$ $\! 12 \otimes 01^{[r]} /$ $\! 20 \otimes 01^{[r]} /$ $\! 00^3$ \\

\hline \noalign{\smallskip}

\multicolumn{2}{l}{$G_2 \hookrightarrow G_2 G_2 < G_2 F_4$  $(p=7)$ via:} \\ \noalign{\smallskip}

$(10^{[r]}, 10^{[s]})$ $(rs=0)$ & $01^{[r]} /$ $\! 10^{[r]} \otimes 20^{[s]} /$ $\! 01^{[s]} /$ $\! 11^{[s]}$  \\

\hline \noalign{\smallskip}

$B_2$ $(p \geq 5)$ & $02 /$ $\! W(06) /$ $\! W(32)$ \\

\hline

\end{longtable}

\subsection*{Acknowledgements} 

This paper was prepared as part of the author's Ph.D. qualification under the supervision of \linebreak Prof.\ M.\ Liebeck, with financial support from the EPSRC. The author would like to thank \linebreak Prof.\ M.\ Liebeck for suggesting the problem and his help in solving it. He would also like to thank Dr\ D.\ Stewart and Dr\ A.\ Litterick for many helpful conversations along the way. Finally, he would like to thank the anonymous referee for many helpful suggestions.

\begin{appendices} 
\section{Levi subgroups and reducible Weyl modules} \label{app}

We give tables of composition factors for Levi subgroups with each simple factor of rank at least $2$, for $G= E_6, E_7$ and $E_8$. If $L'$ is simple then these can be found in \cite[Tables 8.1--8.3, 8.6, 8.7]{LS3}. If $L'$ is not simple then the composition factors can be deduced from those of a maximal subsystem subgroup containing $L'$. We also give the composition factors of the reducible Weyl modules (with at least one trivial composition factor) which appear in Tables \ref{G2tab}--\ref{levie8}.

\pagebreak

\begin{longtable}{p{0.1\textwidth - 2\tabcolsep}>{\raggedright\arraybackslash}p{0.3\textwidth-2\tabcolsep}>{\raggedright\arraybackslash}p{0.6\textwidth-\tabcolsep}@{}}

\caption{The composition factors for the action of Levi subgroups (with no rank 1 simple factors) of $E_6$ on $V_{27}$ and $L(E_6)$. \label{levie6}} \\

\hline \noalign{\smallskip}

Levi $L'$ & Comp. factors of $V_{27} \downarrow L'$ & Comp. factors of $L(E_6) \downarrow L'$ \\

\hline \noalign{\smallskip}

$D_5$ & $\lambda_1 /$ $\! \lambda_4 /$ $\! 0 $ & $W(\lambda_2) /$ $\! \lambda_4 /$ $\! \lambda_5 /$ $\! 0$ \\

$D_4$ & $1000 /$ $\! 0010 /$ $\! 0001 /$ $\! 0000^3$  & $W(0100) /$ $\! 1000^2 /$ $\! 0010^2 /$ $\! 0001^2 /$ $\! 0000^2$ \\

$A_5$ & $\lambda_1^2 /$ $\! \lambda_4$ & $W(\lambda_1 + \lambda_5) /$ $\! \lambda_3^2 /$ $\! 0^3$ \\

$A_4$ & $1000^2 /$ $\! 0010 /$ $\! 0001 /$ $\! 0000^2$ & $W(1001) /$ $\! 1000 /$ $\! 0100^2 /$ $\! 0010^2 /$ $\! 0001 /$ $\! 0000^4$ \\

$A_3$ & $100^2 /$ $\! 001^2 /$ $\! 010 /$ $\! 000^5$ & $W(101) /$ $\! 100^4 /$ $\! 001^4 /$ $\! 010^4 /$ $\! 000^7$ \\

$A_2$ & $10^3 /$ $\! 01^3 /$ $\! 00^9$ & $W(11) /$ $\! 10^9 /$ $\! 01^9 /$ $\! 00^{16}$ \\

$A_2 A_2$ & $(10,01) /$ $\! (00,10)^3 /$ $\! (01,00)^3$ & $(W(11),00) /$ $\! (00,W(11)) /$ $\! (10,10)^3 /$ $\! (01,01)^3 /$ $\! (00,00)^8$  \\

\hline

\end{longtable}

\begin{longtable}{p{0.1\textwidth - 2\tabcolsep}>{\raggedright\arraybackslash}p{0.38\textwidth-2\tabcolsep}>{\raggedright\arraybackslash}p{0.52\textwidth-\tabcolsep}@{}}

\caption{The composition factors for the action of Levi subgroups (with no rank 1 simple factors) of $E_7$ on $V_{56}$ and $L(E_7)$. \label{levie7}} \\

\hline \noalign{\smallskip}

Levi $L'$ & Comp. factors of $V_{56} \downarrow L'$ & Comp. factors of $L(E_7) \downarrow L'$ \\

\hline \noalign{\smallskip}

$E_6$ & $\lambda_1 /$ $\! \lambda_6 /$ $\! 0^2$ & $W(\lambda_2) /$ $\! \lambda_1 /$ $\! \lambda_6 /$ $\! 0$  \\

$D_6$ & $\lambda_1^2 /$ $\! \lambda_5$ & $W(\lambda_2) /$ $\! \lambda_6^2 /$ $\! 0^3$ \\

$D_5$ & $\lambda_1^2 /$ $\! \lambda_4 /$ $\! \lambda_5 /$ $\! 0^4$ & $W(\lambda_2) /$ $\! \lambda_1^2 /$ $\! \lambda_4^2 /$ $\! \lambda_5^2 /$ $\! 0^4$  \\

$D_4$ & $1000^2 /$ $\! 0010^2 /$ $\! 0001^2 /$ $\! 0000^8$ & $W(0100) /$ $\! 1000^4 /$ $\! 0010^4 /$ $\! 0001^4 /$ $\! 0000^9$ \\

$A_6$ & $\lambda_1 /$ $\! \lambda_2 /$ $\! \lambda_5 /$ $\! \lambda_6$ & $W(\lambda_1 + \lambda_6) /$ $\! \lambda_1 /$ $\! \lambda_3 /$ $\! \lambda_4 /$ $\! \lambda_6 /$ $\! 0$  \\

$A_5$ & $\lambda_1^3 /$ $\! \lambda_3 /$ $\! \lambda_5^3$  & $W(\lambda_1 + \lambda_5) /$ $\! \lambda_2^3 /$ $\! \lambda_4^3 /$ $\! 0^8$   \\

$A_5'$ & $\lambda_1^2 /$ $\! \lambda_2 /$ $\! \lambda_4 /$ $\! \lambda_5^2 /$ $\! 0^2$ & $W(\lambda_1 + \lambda_5) /$ $\! \lambda_1^2 /$ $\! \lambda_2 /$ $\! \lambda_3^2 /$ $\! \lambda_4 /$ $\! \lambda_5^2 /$ $\! 0^4$ \\
 
$A_4$ & $ 1000^3 /$ $\! 0100 /$ $\! 0010 /$ $\! 0001^3 /$ $\! 0000^6 $ & $W(1001) /$ $\! 1000^4 /$ $\! 0100^3 /$ $\! 0010^3 /$ $\! 0001^4 /$ $\! 0000^9$  \\

$A_3$ & $100^4 /$ $\! 001^4 /$ $\! 010^2 /$ $\! 000^{12}$ & $W(101) /$ $\! 100^8 /$ $\! 001^8 /$ $\! 010^{6} /$ $\! 000^{18}$  \\

$A_2$ & $10^6 /$ $\! 01^6 /$ $\! 00^{20}$ & $W(11) /$ $\! 10^{15} /$ $\! 01^{15} /$ $\! 00^{35}$ \\

$A_4 A_2$ & $(1000,10) /$ $\! (0000,10) /$ $\! (0001,01) /$ $\! (0000,01) /$ $\! (0100,00) /$ $\! (0010,00)$ & $(W(1001),00) /$ $\! (1000,00) /$ $\! (0001,00) /$ $\! (0000,W(11)) /$ $\! (0010,10) /$ $\! (0001,10) /$ $\! (0100,01) /$ $\! (1000,01) /$ $\! (0000,00)$  \\

$A_3 A_2$ & $(100,10) /$ $\! (000,10)^2 /$ $\! (001,01) /$ $\! (000,01)^2 /$ $\! (010,00)^2 /$ $\! (100,00) /$ $\! (001,00)$ & $(W(101),00) /$ $\! (100,00)^2 /$ $\! (001,00)^2 /$ $\! (000,W(11)) /$ $\! (010,10) /$ $\! (001,10)^2 /$ $\! (000,10) /$ $\! (010,01) /$ $\! (100,01)^2 /$ $\! (000,01) /$ $\! (000,00)^4$ \\

$A_2 A_2$ & $(10,10) /$ $\! (00,10)^3 /$ $\! (01,01) /$ $\! (00,01)^3 /$ $\! (10,00)^3 /$ $\! (01,00)^3 /$ $\! (00,00)^2$ & $(W(11),00) /$ $\! (00,W(11)) /$ $\! (10,10)^3 /$ $\! (01,01)^3 /$ $\! (10,01) /$ $\! (01,10) /$ $\! (10,00)^3 /$ $\! (01,00)^3 /$ $\! (00,10)^3 /$ $\! (00,01)^3 /$ $\! (00,00)^{9} $ \\

\hline

\end{longtable}

\begin{longtable}{p{0.1\textwidth - 2\tabcolsep}>{\raggedright\arraybackslash}p{0.9\textwidth-\tabcolsep}@{}}

\caption{The composition factors for the action of Levi subgroups (with no rank 1 simple factors) of $E_8$ on $L(E_8)$. \label{levie8}} \\

\hline \noalign{\smallskip}

Levi $L'$ & Composition factors of $L(E_8) \downarrow L'$ \\

\hline \noalign{\smallskip}

$E_7$ & $W(\lambda_1) /$ $\! \lambda_7^2 /$ $\! 0^3$ \\

$E_6$ & $W(\lambda_2) /$ $\! \lambda_1^3 /$ $\! \lambda_6^3 /$ $\! 0^8$ \\

$D_7$ & $W(\lambda_2) /$ $\! \lambda_1^2 /$ $\! \lambda_6 /$ $\! \lambda_7 /$ $\! 0$  \\

$D_6$ & $W(\lambda_2) /$ $\! \lambda_1^4 /$ $\! \lambda_5^2 /$ $\! \lambda_6^2 /$ $\! 0^6$ \\

$D_5$ &  $W(\lambda_2) /$ $\! \lambda_1^6 /$ $\! \lambda_4^4 /$ $\! \lambda_5^4 /$ $\! 0^{15}$ \\

$D_4$ & $W(0100) /$ $\! 1000^8 /$ $\! 0010^8 /$ $\! 0001^8 /$ $\! 0000^{28}$ \\

$A_7$ & $W(\lambda_1 + \lambda_7) /$ $\! \lambda_1 /$ $\! \lambda_2 /$ $\! \lambda_3 /$ $\! \lambda_5 /$ $\! \lambda_6 /$ $\! \lambda_7 /$ $\! 0$ \\

$A_6$ & $W(\lambda_1 + \lambda_6) /$ $\! \lambda_1^3 /$ $\! \lambda_2^2 /$ $\! \lambda_3 /$ $\! \lambda_4 /$ $\! \lambda_5^2 /$ $\! \lambda_6^3 /$ $\! 0^4$  \\

$A_5$ & $W(\lambda_1 + \lambda_5) /$ $\! \lambda_1^6 /$ $\! \lambda_2^3 /$ $\! \lambda_3^2 /$ $\! \lambda_4^3 /$ $\! \lambda_5^6 /$ $\! 0^{11}$  \\
 
$A_4$ & $W(1001) /$ $\! 1000^{10} /$ $\! 0100^5 /$ $\! 0010^5 /$ $\! 0001^{10} /$ $\! 0000^{24}$  \\

$A_3$ & $W(101) /$ $\! 100^{16} /$ $\! 001^{16} /$ $\! 010^{10} /$ $\! 000^{45}$ \\

$A_2$ & $W(11) /$ $\! 10^{27} /$ $\! 01^{27} /$ $\! 00^{78}$ \\

$D_5 A_2$ & $(W(\lambda_2),00) /$ $\! (0,W(11)) /$ $\! (\lambda_1,10) /$ $\! (\lambda_1,01) /$ $\! (\lambda_4,01) /$ $\! (\lambda_4,00) /$ $\! (\lambda_5,10) /$ $\! (\lambda_5,00) /$ $\! (0,10) /$ $\! (0,01) /$ $\! (0,00)$  \\

$D_4 A_2$ & $(W(0100),00) /$ $\! (0000,W(11)) /$ $\! (1000,10) /$ $\! (0010,10) /$ $\! (0001,10) /$ $\! (1000,01) /$ $\! (0010,01) /$ $\! (0001,01) /$ $\! (1000,00)^2 /$ $\! (0010,00)^2 /$ $\! (0001,00)^2 /$ $\! (0000,10)^3 /$ $\! (0000,01)^3 /$ $\! (0000,00)^2$  \\

$A_4 A_3$ & $(W(1001),000) /$ $\! (0000,W(101)) /$ $\! (1000,100) /$ $\! (1000,010) /$ $\! (0100,001) /$ $\! (0100,000) /$ $\! (0010,100) /$ $\! (0010,000) /$ $\! (0001,001) /$ $\! (0001,010) /$ $\! (0000,100) /$ $\! (0000,001) /$ $\! (0000,000)$  \\

$A_4 A_2$ & $(W(1001),00) /$ $\! (0000,W(11)) /$ $\! (1000,10)^2 /$ $\! (1000,01) /$ $\! (1000,00) /$ $\! (0100,01) /$ $\! (0100,00)^2 /$ $\! (0010,10) /$ $\! (0010,00)^2 /$ $\! (0001,01)^2 /$ $\! (0001,10) /$ $\! (0001,00) /$ $\! (0000,10)^2 /$ $\! (0000,01)^2 /$ $\! (0000,00)^4$ \\

$A_3 A_3$ & $(W(101),000) /$ $\! (000,W(101)) /$ $\! (100,100) /$ $\! (100,010) /$ $\! (100,001) /$ $\! (100,000)^2 /$ $\! (010,100) /$ $\! (010,001) /$ $\! (010,000)^2  /$ $\! (001,001) /$ $\! (001,010) /$ $\! (001,100) /$ $\! (001,000)^2 /$ $\! (000,010)^2 /$ $\! (000,100)^2 /$ $\! (000,001)^2 /$ $\! (000,000)^{2}$ \\

$A_3 A_2$ & $(W(101),000) /$ $\! (000,W(11)) /$ $\! (100,10)^2 /$ $\! (100,01)^2 /$ $\! (100,00)^4 /$ $\! (010,10) /$ $\! (010,01) /$ $\! (010,00)^4  /$ $\! (001,10)^2 /$ $\! (001,01)^2 /$ $\! (001,00)^4 /$ $\! (000,10)^5 /$ $\! (000,01)^5 /$ $\! (000,00)^{7}$ \\

$A_2 A_2$ & $(W(11),00) /$ $\! (00,W(11)) /$ $\! (10,10)^3 /$ $\! (01,01)^3 /$ $\! (10,01)^3 /$ $\! (01,10)^3 /$ $\! (10,00)^9 /$ $\! (01,00)^9 /$ $\! (00,10)^9 /$ $\! (00,01)^9 /$ $\! (00,00)^{16}$  \\

\hline

\end{longtable}

\setlength\LTleft{0.125\textwidth}
\begin{longtable}{p{0.15\textwidth - 2\tabcolsep}p{0.1\textwidth-2\tabcolsep}>{\raggedright\arraybackslash}p{0.18\textwidth-2\tabcolsep}>{\raggedright\arraybackslash}p{0.32\textwidth-\tabcolsep}@{}}

\caption{The composition factors of some reducible Weyl modules that have a trivial composition factor. \label{weyl}} \\

\hline \noalign{\smallskip}

Group & $p$ &  High weight $\lambda$ & Composition factors of $W(\lambda)$ \\

\hline \noalign{\smallskip}

$A_n$ & $p | n+1$ & $\lambda_1 + \lambda_n$ & $\lambda_1 + \lambda_n / 0$ \\

$B_n$ $(n \geq 2)$ & $2$ & $\lambda_1$ & $\lambda_1 / 0$ \\

$B_n$ $(n>2)$ & $2$ & $\lambda_2$ & $\lambda_2 / \lambda_1 / 0^{(n,2)}$ \\

$D_n$ & $2$ & $\lambda_2$ & $\lambda_2 / 0^{(n,2)}$ \\

$A_2$ & $2$ & $3 \lambda_1$ & $3 \lambda_1 / 0$ \\

& $3$ & $4 \lambda_1 + \lambda_2$ & $4 \lambda_1 + \lambda_2 / 3 \lambda_1 / 3 \lambda_2 / \lambda_1 + \lambda_2 / 0$ \\

$B_2$ & $2$ & $2 \lambda_2$ & $2 \lambda_2 / \lambda_1 / 0^2$ \\

$B_4$ & $2$ & $\lambda_3$ & $\lambda_3 /  \lambda_2 / \lambda_1 / 0^2$ \\

& $3$ & $2 \lambda_1$ & $2 \lambda_1 / 0$ \\

$C_4$ & $2$ & $2 \lambda_1$ & $2 \lambda_1 / \lambda_2 / 0^2$ \\

& $2$ & $\lambda_2$ & $\lambda_2 / 0$ \\

& $3$ & $\lambda_4$ & $\lambda_4 / 0$ \\ 

$G_2$ & $7$ & $2 \lambda_1$ & $2 \lambda_1 / 0$ \\

$F_4$ & $3$ & $\lambda_4$ & $\lambda_4 / 0$ \\

$E_6$ & 3 & $\lambda_2$ & $\lambda_2 / 0$ \\

$E_7$ & 2 & $\lambda_1$ & $\lambda_1 / 0$ \\

\hline

\end{longtable}

\end{appendices}

\bibliographystyle{amsplain}

\bibliography{../biblio}

\end{document}